\documentclass[12pt]{amsart}
\usepackage{latexsym}
\usepackage{indentfirst}
\usepackage{amssymb,amsfonts,amsmath,amsthm}
\usepackage{graphicx}
\usepackage{mathrsfs}
\usepackage{array}
\usepackage{color}
\usepackage{epstopdf}
\usepackage[all]{xy}
\usepackage{verbatim}
\usepackage{url}
\usepackage{tikz-cd}
\usepackage{relsize}

\usepackage[linktocpage=true, hidelinks, colorlinks, linkcolor=blue, citecolor=red, bookmarksopen=true, urlcolor=brown]{hyperref}

\usepackage[a4paper,bindingoffset=0.2in,           left=0.5in,right=0.51in,top=0.6in,bottom=0.5in,            footskip=.25in]{geometry}

\usepackage[bbgreekl]{mathbbol}
\DeclareSymbolFontAlphabet{\mathbb}{AMSb} 
\DeclareSymbolFontAlphabet{\mathbbl}{bbold}

\newtheorem{thm}{Theorem}[section]
\newtheorem{prop}[thm]{Proposition}
\newtheorem{lem}[thm]{Lemma}

\newtheorem{cor}[thm]{Corollary}

\numberwithin{equation}{section}

\theoremstyle{definition}

\newtheorem{construction}[thm]{Construction}

\newtheorem{dfn}[thm]{Definition}
\newtheorem{exam}[thm]{Example}

\newtheorem{rmk}[thm]{Remark}

\begin{document}

\newcommand{\fraka}{{\mathfrak a}}
\newcommand{\frakb}{{\mathfrak b}}
\newcommand{\frakc}{{\mathfrak c}}
\newcommand{\frakd}{{\mathfrak d}}
\newcommand{\frake}{{\mathfrak e}}
\newcommand{\frakf}{{\mathfrak f}}
\newcommand{\frakg}{{\mathfrak g}}
\newcommand{\frakh}{{\mathfrak h}}
\newcommand{\fraki}{{\mathfrak i}}
\newcommand{\frakj}{{\mathfrak j}}
\newcommand{\frakk}{{\mathfrak k}}
\newcommand{\frakl}{{\mathfrak l}}
\newcommand{\frakm}{{\mathfrak m}}
\newcommand{\frakn}{{\mathfrak n}}
\newcommand{\frako}{{\mathfrak o}}
\newcommand{\frakp}{{\mathfrak p}}
\newcommand{\frakq}{{\mathfrak q}}
\newcommand{\frakr}{{\mathfrak r}}
\newcommand{\fraks}{{\mathfrak s}}
\newcommand{\frakt}{{\mathfrak t}}
\newcommand{\fraku}{{\mathfrak u}}
\newcommand{\frakv}{{\mathfrak v}}
\newcommand{\frakw}{{\mathfrak w}}
\newcommand{\frakx}{{\mathfrak x}}
\newcommand{\fraky}{{\mathfrak y}}
\newcommand{\frakz}{{\mathfrak z}}

\newcommand{\frakA}{{\mathfrak A}}
\newcommand{\frakB}{{\mathfrak B}}
\newcommand{\frakC}{{\mathfrak C}}
\newcommand{\frakD}{{\mathfrak D}}
\newcommand{\frakE}{{\mathfrak E}}
\newcommand{\frakF}{{\mathfrak F}}
\newcommand{\frakG}{{\mathfrak G}}
\newcommand{\frakH}{{\mathfrak H}}
\newcommand{\frakI}{{\mathfrak I}}
\newcommand{\frakJ}{{\mathfrak J}}
\newcommand{\frakK}{{\mathfrak K}}
\newcommand{\frakL}{{\mathfrak L}}
\newcommand{\frakM}{{\mathfrak M}}
\newcommand{\frakN}{{\mathfrak N}}
\newcommand{\frakO}{{\mathfrak O}}
\newcommand{\frakP}{{\mathfrak P}}
\newcommand{\frakQ}{{\mathfrak Q}}
\newcommand{\frakR}{{\mathfrak R}}
\newcommand{\frakS}{{\mathfrak S}}
\newcommand{\frakT}{{\mathfrak T}}
\newcommand{\frakU}{{\mathfrak U}}
\newcommand{\frakV}{{\mathfrak V}}
\newcommand{\frakW}{{\mathfrak W}}
\newcommand{\frakX}{{\mathfrak X}}
\newcommand{\frakY}{{\mathfrak Y}}
\newcommand{\frakZ}{{\mathfrak Z}}

\newcommand{\bA}{{\mathbb A}}
\newcommand{\bB}{{\mathbb B}}
\newcommand{\bC}{{\mathbb C}}
\newcommand{\bD}{{\mathbb D}}
\newcommand{\bE}{{\mathbb E}}
\newcommand{\bF}{{\mathbb F}}
\newcommand{\bG}{{\mathbb G}}
\newcommand{\bH}{{\mathbb H}}
\newcommand{\bI}{{\mathbb I}}
\newcommand{\bJ}{{\mathbb J}}
\newcommand{\bK}{{\mathbb K}}
\newcommand{\bL}{{\mathbb L}}
\newcommand{\bM}{{\mathbb M}}
\newcommand{\bN}{{\mathbb N}}
\newcommand{\bO}{{\mathbb O}}
\newcommand{\bP}{{\mathbb P}}
\newcommand{\bQ}{{\mathbb Q}}
\newcommand{\bR}{{\mathbb R}}
\newcommand{\bS}{{\mathbb S}}
\newcommand{\bT}{{\mathbb T}}
\newcommand{\bU}{{\mathbb U}}
\newcommand{\bV}{{\mathbb V}}
\newcommand{\bW}{{\mathbb W}}
\newcommand{\bX}{{\mathbb X}}
\newcommand{\bY}{{\mathbb Y}}
\newcommand{\bZ}{{\mathbb Z}}

\newcommand{\bfA}{{\mathbf A}}
\newcommand{\bfB}{{\mathbf B}}
\newcommand{\bfC}{{\mathbf C}}
\newcommand{\bfD}{{\mathbf D}}
\newcommand{\bfE}{{\mathbf E}}
\newcommand{\bfF}{{\mathbf F}}
\newcommand{\bfG}{{\mathbf G}}
\newcommand{\bfH}{{\mathbf H}}
\newcommand{\bfI}{{\mathbf I}}
\newcommand{\bfJ}{{\mathbf J}}
\newcommand{\bfK}{{\mathbf K}}
\newcommand{\bfL}{{\mathbf L}}
\newcommand{\bfM}{{\mathbf M}}
\newcommand{\bfN}{{\mathbf N}}
\newcommand{\bfO}{{\mathbf O}}
\newcommand{\bfP}{{\mathbf P}}
\newcommand{\bfQ}{{\mathbf Q}}
\newcommand{\bfR}{{\mathbf R}}
\newcommand{\bfS}{{\mathbf S}}
\newcommand{\bfT}{{\mathbf T}}
\newcommand{\bfU}{{\mathbf U}}
\newcommand{\bfV}{{\mathbf V}}
\newcommand{\bfW}{{\mathbf W}}
\newcommand{\bfX}{{\mathbf X}}
\newcommand{\bfY}{{\mathbf Y}}
\newcommand{\bfZ}{{\mathbf Z}}

\newcommand{\calA}{{\mathcal A}}
\newcommand{\calB}{{\mathcal B}}
\newcommand{\calC}{{\mathcal C}}
\newcommand{\calD}{{\mathcal D}}
\newcommand{\calE}{{\mathcal E}}
\newcommand{\calF}{{\mathcal F}}
\newcommand{\calG}{{\mathcal G}}
\newcommand{\calH}{{\mathcal H}}
\newcommand{\calI}{{\mathcal I}}
\newcommand{\calJ}{{\mathcal J}}
\newcommand{\calK}{{\mathcal K}}
\newcommand{\calL}{{\mathcal L}}
\newcommand{\calM}{{\mathcal M}}
\newcommand{\calN}{{\mathcal N}}
\newcommand{\calO}{{\mathcal O}}
\newcommand{\calP}{{\mathcal P}}
\newcommand{\calQ}{{\mathcal Q}}
\newcommand{\calR}{{\mathcal R}}
\newcommand{\calS}{{\mathcal S}}
\newcommand{\calT}{{\mathcal T}}
\newcommand{\calU}{{\mathcal U}}
\newcommand{\calV}{{\mathcal V}}
\newcommand{\calW}{{\mathcal W}}
\newcommand{\calX}{{\mathcal X}}
\newcommand{\calY}{{\mathcal Y}}
\newcommand{\calZ}{{\mathcal Z}}

\newcommand{\rA}{{\mathrm A}}
\newcommand{\rB}{{\mathrm B}}
\newcommand{\rC}{{\mathrm C}}
\newcommand{\rD}{{\mathrm D}}
\newcommand{\rd}{{\mathrm d}}
\newcommand{\rE}{{\mathrm E}}
\newcommand{\rF}{{\mathrm F}}
\newcommand{\rG}{{\mathrm G}}
\newcommand{\rH}{{\mathrm H}}
\newcommand{\rI}{{\mathrm I}}
\newcommand{\rJ}{{\mathrm J}}
\newcommand{\rK}{{\mathrm K}}
\newcommand{\rL}{{\mathrm L}}
\newcommand{\rM}{{\mathrm M}}
\newcommand{\rN}{{\mathrm N}}
\newcommand{\rO}{{\mathrm O}}
\newcommand{\rP}{{\mathrm P}}
\newcommand{\rQ}{{\mathrm Q}}
\newcommand{\rR}{{\mathrm R}}
\newcommand{\rS}{{\mathrm S}}
\newcommand{\rT}{{\mathrm T}}
\newcommand{\rU}{{\mathrm U}}
\newcommand{\rV}{{\mathrm V}}
\newcommand{\rW}{{\mathrm W}}
\newcommand{\rX}{{\mathrm X}}
\newcommand{\rY}{{\mathrm Y}}
\newcommand{\rZ}{{\mathrm Z}}

\newcommand{\Zp}{{\bZ_p}}
\newcommand{\Zl}{{\bZ_l}}
\newcommand{\Qp}{{\bQ_p}}
\newcommand{\Ql}{{\bQ_l}}
\newcommand{\Cp}{{\bC_p}}
\newcommand{\Fp}{{\bF_p}}
\newcommand{\Fl}{{\bF_l}}

\newcommand{\Ainf}{{\mathbf{A}_{\mathrm{inf}}}}
\newcommand{\AdR}{{\mathbf{A}_{\dR}}}
\newcommand{\BdRp}{{\mathbf{B}_{\mathrm{dR}}^+}}
\newcommand{\BdR}{{\mathbf{B}_{\mathrm{dR}}}}

\newcommand{\OXp}{{\widehat \calO_X^+}}
\newcommand{\OX}{{\widehat \calO_X}}
\newcommand{\AAinf}{{\bA_{\mathrm{inf}}}}            
\newcommand{\AAdR}{{\bA_{\dR}}}
\newcommand{\BBinf}{{\bB_{\mathrm{inf}}}}            
\newcommand{\BBbd}{{\bB^{\mathrm{bd}}}}              
\newcommand{\BBdRp}{{\bB_{\mathrm{dR}}^+}}           
\newcommand{\BBdRpn}{{\bB_{\mathrm{dR},n}^+}}     
\newcommand{\BBdR}{{\bB_{\mathrm{dR}}}}              
\newcommand{\OAinf}{{\calO\bA_{\mathrm{inf}}}}         
\newcommand{\OBinf}{{\calO\bB_{\mathrm{inf}}}}         
\newcommand{\OBdRp}{{\calO\bB_{\mathrm{dR}}^+}}        
\newcommand{\OBdR}{{\calO\bB_{\mathrm{dR}}}}           
\newcommand{\OC}{{\calO\bC}}                           

\newcommand{\Aut}{{\mathrm{Aut}}}
\newcommand{\Bun}{{\mathrm{Bun}}}           
\newcommand{\Cech}{{\check{H}}}             
\newcommand{\Coker}{{\mathrm{Coker}}}       
\newcommand{\colim}{{\mathrm{colim}}}       
\newcommand{\dlog}{{\mathrm{dlog}}}         
\newcommand{\End}{{\mathrm{End}}}           
\newcommand{\Exp}{{\mathrm{Exp}}}           
\newcommand{\Ext}{{\mathrm{Ext}}}           
\newcommand{\Fil}{{\mathrm{Fil}}}           
\newcommand{\Fitt}{{\mathrm{Fitt}}}         
\newcommand{\Frac}{{\mathrm{Frac}}}         
\newcommand{\Frob}{{\mathrm{Frob}}}         
\newcommand{\Gal}{{\mathrm{Gal}}}           
\newcommand{\Gr}{{\mathrm{Gr}}}             
\newcommand{\Hom}{{\mathrm{Hom}}}           
\newcommand{\HIG}{{\mathrm{HIG}}}           
\newcommand{\id}{{\mathrm{id}}}             
\newcommand{\Ima}{{\mathrm{Im}}}            
\newcommand{\Isom}{{\mathrm{Isom}}}         
\newcommand{\Ker}{{\mathrm{Ker}}}           
\newcommand{\Lie}{{\mathrm{Lie}}}           
\newcommand{\Log}{{\mathrm{Log}}}           
\newcommand{\Mat}{{\mathrm{Mat}}}           
\newcommand{\MIC}{{\mathrm{MIC}}}           
\newcommand{\Mod}{{\mathrm{Mod}}}           
\newcommand{\Perfd}{{\mathrm{Perfd}}}         
\newcommand{\pr}{{\mathrm{pr}}}             
\newcommand{\Proj}{\mathrm{Proj}}        
\newcommand{\Rep}{{\mathrm{Rep}}}           
\newcommand{\Res}{{\mathrm{Res}}}           
\newcommand{\RGamma}{{\mathrm{\rR\Gamma}}}    
\newcommand{\rk}{{\mathrm{rk}}}             
\newcommand{\Rlim}{{\mathrm{R}\underleftarrow{\lim}}} 
\newcommand{\sgn}{{\mathrm{sgn}}}           
\newcommand{\Sh}{{\mathrm{Shv}}}             
\newcommand{\Sht}{{\mathrm{Sht}}}           
\newcommand{\Spa}{{\mathrm{Spa}}}           
\newcommand{\Spf}{{\mathrm{Spf}}}           
\newcommand{\Spec}{{\mathrm{Spec}}}         
\newcommand{\Strat}{{\mathrm{Strat}}}       
\newcommand{\Sym}{{\mathrm{Sym}}}           
\newcommand{\Tor}{{\mathrm{{Tor}}}}         
\newcommand{\Tot}{{\mathrm{Tot}}}           
\newcommand{\Vect}{{\mathrm{Vect}}}         


\newcommand{\GL}{{\mathrm{GL}}}             
\newcommand{\SL}{{\mathrm{SL}}}             


\newcommand{\aff}{{\mathrm{aff}}}          
\newcommand{\an}{{\mathrm{an}}}             
\newcommand{\can}{{\mathrm{can}}}           
\newcommand{\cl}{{\mathrm{cl}}}             
\newcommand{\cofib}{{\mathrm{cofib}}}        
\newcommand{\cts}{{\mathrm{cts}}}           
\newcommand{\cris}{{\mathrm{cris}}}         
\newcommand{\crys}{{\mathrm{crys}}}         
\newcommand{\cyc}{{\mathrm{cyc}}}           
\newcommand{\dR}{{\mathrm{dR}}}             
\newcommand{\et}{{\mathrm{\acute{e}t}}}    
\newcommand{\fib}{{\mathrm{fib}}}          
\newcommand{\fl}{{\mathrm{fl}}}             
\newcommand{\fppf}{{\mathrm{fppf}}}         
\newcommand{\ket}{{\mathrm{k\acute{e}t}}}    
\newcommand{\geo}{{\mathrm{geo}}}           
\newcommand{\gp}{{\mathrm{gp}}}             
\newcommand{\la}{{\mathrm{la}}}             
\newcommand{\nil}{{\mathrm{nil}}}           
\newcommand{\pd}{{\mathrm{pd}}}             
\newcommand{\perf}{\mathrm{perf}}           
\newcommand{\proet}{{\mathrm{pro\acute{e}t}}}
\newcommand{\proket}{{prok\mathrm{\acute{e}t}}}     
\newcommand{\rig}{{\mathrm{rig}}}            
\newcommand{\sm}{{\mathrm{sm}}}             
\newcommand{\st}{{\mathrm{st}}}             
\newcommand{\tor}{{\mathrm{tor}}}           
\newcommand{\Zar}{{\mathrm{Zar}}}           



\newcommand{\Prism}{{\mathlarger{\mathbbl{\Delta}}}} 
\newcommand{\OPrism}{{\calO_{\Prism}}}
\newcommand{\IPrism}{{\calI_{\Prism}}}

\newcommand{\ya}{{\rangle}}
\newcommand{\za}{{\langle}}

\newcommand{\smat}[1]{\left( \begin{smallmatrix} #1 \end{smallmatrix} \right)}



\newcommand{\Ctilde}{\widetilde{\bfC}}
\newcommand{\CCtilde}{\widetilde{\bC}}
\newcommand{\Cpsi}{\widetilde{\mathbb{C}}_\psi^I}
\newcommand{\CCpsi}{\widetilde{\mathbb{C}}_{\tilde{\psi}}^I}
\newcommand{\Cppsi}{\widetilde{\mathbb{C}}_\psi^{I, +}}
\newcommand{\CCppsi}{\widetilde{\mathbb{C}}_{\tilde{\psi}}^{I,+}}
\newcommand{\Bnpsi}{\widetilde{\bfB}_{\psi,n}}
\newcommand{\Bpsi}{\widetilde{\mathbf{B}}_{\psi}}
\newcommand{\Bnppsi}{\widetilde{\mathbf{B}}_{\psi,\infty,n}}
\newcommand{\Bppsi}{\widetilde{\mathbf{B}}_{\psi,\infty}}
\newcommand{\norm}{{|\!|}}

\newcommand{\wtA}{\widetilde{A}}
\newcommand{\wtJ}{\widetilde{J}}
\newcommand{\wtS}{\widetilde{S}}
\newcommand{\wtR}{\widetilde{R}}
\newcommand{\wtX}{\widetilde{X}}
\newcommand{\wtY}{\widetilde{Y}}
\newcommand{\wtZ}{\widetilde{Z}}

\newcommand{\wtr}{\widetilde{\calR}}
\newcommand{\wtx}{\widetilde{\frakX}}
\newcommand{\wty}{\widetilde{\frakY}}
\newcommand{\wtz}{\widetilde{\frakZ}}

\newcommand{\Hsmall}{{\rm H}\text{-sm}}

\title{A stacky $p$-adic Riemann--Hilbert correspondence on Hitchin-small locus}

\author[]{Yudong Liu}
\address{Y.L., School of Mathematical Science, Peking University, YiHeYuan Road 5, Beijing, 100190, China.}
\email{2100010764@stu.pku.edu.cn}

\author[]{Chenglong Ma}
\address{C.M., Department of Mathematical Sciences, University of Copenhagen, Universitetsparken 5, 2100 København, Denmark}
\email{chenglong.ma@math.ku.dk}

\author[]{Xiecheng Nie}
\address{X.N., Academy of Mathematics and Systems Science, University of Chinese Academy of Science, Zhongguancun East Road 80, Beijing, 100190, China}
\email{niexiecheng23@mails.ucas.ac.cn}

\author[]{Xiaoyu Qu}
\address{X.Q., Department of Mathematics, Princeton University, Fine Hall, Washington Road, Princeton NJ 08544-1000 USA}
\email{xq9867@princeton.edu}

\author[]{Yupeng Wang}
\address{Y.W., Shanghai Center for Mathematical Sciences, Fudan University, 2005 Songhu Road, Shanghai, 200438, China.}
\email{wangyupeng@fudan.edu.cn}




\date{}

\maketitle
\setcounter{tocdepth}{1}

\begin{abstract}
  Let $C$ be an algebraically closed perfectoid field over $\Qp$ with ring of integer $\calO_C$ and infinitesimal thickening $\Ainf$. Let $\frakX$ be a semi-stable formal scheme over $\calO_C$ with a fixed flat lifting $\wtx$ over $\Ainf$. Let $X$ be the generic fiber of $\frakX$ and $\wtX$ be its lifting over $\BdRp$ induced by $\wtx$. Let $\MIC_r(\wtX)^{\Hsmall}$ and $\rL\rS_r(X,\BBdRp)^{\Hsmall}$ be the $v$-stacks of rank-$r$ Hitchin-small integrable connections on $\wtX$ and $\BBdRp$-local systems on $X_{v}$, respectively. In this paper, we establish an equivalence between these two stacks by introducing a new period sheaf with connection $(\calO\bB_{\dR,\pd}^+,\rd)$ on $X_{v}$.
\end{abstract}

\tableofcontents

\section{Introduction}\label{sec:Introduction}
\subsection{Overview}\label{ssec:overview}
  For a smooth projective variety $X$ over the field $\bC$ of complex numbers, the Riemann--Hilbert correspondence describes an equivalence between the category of $\bC$-local systems and the category of integrable connections on $X$. This equivalence upgrades to the moduli level; that is, there is a homeomorphism between the moduli space $\bfM_B$ of $\bC$-local systems and the moduli space $\bfM_{\dR}$ of integrable connections. 
  
  The $p$-adic Riemann--Hilbert correspondence aims to give an analogue of the above correspondence for rigid spaces; that is, it suggests an equivalence between the category of certain local systems and the category of certain integrable connections on a rigid space over a complete $p$-adic field. The first step towards this direction is due to Scholze \cite{Sch-Pi}. In \emph{loc.cit.}, for a smooth rigid variety $X$ over a complete discrete valuation field $K$ of mixed characteristic $(0,p)$ with perfect residue field, he introduced so-call \emph{$\BBdRp$-local systems} on the pro-\'etale site $X_{\proet}$, constructed a period sheaf with connection $(\calO\BBdRp,\rd)$ on $X_{\proet}$ whose de Rham complex gives a resolution of $\BBdRp$, and established an equivalence between the category of \emph{de Rham} $\BBdRp$-local systems and the category of \emph{filtered integrable connections} on $X_{\et}$. Based on his work, Liu and Zhu constructed two functors $\calR\calH$ and $\rD_{\dR}$ from the category of $\Qp$-local systems on $X_{\et}$ to the category of $G_K$-equivariant filtered integrable connections on ``$(X\widehat \otimes_K\BdRp)_{\et}$'' and the category of filtered integrable connections on $X_{\et}$, respectively. Using this, they proved the rigidity for ``being de Rham'' of $\Qp$-local systems \cite{LZ}. Their approach also works in the logarithmic case \cite{DLLZ}. By taking grades, the functor $\calR\calH$ induces a functor $\calH$ from the category of $\Qp$-local systems on $X_{\et}$ to the category of $G_K$-equivariant Higgs bundles on $X_{\widehat{\overline K},\et}$. More generally, it was shown that $\calH$ upgrades to an equivalence from the category of generalised representations on $X_{\proet}$ to the category of $G_K$-equivariant Higgs bundles on $X_{\widehat{\overline K},\et}$ (cf. \cite{MW-JEMS}). The similar phenomenon also occurs when studying Riemann--Hilbert correspondence, in \cite{GMWrelative}, one can establish an equivalence between the category of $\BBdRp$-local systems on $X_{\proet}$ and the category of $G_K$-equivariant integrable connections on ``$X\widehat \otimes_K\BdRp$'', which can be used to classify $\BBdRp$-local systems with prismatic source (see \cite{MW-JEMS}, \cite{AHLB23} for relevant results in ``mod $t$'' case). On the other hand, in the lecture note of his ICM talk, Bhatt described his joint work with Lurie \cite[Th. 5.4]{Bha-ICM}, a Riemann--Hilbert functor from the category of certain $\Qp$-sheaves to the category of certain $D$-modules on $X$. Parts of their results are also obtained by Li \cite{Li} independently, by using a certain variant of $\calO\BBdRp$ and applying the approach of Liu--Zhu \cite{LZ}. 

  It is worth highlighting that all results above require that $X$ is defined over a complete discrete valuation field $K$ containing $\Qp$ with perfect residue field. The advantage in this case is that all $\BBdRp$-local systems and integrable connections are contained in the fibers of the corresponding moduli stacks at the origin of the Hitchin base (cf. Remark \ref{rmk:nilpotent case} below). 

  Now, let $C$ be an algebraically closed perfectoid field containing $\Qp$ with ring of integers $\calO_C$ and infinitesimal thickening $\Ainf$ and let $\xi$ be a generator of the natural surjection $\Ainf\to \calO_C$. One may ask if there is a $p$-adic Simpson correspondence or a $p$-adic Riemann--Hilbert correspondence for a smooth rigid space $X$ defined over $C$ (rather than $K$ as above). Toward this direction, all existing results concern about $p$-adic Simpson correspondence, and the first is due to Faltings. When $X$ admits a nice formal model $\frakX$ over $\calO_C$, Faltings established an equivalence between the category of \emph{(Faltings-)small} generalized representations on $X$ and the category of \emph{(Faltings-)small} Higgs bundles on $\frakX$ (cf. \cite{Fal, AGT,Wan23} and etc.). There are two kinds of generalization of Faltings' result. The first is to generalize his equivalence to the whole categories on both sides. This was done by Heuer \cite{Heu25} very recently. When $X$ is furthermore proper, Heuer established an equivalence between the category of generalized representations and the category of Higgs bundles on $X$, generalizing the result of Faltings for curves. The other is to generalize Faltings' result to the moduli level. This was done by Ansch\"utz--Heuer--Le Bras \cite{AHLB23} and Sheng--Wang \cite{SW24} when $X$ respectively admits a smooth and semi-stable formal model $\frakX$ over $\calO_C$. More precisely, in this case, it was proved that there exists an equivalence between the moduli stack of \emph{Hitchin-small} generalized representations and the moduli stack of \emph{Hitchin-small} Higgs bundles on $X$. We point out that a (Faltings-)small generalized representation (resp. Higgs bundle) is automatically Hitchin-small and this equivalence of moduli stacks is compatible with the equivalence established by Faltings. According to these progress on $p$-adic non-abelian Hodge theory, one may ask the following question:

  Is there a Riemann--Hilbert correspondence for a smooth rigid space $X$ over $C$ (rather than $K$)?

  Our paper aims to give the first answer to this question, and let us give a brief introduction here. Let $C$ be the $p$-adic completion of a complete discrete valuation field of mixed characteristic $(0,p)$ such that its residue field is perfect. Let $\calO_C$ be the ring of integers of $C$ and $\Ainf$ be its infinitesimal thickening. Let $\frakX$ be a semi-stable formal scheme over $\calO_C$ in the sense of \cite{CK19} with a fixed flat lifting $\wtx$ over $\Ainf$. We shall establish an equivalence between the category of \emph{Hitchin-small} $\BBdRp$-local systems on $X_v$, the $v$-site of $X$ introduced in \cite{Sch-Diamond}, and the category of \emph{Hitchin-small} integrable connections on $\wtX$, the lifting of $X$ over $\BdRp$ induced from $\wtx$. Indeed, we shall establish this equivalence at the moduli level; that is, we shall give an equivalence between the moduli stack of Hitchin-small $\BBdRp$-local systems and the moduli stack of Hitchin-small integrable connections on $X$.
  This result generalizes previous work in \cite{AHLB23} and \cite{SW24} on the $p$-adic Simpson correspondence.

\subsection{Main results}\label{ssec:main result}
  Now let us state our main theorem. From now on, we freely use the notation in \S\ref{ssec:notation}. In particular, we always let $K$ denote a complete discrete valuation field with ring of integers $\calO_K$ of mixed characteristic $(0,p)$ and perfect residue field $\kappa$.
  Let $C$ be the $p$-adic completion of a fixed algebraic closure of $K$ with ring of integers $\calO_C$, maximal ideal $\frakm_C$ of $\calO_C$ and residue field $\overline \kappa$.
  Denote by $\bfA_{\inf,K} = \rW(\calO_C^{\flat})\otimes_{\rW(\kappa)}\calO_K$ the ramified infinitesimal thickening of $\calO_C$ with respect to $K$. 
  We also assume that $\frakX$ is a semi-stable formal scheme over $\calO_C$ with a fixed flat lifting $\wtx$ over $\bfA_{\inf,K}$. Let $X$ be the generic fiber of $\frakX$ and $\wtX$ be its lifting over $\BdRp$ induced from $\wtx$. For any $n\geq 1$, let $\wtX_n$ be the reduction of $\wtX$ modulo $t^n$.
  Let $\Perfd$ be the $v$-site of all affinoid perfectoid spaces over $\Spa(C,\calO_C)$. 
  \begin{thm}\label{thm:Stacky RH}
      For any $n\geq 1$, the flat lifting $\wtx$ of $\frakX$ induces an equivalence
      \[\rho_{\wtx}:\rL\rS_r(X,\BBdRpn)^{{\Hsmall}}\simeq\MIC_r(\wtX_n)^{\Hsmall}\]
      between the moduli stack $\rL\rS_r(X,\BBdRpn)^{{\Hsmall}}$ of Hitchin-small $\BBdRpn$-local systems of rank $r$ on $X_{v}$ and the moduli stack $\MIC_r(\wtX_n)^{\Hsmall}$ of Hitchin-small integrable connections of rank $r$ on $\wtX_n$. 
  \end{thm}
  \begin{rmk}\label{rmk:n=1 case for main theorem}
      The $n=1$ case for the above theorem was already obtained by Ansch\"utz, Heuer, and Le Bras \cite{AHLB23} when $\frakX$ is smooth, and by Sheng and the last author \cite{SW24} when $\frakX$ is semi-stable. When $n\geq 2$, to the best of our knowledge, Theorem \ref{thm:Stacky RH} is the \emph{only} result on the stacky $p$-adic Riemann--Hilbert correspondence so far.
  \end{rmk}
  
  Now, we explain some notions appearing in the above Theorem \ref{thm:Stacky RH}.
  Let $Z = \Spa(A,A^+)\in \Perfd$ be an affinoid perfectoid space. Let $\frakX_Z$ be a semi-stable formal scheme over $\Spf(A^+)$ (cf. Definition \ref{dfn:small affine}) with a flat lifting $\wtx_Z$ over $\bA_{\inf,K}(Z):=\bA_{\inf}(Z)\otimes_{\rW(\kappa)}\calO_K$. Let $X_Z$ be the generic fiber of $\frakX_Z$ and $\wtX_Z$ be its lifting over $\BBdRp(Z)$ associated to $\wtx_Z$. For any $n\geq 1$, let $\wtX_{Z,n}$ be the reduction of $\wtX_Z$ modulo $t^n$ and $\BBdRpn:=\BBdRp/t^n$. 
  
  By a \emph{$\BBdRpn$-local system} of rank $r$ on $X_{Z,v}$, the $v$-site of $X_Z$ in the sense of \cite{Sch-Diamond}, we mean a locally finite free $\BBdRpn$-module of rank $r$. When $n=1$, the sheaf $\bB_{\dR,1}^+$ is exactly the structure sheaf $\widehat \calO_{X_Z}$ and $\bB_{\dR,1}^+$-local systems coincide with $v$-vector bundles on $X_{Z,v}$. One can similarly define $\BBdRp$-local systems of rank $r$ on $X_{Z,v}$.
  
  Let $\widetilde \rd:\wtx_Z\to\Omega^{1,\log}_{\wtx_Z}$ be the usual $(p,\xi_K)$-complete log-derivation on $\wtx_Z$ and let $\rd$ be the following composite
  \[\rd: \calO_{\wtx_Z}\xrightarrow{\widetilde \rd}\Omega^{1,\log}_{\wtx_Z}\hookrightarrow\Omega^{1,\log}_{\wtx_Z}\{-1\}.\]
  Here $\Omega^{1,\log}_{\wtx_Z}\{-1\}$ denotes the Breuil--Kisin--Fargues twist of $\Omega^{1,\log}_{\wtx_Z}$ (cf. \S\ref{ssec:notation}). By an \emph{integrable connection} of rank $r$ on $\wtX_{Z,n}$, we mean a pair $(\calM,\nabla)$ consisting of a locally finite free $\calO_{\wtX_{Z,n}}$-module $\calM$ together with a $\BBdRpn(Z)$-linear map
  \[\nabla:\calM\to\calM\otimes_{\calO_{\wtx_Z}}\Omega^{1,\log}_{\wtx_Z}\{-1\}\]
  satisfying the Leibniz rule with respect to $\rd$ such that $\nabla\wedge\nabla = 0$. One can similarly define integrable connections of rank $r$ on $\wtX_Z$. 
  We remark that if we trivialize the Tate twist $\Omega^1_{\wtX_Z}(-1)$ by $\Omega^1_{\wtX_Z}\cdot t^{-1}$, then via the identification
  \[\calM\otimes_{\calO_{\wtx_Z}}\Omega^{1,\log}_{\wtx_Z}\{-1\} = \calM\otimes_{\calO_{\wtx_Z}}\Omega^{1,\log}_{\wtx_Z}(-1),\]
  \'etale locally on $\wtX_{Z,n}$, $\nabla$ behaves like a \emph{$t$-connection}: Suppose that $T_1,\dots,T_d$ are local coordinates on $\widetilde X_{Z,\et}$, and then via the trivialization
  \[\Omega^1_{\widetilde X_Z}(-1) = \oplus_{i=1}^d\calO_{\widetilde X_Z}\cdot \frac{\rd T_i}{t},\]
  the trivial connection 
  \[\rd:\calO_{\widetilde X_Z}\to \calO_{\widetilde X_Z}\otimes_{\calO_{\widetilde \frakX_Z}}\Omega^{1,\log}_{\widetilde \frakX_Z}\{-1\} = \Omega^1_{\widetilde X_Z}(-1)\]
  is given by $\rd = \sum_{i=1}^dt\frac{\partial}{\partial T_i}\otimes\frac{\rd T_i}{t}$. So $\nabla$ behaves like a $t$-connection because it satisfies the Leibniz rule with respect to $\rd$.
  In particular, when $n=1$, an integrable connection on $\wtX_{Z,1} = X_Z$ is exactly a Higgs bundle. 

  For a semi-stable $\frakX$ over $\calO_C$ with a fixed flat lifting $\wtx$ over $\bfA_{\inf,K}$, for any $Z=\Spa(A,A^+)\in \Perfd$, let $\frakX_Z$, $\wtx_Z$, $X_Z$ and $\wtX_Z$ be the base-changes of $\frakX$, $\wtx$, $X$ and $\wtX$ to $A^+$, $\bA_{\inf,K}(Z)$, $A$ and $\BBdRp(Z)$, respectively.
  Consider the following two functors:
  \[\rL\rS_r(X,\BBdRpn):\Perfd\to\text{Groupoid},\quad Z\mapsto\{\text{$\BBdRpn$-local systems on $X_{Z,v}$ of rank $r$}\}\]
  and 
  \[\MIC_r(\wtX_{Z,n}):\Perfd\to\text{Groupoid},\quad Z\mapsto\{\text{integrable connections on $\wtX_{Z,n}$ of rank $r$}\}.\]
  \begin{rmk}\label{rmk:stacks}
      For $n=1$, the above functors were first considered by Heuer \cite{Heu-Moduli} without assuming that $X$ has a nice integral model $\frakX$ over $\calO_C$. In \emph{loc.cit.}, he introduced the functor $v\Bun_r(X)$ (resp. $\HIG_r(X)$) of $v$-vector bundles (resp. Higgs bundles) of rank r on $X_{v}$ (resp. $X_{\et}$), which is exactly $\rL\rS_r(X,\bB_{\dR,1}^+)$ (resp. $\MIC_r(\wtX_1)$) defined above. He proved that these two functors are small $v$-stacks on $\Perfd$. For general $n$, the above functors are introduced by Yu \cite{Yu24}. He proved that they are also small $v$-stacks.
  \end{rmk}
  Clearly, we have the following diagram:
  \begin{equation}\label{diag:reduction}
      \begin{tikzcd}
          \vdots \arrow[d] & & \vdots \arrow[d] \\
          {\rL\rS(X,\bB_{\dR,n+1}^+)} \arrow[d] & & \MIC_r(\wtX_{n+1}) \arrow[d] \\
          {\rL\rS(X,\bB_{\dR,n}^+)} \arrow[d] & & \MIC_r(\wtX_n) \arrow[d] \\
          \vdots \arrow[d] & & \vdots \arrow[d]  \\
          {v\Bun_r(X)=\rL\rS(X,\bB_{\dR,1}^+)} \qquad\qquad\qquad\arrow[rd,"\widetilde h"'] & & \qquad\qquad\qquad\MIC_r(\wtX_1)=\HIG_r(X) \arrow[ld,"h"] \\
          & \calA_r &                                     
      \end{tikzcd}
  \end{equation}
  where all vertical maps are induced by taking the obvious reduction, $h$ and $\widetilde h$ denote the Hitchin fibrations introduced in \cite{Heu-Moduli}, and $\calA_r$ is the Hitchin base, which is defined as a functor
  \[\calA_r:\Perfd\to \text{Sets},\quad Z=\Spa(A,A^+)\mapsto \oplus_{i=1}^{r}\rH^0(X_Z,\Sym^i(\Omega^1_{X_Z/A}\{-1\})).\]
  One can define the \emph{Hitchin-small locus} $\calA_r^{\Hsmall}\subset \calA_r$ as the following sub-functor
  \[\calA_r^{\Hsmall}:\Perfd\to \text{Sets},\quad Z = \Spa(A,A^+)\mapsto\oplus_{i=1}^{r}p^{>\frac{i}{p-1}}\rH^0(\frakX_Z,\Sym^i(\Omega^{1,\log}_{\frakX_Z/A^+}\{-1\}))\]
  as in \cite{AHLB23}, where $p^{>\frac{i}{p-1}}:=(\zeta_p-1)^i\frakm_C\subset \calO_C$. Both functors $\calA_r$ and $\calA_r^{\Hsmall}$ make sense as $v$-sheaves. For any stack $\Sigma$ lying over $\calA_r$, define its Hitchin-small locus $\Sigma^{\Hsmall}$ as the sub-stack
  \[\Sigma^{\Hsmall}:=\Sigma\times_{\calA_r}\calA_r^{\Hsmall}.\]
  For example, we have $\rL\rS_r(X,\BBdRpn)^{\Hsmall}$ and $\MIC_r(\wtX_n)^{\Hsmall}$.
  Now, we have explained all the notions involved in Theorem \ref{thm:Stacky RH}.

  For $\BBdRpn$-local systems and integrable connections over $\wtX_n$, one can check the Hitchin-smallness by taking reduction modulo $t$:
  \begin{rmk}\label{rmk:Hitchin small can be checked modulo t}
      A $\BBdRpn$-local system (resp. integrable connection on $\wtX_n$) is Hitchin-small \emph{if and only if} its reduction modulo $t$ is a Hitchin-small $v$-bundle on $X_{v}$ (resp. Higgs bundle on $X_{\et}$).
  \end{rmk}

  We list some immediate corollaries of Theorem \ref{thm:Stacky RH}. Again, let $\frakX$ be a semi-stable formal scheme over $\calO_C$ which admits a flat lifting over $\bfA_{\inf,K}$ and fix such a lifting $\wtx$. First, letting $n$ go to $\infty$, we obtain the following equivalence of stacks:
  \begin{cor}\label{cor:main result}
      Keep notations as above. The lifting $\wtx$ of $\frakX$ induces an equivalence
      \[\rho_{\wtx}:\rL\rS_r(X,\BBdRp)^{\Hsmall}\simeq \MIC_r(\wtX)^{\Hsmall}\]
      between the moduli stack $\rL\rS_r(X,\BBdRp)^{\Hsmall}$ of Hitchin-small $\BBdRp$-local systems of rank $r$ and the moduli stack $\MIC_r(\wtX)^{\Hsmall}$ of Hitchin-small integrable connections of rank $r$.
  \end{cor}
  To simplify the notations, set $\BBdRp:=\bB_{\dR,\infty}^+$ and $\wtX_{\infty}:=\wtX$. The next corollary follows from taking $C$-points in Theorem \ref{thm:Stacky RH} and Corollary \ref{cor:main result} immediately.
  \begin{cor}\label{cor:Hitchin-small RH}
      Keep notations as above.
      For any $1\leq n\leq \infty$, the lifting $\wtx$ induces an equivalence 
      \[\rL\rS(X,\BBdRpn)^{\Hsmall}(C)\simeq\MIC(\wtX_n)^{\Hsmall}(C)\]
      between the category $\rL\rS(X,\BBdRpn)^{\Hsmall}(C)$ of Hitchin-small $\BBdRpn$-local systems on $X_{v}$ and the category $\MIC(\wtX_n)^{\Hsmall}(C)$ of Hitchin-small integrable connections on $\wtX_n$.
  \end{cor}
  
  \begin{cor}\label{cor:whole correspondence for P}
      Let $\bP^d$ be the projective space over $C$ of dimension $d$.
      For any $1\leq n\leq \infty$, there exists an equivalence of small $v$-stacks
      \[\rL\rS_r(\bP^d,\BBdRpn)\simeq \MIC_r(\widetilde \bP_n^d).\]
      Here $\widetilde \bP_n^d$ is defined as follows: Let $\widetilde{\frakP}^d$ be the projective space of relative dimension $d$ over $\bfA_{\inf}$ (whose reduction $\frakP^d$ modulo $\xi$ is obviously a smooth formal model of $\bP^d$ over $\calO_C$) and let $\widetilde \bP_n^d$ be the lifting of $\bP^d$ associated to $\widetilde{\frakP}^d$.
  \end{cor}
  \begin{proof}
      This holds true as $\BBdRpn$-local systems (resp. intergable connections on $\widetilde \bP_n^d$) are automatically Hitchin-small (and even belong to the fiber of corresponding stacks at the origin $0\in \calA_r$). See \cite[Cor. 3.26]{AHLB23} for details.
  \end{proof}
  \begin{rmk}\label{rmk:nilpotent case}
      In Corollary \ref{cor:whole correspondence for P}, it is not necessary to work with the smooth integral model of $\bP^n$. More generally, for any smooth rigid variety $X$ over $C$ with a smooth lifting $\wtX$ over $\BdRp$, using a period sheaf with connection $(\calO\bB_{\dR},\rd)$ described in \cite[Def. 2.36]{Yu24} and the same argument in this paper, we can give an equivalence of stacks
      \[\rL\rS_r(X,\BBdRpn)^{0}\simeq\MIC_r(\wtX_n)^{0}\]
      where $\Sigma^0$ denotes the fiber at the origin $0\in \calA_r$ for any stack $\Sigma$ over $\calA_r$. See \cite{LMNQ} for more details. When $X$ is the base-change along $K\to C$ for some smooth rigid space $X_0$ defined over a complete discrete valuation sub-field $K\subset C$, the base-change $\wtX$ of $X_0$ along $K\to\BdRp$ is a lifting of $X$. In this case, $\BBdRp$-local systems on $X_{0,v}$ and $G_K$-equivariant integrable connections on $\wtX$ always belong to the fiber of the corresponding stacks at the origin $0\in\calA_r$ (cf. \cite[Rem. 3.2]{MW-JEMS}).
  \end{rmk}
  \begin{rmk}\label{rmk:whole stack}
      For general $X$, it is still a question if there exists an equivalence of the whole stacks
      \begin{equation}\label{equ:whole stack equivalence}
          \rL\rS_r(X,\BBdRpn)\simeq\MIC_r(\wtX_n).
      \end{equation}
      Following \cite[\S 6]{Heu-Sigma}, in the rest of this remark we always assume that $X$ is \emph{proper} smooth.
      Note that for $n=1$; that is, in the case for $v$-vector bundles and Higgs bundles, Heuer and Xu proved that if $X$ is a smooth curve, such an equivalence (\ref{equ:whole stack equivalence}) exists up to a certain normalization \cite{HX}. Beyond curves, we only get an equivalence between the category of $v$-vector bundles on $X_v$ and the category Higgs bundles on $X_{\et}$ \cite{Heu25}. It seems that the problem (at least in the $n=1$ case) could be solved using the Simpson gerbe of Bhargav Bhatt and Mingjia Zhang. Based on their work, Heuer--Petrov--Vologodsky \cite{HPV} announced very recently that they can construct a fully faithful functor 
      \[\rL\rS_r(X,\BBdRpn)(C)\to\MIC_r(\wtX_n)(C)\]
      which becomes an equivalence when $r = 1$.
  \end{rmk}

\subsection{Strategy for the proof and organization of this paper}\label{ssec:strategy}
  Here, we sketch the idea for the proof of Theorem \ref{thm:Stacky RH} and explain how the paper is organized. Again, we freely use the notation in \S\ref{ssec:notation} below.

  Fix a semi-stable scheme $\frakX$ over $\calO_C$ of relative dimension $d$. Assume that it admits a flat lifting and fix such a lifting $\wtx$ over $\bfA_{\inf,K}$. For any $Z=\Spa(A,A^+)\in \Perfd$, let $\frakX_Z$, $\wtx_Z$, $X_Z$ and $\wtX_Z$ be the base-changes of $\frakX$, $\wtx$, $X$ and $\wtX$ to $A^+$, $\bA_{\inf,K}(Z)$, $A$ and $\BBdRp(Z)$, respectively.
  To prove Theorem \ref{thm:Stacky RH}, one has to construct an equivalence
  \[\rho_{\wtx_Z}:\rL\rS_r(X,\bB_{\dR,n}^+)^{\Hsmall}(Z)\xrightarrow{\simeq}\MIC_r(\wtX_n)^{\Hsmall}(Z)\]
  between the category $\rL\rS_r(X,\bB_{\dR,n}^+)^{\Hsmall}(Z)$ of Hitchin-small $\BBdRpn$-local systems on $X_{Z,v}$ and the category $\MIC_r(\widetilde X_n)^{\Hsmall}(Z)$ of Hitchin-small integrable connections on $\wtX_{Z,n}$,
  which is functorial in $Z$. To do so, the key point is that using the lifting $\wtx$, one can construct a period sheaf with connection $(\calO\bB_{\dR,\pd,Z}^+,\rd)$ on $X_{Z,v}$, which is functorial in $Z$, such that the following Riemann--Hilbert correspondence holds.
  \begin{thm}[Theorem \ref{thm:global RH}]\label{thm:global RH-intro}
      Fix a $Z=\Spa(A,A^+)\in \Perfd$.
      Let $\nu:X_{Z,v}\to X_{Z,\et}$ be the natural morphism of sites. Let $1\leq n\leq \infty$. 
      \begin{enumerate}
          \item For any $\bL\in \rL\rS(X,\bB_{\dR,n}^+)^{\Hsmall}(Z)$ of rank $r$, we have 
          \[\rR^i\nu_*(\bL\otimes_{\BBdRp}\calO\bB_{\dR,\pd,Z}^+) = \left\{
          \begin{array}{rcl}
              \calD(\bL), & i=0 \\
              0, & i\geq 1
          \end{array}
          \right.\]
          where $\calD(\bL)$ is a locally finite free $\calO_{\wtX_{Z,n}}$-module of rank $r$ on $X_{Z,\et}$ such that 
          \[\id_{\bL}\otimes\rd:\bL\otimes_{\BBdRp}\calO\bB_{\dR,\pd,Z}^+\to \bL\otimes_{\BBdRp}\calO\bB_{\dR,\pd,Z}^+\otimes_{\calO_{\wtx_Z}}\Omega^{1,\log}_{\wtx_Z}\{-1\}\]
          induces a flat connection $\nabla_{\bL}$ on $\calD(\bL)$ such that $(\calD(\bL),\nabla_{\bL})$ defines an object in $\MIC(\wtX_{n})^{\Hsmall}(Z)$.

          \item For any $(\calD,\nabla)\in \MIC^{\Hsmall}(\wtX_n)(Z)$ of rank $r$, define 
          \[\nabla_{\calD}:=\nabla\otimes\id+\id\otimes\rd:\calD\otimes_{\calO_{\wtX_Z}}\calO\bB_{\dR,\pd,Z}^+\to \calD\otimes_{\calO_{\wtX_Z}}\calO\bB_{\dR,\pd,Z}^+\otimes_{\calO_{\wtx_Z}}\Omega^{1,\log}_{\wtx_Z}\{-1\}.\]
          Then 
          \[\bL(\calD,\nabla):=(\calD\otimes_{\calO_{\wtX_Z}}\calO\bB_{\dR,\pd,Z}^+)^{\nabla_{\calD} = 0}\]
          is an object of rank $r$ in $\rL\rS(X,\bB_{\dR,n}^+)^{\Hsmall}(Z)$.

          \item The functors $\bL\mapsto (\calD(\bL),\nabla_{\bL})$ and $(\calD,\nabla)\mapsto \bL(\calD,\nabla)$ in items (1) and (2) respectively define an equivalence of categories
          \[\rho_{\wtx_Z}:\rL\rS(X,\bB_{\dR,n}^+)^{\Hsmall}(Z)\xrightarrow{\simeq} \MIC(\wtX_{n})^{\Hsmall}(Z)\]
          which preserves ranks, tensor products and dualities. Moreover, for any $\bL\in \rL\rS(X,\bB_{\dR,n}^+)^{\Hsmall}(Z)$ with the associated $(\calD,\nabla)\in\MIC(\wtX_{n})^{\Hsmall}(Z)$, this equivalence induces a quasi-isomorphism of complexes of sheaves of $\BBdRpn(Z)$-modules on $X_{Z,\et}$
          \[\rR\nu_*\bL\simeq \rD\rR(\calD,\nabla).\]
          In particular, we have a quasi-isomorphism of complexes of $\BBdRpn(Z)$-modules
          \[\rR\Gamma(X_{Z,v},\bL)\simeq \rR\Gamma(X_{Z,\et},\rD\rR(\calD,\nabla)).\]
      \end{enumerate}
  \end{thm}
  \begin{rmk}\label{rmk:IntroPeriodSheaf}
      The reduction of $\calO\bB_{\dR,\pd,Z}^+$ modulo $t$ already appeared in \cite{AHLB23} and \cite{SW24}. In \cite{AHLB23}, this reduction was indicated by $\calB_{\wtx}$ and constructed in another paper \cite{AHLB25}, based on the prismatic theory of Bhatt and Lurie \cite{BL22a,BL22b}. In \cite{SW24}, this reduction was indicated by $\calO\widehat \bC_{\pd}^+$ and constructed by using the integral Faltings' extension. It seems that their methods do not work in our setting. Our construction of $\calO\bB_{\dR,\pd,Z}^+$ is \emph{self-contained} in the sense that we will not use either the prismatic theory or the integral Faltings' extension any more. 
  \end{rmk}
  By construction, it is clear that $\rho_{\wtx_Z}$ is functorial in $Z$ and then Theorem \ref{thm:Stacky RH} follows immediately. We shall construct the desired period sheaf with connection $(\calO\bB_{\dR,\pd,Z}^+,\rd)$ in \S\ref{sec:period sheaf} and prove Theorem \ref{thm:global RH-intro} in \S\ref{sec:global RH}. The approach is standard: We reduce the proof to the case where $\frakX$ admits a chart and then we check everything locally. We shall show the local calculations in \S\ref{sec:local RH}. Finally, we emphasize that, as we will work with semi-stable formal schemes over an integral perfectoid $\calO_C$-algebra $A^+$ (with $\Spa(A,A^+)\in \Perfd$), we have to equip $A^+$ with a suitable log-structure, which we call the canonical log-structure. The \S\ref{sec:semi-stable schemes} is devoted to doing some preparations on this. 


\subsection{Notations}\label{ssec:notation}
  Throughout this paper, let $K$ be a complete discrete valuation field of mixed characteristic $(0,p)$ with ring of integers $\calO_K$ and perfect residue field $\kappa$. Let $\pi\in\calO_K$ be a uniformizer of $\calO_K$ and $E(u)\in \rW(\kappa)[u]$ be the minimal polynomial of $\pi$, where $\rW(\kappa)$ is the ring of Witt vectors over $\kappa$.
  Let $C$ be the $p$-adic completion of a fixed algebraic closure of $K$ with ring of integers $\calO_C$, maximal ideal $\frakm_C$ of $\calO_C$ and residue field $\overline \kappa$.
  Let $\bfA_{\inf,K} = \Ainf\otimes_{\rW(\kappa)}\calO_K$ and $\BdRp$ be the ramified infinitesimal period ring with respect to $K$ and the de Rham period ring, respectively.
  Fix an embedding $p^{\bQ}\subset C^{\times}$, which induces an embedding $\varpi^{\bQ}\subset C^{\flat\times}$, where $\varpi = (p,p^{1/p},p^{1/p^2},\dots)\in C^{\flat}$. Fix a coherent system $\{\zeta_{p^n}\}_{n\geq 0}$ of primitive $p^n$-th roots of unity in $C$, and let $\epsilon:=(1,\zeta_p,\zeta_p^2,\dots)\in C^{\flat}$ and $u = [\epsilon^{\frac{1}{p}}]-1\in\bfA_{\inf,K}$. There is a canonical surjection $\theta_K:\bfA_{\inf,K}\to \calO_C$ whose kernel $\Ker(\theta_K)$ is a principal ideal of $\bfA_{\inf,K}$ and let $\xi_K\in \Ker(\theta_K)$ be a generator. For example, we may choose $\xi_K = \frac{\phi(u)}{u}$ when $K = \rW(\kappa)[\frac{1}{p}]$ is unramified and $\xi_K = \pi-[\underline \pi]$ for general $K$, where $\underline \pi = (\pi,\pi^{\frac{1}{p}},\dots)\in \calO_C^{\flat}$ is determined by a compatible sequence of $p^n$-th root $\pi^{\frac{1}{p^n}}$ of $\pi$. Let $t = \log([\epsilon])\in \BdRp$ be the Fontaine's $p$-adic analogue of ``$2\pi i$''. For any (sheaf of) $\bfA_{\inf,K}$-module $M$ and any $n\in \bZ$, denote by 
  \[M\{n\}:=M\otimes_{\bfA_{\inf,K}}\Ker(\theta_K)^{\otimes n}\]
  the $n$-th Breuil--Kisin--Fargues twist of $M$, which can be trivialized by $\xi_K^n$; that is, we have the identification $M\{n\} = M\cdot\xi_K^n$.
  Using this, we can regard $M$ as a sub-$\bfA_{\inf,K}$-module of $M\{-1\}$ via the identification $M = \xi_K M\{-1\}$. In what follows, we always omit the subscript $K$ when $K = \rW(\kappa)[\frac{1}{p}]$.

  Fix a ring $R$. If an element $x\in R$ admits arbitrary pd-powers, we denote by $x^{[n]}$ its $n$-th pd-power (i.e. analogue of $\frac{x^n}{n!}$) in $R$. Put $E_i = (0,\dots,1,\dots,0)\in \bN^d$ with $1$ appearing exactly in the $i$-th component. For any $J=(j_1,\dots,j_d)\in\bN^d$ and any $x_1,\dots,x_d\in R$, we put 
  \[\underline x^J:=x_1^{j_1}\cdots x_d^{j_d}\]
  and if moreover, $x_i$ admits arbitrary pd-powers in $A$ for all $i$, we put
  \[\underline x^{[J]}:=x_1^{[j_1]}\cdots x_d^{[j_d]}.\]
  For any $\alpha\in\bN[1/p]\cap[0,1)$, we put
  \[\zeta^{\alpha} = \zeta_{p^n}^m\]
  if $\alpha = \frac{m}{p^n}$ such that $p,m$ are coprime in $\bN$. If $x\in A$ admits compatible $p^n$-th roots $x^{\frac{1}{p^n}}$, we put
  \[x^{\alpha} = x^{\frac{m}{p^n}}\]
  for $\alpha = \frac{m}{p^n}$ as above. In general, for any $\underline \alpha:=(\alpha_1,\dots,\alpha_d)\in(\bN[1/p]\cap[0,1))^d$ and any $x_1,\dots,x_d$ admitting compatible $p^n$-th roots in $A$, we put
  \[\underline x^{\underline \alpha}:=x_1^{\alpha_1}\cdots x_d^{\alpha_d}.\]

\subsection*{Acknowledgement}
  The work started as a research project of the 2024 Summer School on Algebra and Number Theory at Peking University, organized by the School of Mathematical Science, Peking University and the Academy of Mathematics and System Science, Chinese Academy of Science. The authors thank the organizers, especially Liang Xiao and Shouwu Zhang, for providing the opportunity and the institutes for providing a great environment for our collaboration. We would like to thank Jiahong Yu for explaining his paper \cite{Yu24} and thank Yue Chen and Wenrong Zou for their interest in this project. The authors also thank Ben Heuer and Ruochuan Liu for valuable comments.
  The idea for constructing $(\calO\bB_{\dR,\pd}^+,\rd)$ occurs during the period of the last author's collaboration \cite{CLWYZ} with Zekun Chen, Ruochuan Liu, Jiahong Yu and Xinwen Zhu, and we would like to thank them for allowing us to exhibit Theorem \ref{thm:Stacky RH} here. The authors also thank anonymous referees for their valuable comments and suggestions.
  Y.W. is partially supported by the CAS Project for Young Scientists in Basic Research, Grant No. YSBR-032, and by the New Cornerstone Science Foundation.

\section{Review of semi-stable formal schemes}\label{sec:semi-stable schemes}
  We give a quick review of semi-stable formal schemes over an integral perfectoid ring $A^+$, following \cite{CK19} and \cite{SW24}. We first recall the canonical log structure on $A^+$.

  Put $A:= A^+[\frac{1}{p}]$ and denote by $A^{+,\flat}$ and $A^{\flat}$ the tiltings of $A^+$ and $A$, respectively. We remark that in general, for $?\in \{\emptyset,\flat\}$, the natural map $A^{+,?}\to A^?$ is not injective because $A^+$ is not necessarily $p$-torsion free. But it is still reasonable to define
  \[A^{+,?}\cap A^{?,\times}:=\{x\in A^{+,?}\mid\text{the image of $x$ in $A^?$ is invertible}.\}\]
  The following result was proved in \cite{SW24}.
  \begin{prop}\label{prop:unit group}
      For $?\in\{\emptyset,+\}$, denote by $\sharp:A^{?,\flat} = \varprojlim_{x\mapsto x^p}A^? \to A^?$ the projection to the first component (which is well-known as the sharp map). Then the following statements are true.
      \begin{enumerate}
          \item[(1)] We have $A^+\cap A^{\times} = (A^{+,\flat}\cap A^{\flat,\times})^{\sharp}\cdot A^{+,\times}$.

          \item[(2)] If, moreover, $(A,A^+)$ is a perfectoid Tate algebra (in particular, $A^+$ is $p$-torsion free and integrally closed in $A$), then 
          \[A^+\cap A^{\times} = (A^{+,\flat}\cap A^{\flat,\times})^{\sharp}\cdot (1+pA^+).\]
      \end{enumerate}
  \end{prop}
  \begin{proof}
      See \cite[Cor. 2.11 and Cor. 2.4]{SW24}.
  \end{proof}
  As a consequence, the log-structure $(A^{\times}\cap A^+\hookrightarrow A^+)$ on $A^+$ is associated to the pre-log-structure $(A^{+,\flat}\cap A^{\flat,\times}\xrightarrow{\sharp}A^+)$. In general, one can make the following definition.
  \begin{dfn}[Canonical log-structure]\label{dfn:canonical log-structure}
      
      Denote by $\theta_K:\bA_{\inf,K}(A^+):=\rW(A^{+,\flat})\otimes_{\rW(\kappa)}\calO_K\to A^+$ the canonical surjection and then $\xi_K$ is a generator of $\Ker(\theta_K)$.
      For any 
      \[\bfB\in\{\bA_{\inf,K}(A^+),\bA_{\inf,K}(A^+)/\pi^n, \bA_{\inf,K}(A^+)/\xi_K^m\}\]
      with $n,m\geq 1$, the \emph{canonical log-structure} on $\bfB$ is the log-structure associated to the pre-log-structure
      \[A^{+,\flat}\cap A^{\flat,\times}\xrightarrow{[\cdot]}\bA_{\inf,K}(A^+)\to\bfB.\]
      In particular, the canonical log-structure on $A^+$ is exactly the log-structure $(A^+\cap A^{\times}\hookrightarrow A^+)$.
  \end{dfn}
  \begin{lem}\label{lem:saturated integral}
      Fix a $\bfB\in\{\bA_{\inf,K}(A^+),\bA_{\inf,K}(A^+)/\pi^n, \bA_{\inf,K}(A^+)/\xi_K^m\}$. If $A^+$ is $p$-torsion free, then the canonical log-structure on $\bfB$ is integral. If moreover, $(A,A^+)$ is a perfectoid Tate algebra, then the canonical log-structure on $\bfB$ is saturated.
  \end{lem}
  \begin{proof}
      The torsion-freeness of $A^+$ implies that the natural map $A^{+,\flat}\to A^{\flat}$ is injective (cf. \cite[(2.1.2.8)]{CS}), yielding that the group-completion of $A^{+,\flat}\cap A^{\flat}$ is exactly $A^{\flat,\times}$.
      To see the canonical log-structure on $\bfB$ is integral, it suffices to show that the monoid $A^{+,\flat}\cap A^{\flat,\times}$ is integral; equivalently, for any $f,f^{\prime},g\in A^{+,\flat}\cap A^{\flat,\times}$, we have to show that
      \[fg = f^{\prime}g\Leftrightarrow f = f^{\prime}.\]
      However, this is trivial because $A^{+,\flat}\cap A^{\flat,\times}\subset A^{\flat,\times}$.

      Suppose that $(A,A^+)$ is a perfectoid Tate algebra. Then so is $(A^{\flat},A^{+,\flat})$. To see that the canonical log-structure on $\bfB$ is saturated, it suffices to show that the monoid $A^{+,\flat}\cap A^{\flat,\times}$ is saturated; equivalently, we have to show that for any $h\in A^{\flat,\times}$ and $n\geq 0$, 
      \[h^n\in A^{+,\flat}\cap A^{\flat,\times}\Leftrightarrow h\in A^{+,\flat}\cap A^{\flat,\times}.\]
      But this follows immediately from that $A^{+,\flat}$ is integrally closed in $A^{\flat}$.
  \end{proof}

  Let $\alpha:(M\to B)\to(\overline M\to \overline B)$ be a morphism of pre-log-structures such that the induced map $M\to\overline M$ is a surjection of integral monoids and the induced map $B\to \overline B$ is surjective as well. In general, $\alpha$ is not necessarily \emph{exact}; that is, the log-structures on $\overline B$ associated to $(\overline M\to \overline B)$ and the composite $(M\to B\to \overline B)$ may not coincide. So the closed embedding $\Spec(\overline B)\to \Spec(B)$ (or $\Spf(\overline B)\to \Spf(B)$ if we work with formal schemes) is not exact in the sense of log-geometry. To overcome this problem, we need the following construction, called \emph{exactification}, which we learn from \cite{Kos20}.
  \begin{construction}[Exactification]\label{construction:exactification}
      Let $\alpha:(M\to B)\to(\overline M\to \overline B)$ be a morphism of pre-log-structures such that the induced map $M\to\overline M$ is a surjection of integral monoids and the induced map $B\to \overline B$ is surjective as above. Denote by $\alpha^{\gp}:M^{\gp}\to \overline{M}^{\gp}$ the associated surjection on the group-completions and put $G:=\Ker(\alpha^{\gp})$. As $M$ and $\overline M$ are integral, they are sub-monoids of $M^{\gp}$ and $\overline{M}^{\gp}$ respectively. Define
      \[M^{\prime}:=(\alpha^{\gp})^{-1}(\overline M)\subset M^{\gp} \text{ and }B^{\prime}:=B\otimes_{\bZ[M\cap G]}\bZ[G].\]
      Clearly, $M^{\prime}$ is the sub-monoid of $M^{\gp}$ generated by $M$ and $G$, and there exist canonical morphisms 
      \[(M\to B)\to(M^{\prime}\to B^{\prime}) \text{ and }(M^{\prime}\to B^{\prime})\xrightarrow{\alpha^{\prime}}(\overline M\to \overline B)\]
      of pre-log-structures such that $\alpha^{\prime}$ is surjective as well and that $\alpha$ uniquely factors as 
      \[(M\to B)\to(M^{\prime}\to B^{\prime})\xrightarrow{\alpha^{\prime}}(\overline M\to \overline B).\]
      Now, one can directly check that 
      \[M^{\prime}/(M^{\prime}\cap B^{\prime,\times}) \cong M/(M\cap B^{\times})\]
      and thus the log-structures on $\overline B$ associated to the pre-log-structures $(\overline M\to \overline B)$ and $(M^{\prime}\to B^{\prime}\to \overline B)$ are eventually same.
  \end{construction}

  From now on, we always assume that $A^+$ is a perfectoid $\calO_C$-algebra such that $(A,A^+)$ is a perfectoid Tate algebra; that is, we always assume $Z=\Spa(A,A^+)\in \Perfd$. Let us recall the definition of semi-stable formal schemes over $A^+$.
  \begin{dfn}\label{dfn:small affine}
      Fix an affinoid perfectoid $Z = \Spa(A,A^+)\in\Perfd$.
      \begin{enumerate}
          \item[(1)] A $p$-complete $A^+$-algebra $\calR_Z$ is called \emph{small semi-stable} of relative dimension $d$ over $A^+$, if there exists an \'etale morphism of $p$-adic formal schemes
      \[\psi: \Spf(\calR_Z)\to \Spf(A^+\za T_0,\dots,T_r,T_{r+1}^{\pm 1},\dots, T_d^{\pm 1}\ya/(T_0\cdots T_r-p^a))\]
      for some $a\in\bQ_{\geq 0}$ and $1\leq r\leq d$, where $A^+\za T_0,\dots,T_r,T_{r+1}^{\pm 1},\dots, T_d^{\pm 1}\ya/(T_0\cdots T_r-p^a))$ stands for the $p$-adic completion of
      \[A^+[ T_0,\dots,T_r,T_{r+1}^{\pm 1},\dots, T_d^{\pm 1}]/(T_0\cdots T_r-p^a).\]
      Equip $\calR_Z$ with the log-structure associated to the pre-log-structure
      \begin{equation}\label{equ:log-structure on chart-I}
          M_{r,a}(A^+):=(\oplus_{i=0}^r\bN\cdot e_i)\oplus_{\bN\cdot(e_1+\cdots+e_r)}(A^{\times}\cap A^+)\xrightarrow{(\sum_{i=0}^rn_ie_i,x)\mapsto x\prod_{i=0}^rT_i^{n_i}} \calR_Z,
      \end{equation}
      where $(\oplus_{i=0}^r\bN\cdot e_i)\oplus_{\bN\cdot(e_1+\cdots+e_r)}(A^{\times}\cap A^+)$ denotes the push-out of monoids
      \[\xymatrix@C=0.5cm{
        \bN\ar[d]^{1\mapsto p^a}\ar[rrr]^{1\mapsto e_0+\cdots+e_r\qquad}&&&\bN\cdot e_1\oplus\cdots\oplus\bN\cdot e_r\\
        (A^{\times}\cap A^+).
      }\]
      In this case, we also say that $\Spf(\calR_Z)$ is \emph{small affine}. We call such a $\psi$ (resp. $T_0,\dots,T_d$) a \emph{chart} (resp. \emph{coordinates}) on $\calR_Z$ or on $\Spf(\calR_Z)$. Note that the generic fiber $\Spa(R_Z,R_Z^+)$ of $\Spf(\calR_Z)$ is smooth over $Z$ and endowed with the associated chart $\psi$, where $R_Z^+=\calR_Z$.

      \item[(2)] A \emph{semi-stable} formal scheme $\frakX_Z$ of relative dimension $d$ over $A^+$ is a $p$-adic formal scheme $\frakX_Z$ together with a log-structure $\calM_{\frakX_Z}$ over $A^+$ (endowed with the canonical log-structure), which is \'etale locally of the form $\Spf(\calR_Z)$ such that $\calR_Z$ is small semi-stable of relative dimension $d$ over $A^+$ with the log-structure as defined above. Note that the generic fiber $X_Z$ of $\frakX_Z$ is always smooth over $Z$.
      \end{enumerate}
  \end{dfn}

  There are two typical examples, which are the main cases that we shall deal with in this paper.
  \begin{exam}\label{exam:semi-stable formal scheme}
      \begin{enumerate}
          \item[(1)] Any smooth formal scheme $\frakX_Z$ over $A^+$ with the log-structure induced from the canonical log-structure on $A^+$ (i.e., the log-structure which is associated to the composite $(A^{\times}\cap A^+)\hookrightarrow A^+\to \calO_{\frakX_Z}$) is always semi-stable.

          \item[(2)] Let $\frakX$ be a semi-stable formal scheme over $\calO_C$ in the sense of \cite{CK19} with the log-structure $(\calM_{\frakX} = \calO_X^{\times}\cap\calO_{\frakX}\hookrightarrow\calO_{\frakX})$ as considered in \cite[\S1.5 and \S1.6]{CK19}, where $X$ is the generic fiber of $\frakX$. For any $Z = \Spa(A,A^+)$, denote by $\frakX_Z$ the base-change of $\frakX$ along $\Spf(A^+)\to\Spf(\calO_C)$ with the log-structure $\calM_{\frakX_Z}$ induced from this fiber product. Then $\frakX_Z$ is a semi-stable formal scheme over $A^+$.
      \end{enumerate}
  \end{exam}
  \begin{dfn}\label{dfn:liftablity}
      Fix an affinoid perfectoid $Z = \Spa(A,A^+)\in\Perfd$ and let $\frakX_Z$ be a semi-stable formal scheme over $A^+$ with log-structure $\calM_{\frakX_Z}$. By a \emph{lifting of $\frakX_Z$ over $\bA_{\inf,K}(Z)$}, we mean a compatible sequence $\{(\wtx_{Z,n},\calM_{\wtx_{Z,n}}\}_{n\geq 1}$ of flat $p$-adic formal schemes $\wtx_{Z,n}$ with log-structure $\calM_{\wtx_{Z,n}}$ over $\bA_{\inf,K}(Z)/\xi_K^n$ (equipped with the canonical log-structure) satisfying the following conditions:
      \begin{enumerate}
          \item[(1)] We have $(\wtx_{Z,1},\calM_{\wtx_{Z,1}}) = (\frakX_Z,\calM_{\frakX_Z})$.
          
          \item[(2)] For any $n\geq m\geq 1$, we have 
          \[\calO_{\wtx_{Z,m}} = \calO_{\wtx_{Z,n}}\otimes_{\bA_{\inf,K}(Z)/\xi_K^n}\bA_{\inf,K}(Z)/\xi_K^m\]
          such that the log-structure $\calM_{\wtx_{Z,m}}$ is induced from the composite
          \[\calM_{\wtx_{Z,n}}\to\calO_{\wtx_{Z,n}}\to\calO_{\wtx_{Z,m}}.\]
      \end{enumerate}
      Define 
      \[(\calO_{\wtx_Z},\calM_{\wtx_Z}):=\varprojlim_n(\calO_{\wtx_{Z,n}},\calM_{\wtx_{Z,n}})\]
      and then $\calO_{\wtx_Z}$ is a sheaf of flat algebras over $\bA_{\inf,K}(Z)$ and $\calM_{\wtx_Z}$ defines a log-structure on $\calO_{\wtx_Z}$. We denote the above lifting $\{(\wtx_{Z,n},\calM_{\wtx_{Z,n}})\}_{n\geq 1}$ by $(\wtx_Z,\calM_{\wtx_Z})$, which is referred to as a lifting of $\frakX_Z$ over $\bA_{\inf,K}(Z)$, for simplicity. We say a semi-stable formal scheme $\frakX$ is \emph{liftable}, if it admits at least one lifting $(\wtx_Z,\calM_{\wtx_Z})$ over $\bA_{\inf,K}(Z)$.
  \end{dfn}
  \begin{exam}\label{exam:small affine is liftable}
      Suppose $\frakX_Z = \Spf(\calR_Z)$ is small affine in the sense of Definition \ref{dfn:small affine}. Then $\frakX_Z$ is liftable. Indeed, by \cite[Lem. 3.1]{SW24}, the $p$-complete log-cotangent complex (in the sense of \cite[\S 8]{Ols}) of $\calR_Z$ over $A^+$ is
      \[\widehat \rL_{(M_{r,a}(A^+)\to \calR_Z)/(A^{\times}\cap A^+\to A^+)} \simeq \Omega^{1,\log}_{\calR_Z}[0],\]
      where $\Omega^{1,\log}_{\calR_Z}$ denotes the module of continuous log-differentials of $\calR_Z$ over $A^+$, which is a finite free $\calR_Z$-module of rank $d$ and isomorphic to
      \begin{equation}
          \Omega^{1,\log}_{\calR_Z} \cong \left((\oplus_{i=0}^r\calR_Z\cdot e_i)/\calR_Z\cdot (e_0+\cdots+e_r)\right)\oplus\left(\oplus_{j=r+1}^d\calR_Z\cdot e_j\right),
      \end{equation}
      where $e_j$ stands for $\dlog T_j$ for any $r+1\leq j\leq d$. By deformation theory, (up to non-canonical isomorphisms,) $\frakX_Z$ admits a unique lifting $\wtx_Z = \Spf(\wtr_Z)$ over $\bA_{\inf,K}(Z)$. More precisely, consider the log-ring 
      \[M_{r,a}(Z) = (\oplus_{i=0}^r\bN\cdot e_i)\oplus_{\bN\cdot(e_0+\cdots+e_r)}M_Z\xrightarrow{\alpha} \bA_{\inf,K}(Z)\za T_0,\dots,T_r,T_{r+1}^{\pm 1},\dots,T_d^{\pm 1}\ya/(T_0\cdots T_r-[\varpi^a]),\]
      where $M_Z$ denotes the canonical log-structure on $\bA_{\inf,K}(Z)$, the map $\alpha$ of monoids is determined by sending each $e_i$ to $\alpha(e_i) = T_i$, and $M_{r,a}(Z)$ denotes the push-out of monoids
      \[\xymatrix@C=0.5cm{
        \bN\ar[rrrrrr]^{1\mapsto e_0+\cdots+e_r}\ar[d]_{1\mapsto[\varpi^a]}&&&&&&\bN\cdot e_0\oplus\cdots\oplus\bN\cdot e_r\\
        M_Z.
      }\]
      By \'etaleness of the chart $\psi$, there exists a unique lifting of $\calR_Z$ over
      
      \[\bA_{\inf,K}(Z)\za T_0,\dots,T_r,T_{r+1}^{\pm 1},\dots,T_d^{\pm 1}\ya/(T_0\cdots T_r-[\varpi^a]),\]
      which is exactly $\wtr_Z$, and a unique $\bA_{\inf,K}(Z)\za T_0,\dots,T_r,T_{r+1}^{\pm 1},\dots,T_d^{\pm 1}\ya/(T_0\cdots T_r-[\varpi^a])$-linear lifting 
      \[\widetilde \psi:\bA_{\inf,K}(Z)\za T_0,\dots,T_r,T_{r+1}^{\pm 1},\dots,T_d^{\pm 1}\ya/(T_0\cdots T_r-[\varpi^a])\to\wtr_Z\]
      of $\psi$ such that the log-structure on $\wtr_Z$ is associated to the composite
      \[M_{r,a}(Z)\xrightarrow{\alpha} \bA_{\inf,K}(Z)\za T_0,\dots,T_r,T_{r+1}^{\pm 1},\dots,T_d^{\pm 1}\ya/(T_0\cdots T_r-[\varpi^a])\to\wtr_Z.\]
      Here, $\bA_{\inf,K}(Z)\za T_0,\dots,T_r,T_{r+1}^{\pm 1},\dots,T_d^{\pm 1}\ya/(T_0\cdots T_r-[\varpi^a])$ stands for the $(p,\xi_K)$-adic completion of 
      \[\bA_{\inf,K}(Z)[ T_0,\dots,T_r,T_{r+1}^{\pm 1},\dots,T_d^{\pm 1}]/(T_0\cdots T_r-[\varpi^a]).\]
  \end{exam}

  At the end of this section, we give a typical example about Construction \ref{construction:exactification}.
  \begin{lem}\label{lem:existence of log-structure}
      Keep notations in Example \ref{exam:small affine is liftable}. For any affinoid perfectoid $U = \Spa(S,S^+)\in X_{Z,v}$, there exist $\underline t_i\in S^{\flat,\times}\cap S^{+,\flat}$ and $v_i\in \AAinf(U)^{\times}$ satisfying
      \[([\underline t_0]v_0)\cdots([\underline t_r]v_r) = [\varpi^a]\]
      such that the image of $[\underline t_i]v_i$ modulo $\xi = \frac{\phi(u)}{u}$ is $T_i$ for any $0\leq i\leq r$.
  \end{lem}
  \begin{proof}
      As $T_i\in S^{\times}\cap S^+$, by Proposition \ref{prop:unit group}, there exist $\underline t_i\in S^{\flat,\times}\cap S^{+,\flat}$ and $v_i\in \AAinf(U)^{\times}$
      such that the reduction $[\underline t_i]v_i\mod \xi$ is $T_i$. In particular, the product
      \[[\underline t]v:=([\underline t_0]v_0)\cdots([\underline t_r]v_r)\]
      is mapped to $p^a$ modulo $\xi$, where $\underline t = \prod_{i=0}^r\underline t_i\in S^{\flat,\times}\cap S^{+,\flat}$ and $v = \prod_{i=0}^rv_i\in\AAinf(U)^{\times}$. As a consequence, $\underline{t}^{\sharp}$ is a unit-multiple of $p^a$ in $S^+$, yielding that $t$ is a unit-multiple of $\varpi^a$ in $S^{+,\flat}$. So we have
      \[[\underline t]v = [\varpi^a]v^{\prime}\]
      for some $v^{\prime}\in 1+\xi\AAinf(U)$. So we can conclude by changing $v_0$ to $v_0(v^{\prime})^{-1}$.
  \end{proof}
  \begin{cor}\label{cor:existence of log-structure}
      Keep notations in Lemma \ref{lem:existence of log-structure} and Example \ref{exam:small affine is liftable}. Then there exists a morphism
      \[(M_{r,a}(Z)\to\wtr_Z)\to (M_U\to \bA_{\inf,K}(U))\]
      lifting the natural morphism $(M_{r,a}(A^+)\to\calR_Z)\to (M_U\to S^+)$, where $M_U$ denotes the canonical log-structures on $\bA_{\inf,K}(U)$ and $S^+$ by abuse of notations.
  \end{cor}
  \begin{proof}
      Put $t_i=[\underline t_i]v_i$ for any $0\leq i\leq r$ and pick $t_{r+1},\dots,t_d\in \bA_{\inf,K}(U)^{\times}$ lifting $T_{r+1},\dots,T_d$ respectively. Then the map
      \[(M_{r,a}(Z)\to\bA_{\inf,K}(Z)\za T_0,\dots,T_r,T_{r+1}^{\pm 1},\dots,T_d^{\pm 1}\ya/(T_0\cdots T_r-[\varpi^a]))\to (M_U\to\bA_{\inf,K}(U))\]
      sending each $T_i$ to $t_i$ is well-defined and lifts the natural map
      \[(M_{r,a}(Z)\to A^+\za T_0,\dots,T_r,T_{r+1}^{\pm 1},\dots,T_d^{\pm 1}\ya/(T_0\cdots T_r-p^a))\to (M_U\to S^+).\]
      Then one can conclude by $(p,\xi_K)$-complete \'etaleness of the map
      \[\bA_{\inf,K}(Z)\za T_0,\dots,T_r,T_{r+1}^{\pm 1},\dots,T_d^{\pm 1}\ya/(T_0\cdots T_r-[\varpi^a])\to\wtr_Z.\]
      This completes the proof.
  \end{proof}

  \begin{exam}\label{exam:exactification}
      Fix an affinoid perfectoid $Z = \Spa(A,A^+)\in\Perfd$. Suppose that $\frakX_Z = \Spf(\calR_Z)$ is small semi-stable with the generic fiber $X_Z = \Spa(R_Z,R_Z^+)$. Let $U = \Spa(S,S^+)\in X_{Z,v}$ be affinoid perfectoid and $(M_{r,a}(Z)\to\wtr_Z)$ be the lifting of $(M_{r,a}(A^+)\to\calR_Z)$ given in Example \ref{exam:small affine is liftable}. Consider the $(p,\xi_K)$-complete tensor product
      \[\wtr_Z\widehat \otimes_{\bA_{\inf,K}(Z)}\bA_{\inf,K}(U)\]
      with log-structure induced by fiber product, which is indeed the log-structure associated to the following pre-log-structure
      \[M_{r,a}(U):=M_{r,a}(Z)\oplus_{M_Z}M_U = (\oplus_{i=0}^r\bN\cdot e_i)\oplus_{\bN\cdot(e_0+\cdots+e_r)}M_U\xrightarrow{e_i\mapsto T_i\otimes1,~\forall~i}\wtr_Z\widehat \otimes_{\bA_{\inf,K}(Z)}\bA_{\inf,K}(U).\]
      Here, $M_Z$ and $M_U$ denote the canonical log-structure on $\bA_{\inf,K}(Z)$ and $\bA_{\inf,K}(U)$ respectively, and $M_{r,a}(Z)\oplus_{M_Z}M_U$ denotes the push-out of monoids
      \[\xymatrix@C=0.45cm{
        M_U&&\ar[ll]M_Z\ar[rr]&& M_{r,a}(Z).
      }\]
      Then there exists a surjection 
      \begin{equation}\label{equ:example for exactification-I}
          \theta_U:(M_{r,a}(U)\to\wtr_Z\widehat \otimes_{\bA_{\inf,K}(Z)}\bA_{\inf,K}(U))\to (M_U\to S^+)
      \end{equation}
      induced by the composite
      \[\wtr_Z\widehat \otimes_{\bA_{\inf,K}(Z)}\bA_{\inf,K}(U)\to\calR_Z\widehat \otimes_{A^+}S^+\to S^+,\]
      where by abuse of notation, we still denote by $M_U\to S^+$ the canonical log-structure on $S^+$. Put
      \[t_i:=[\underline t_i]v_i\]
      such that $\underline t_i$ and $v_i$ are as in Lemma \ref{lem:existence of log-structure}. One can directly check that
      \[G:=\Ker(M_{r,a}(U)^{\gp}\to M_U^{\gp})\]
      is exactly
      \[\begin{split}
        G &= \left(\oplus_{i=0}^r\bZ\cdot(e_i,t_i^{-1})\right)\oplus_{\bZ\cdot (e_0+\cdots+e_r,\prod_{i=0}^rt_i^{-1} = [\varpi^{-a}])}(1+\xi_K\bA_{\inf,K}(U))\\
        &\cong \left((\oplus_{i=0}^r\bZ\cdot(e_i,t_i^{-1}))/\bZ\cdot(e_0+\cdots+e_r,t_1^{-1}\cdots t_r^{-1})\right)\oplus(1+\xi_K\bA_{\inf,K}(U)).
      \end{split}\]
      Thus, the $(p,\xi_K)$-adic completion of the exactification associated to the surjection \eqref{equ:example for exactification-I} is given by
      \begin{equation}\label{equ:example for exactification-II}
          \theta_U^{\prime}:(M_U\to (\wtr_Z\widehat \otimes_{\bA_{\inf,K}(Z)}\bA_{\inf,K}(U))[(\frac{T_0}{t_0})^{\pm 1},\cdots,(\frac{T_r}{t_r})^{\pm 1}]^{\wedge_{(p,\xi_K)}})\to (M_U\to S^+),
      \end{equation}
      where $\Box^{\wedge_{(p,\xi_K)}}$ stands for the $(p,\xi_K)$-adic completion and $\theta_U^{\prime}$ is the extension of $\theta_U$ sending each $\frac{T_i}{t_i}$ to $1$.
  \end{exam}

\section{The period sheaf with connection $(\calO\bB_{\dR,\pd}^+,\rd)$}\label{sec:period sheaf}
  Fix a $Z = \Spa(A,A^+)\in \Perfd$. Let $\frakX_Z$ be a liftable semi-stable formal scheme over $A^+$ with the log-structure $\calM_{\frakX_Z}$ and let $X_Z$ be its generic fiber. Let $\wtx_Z$ be a fixed lifting of $\frakX_Z$ with the log-structure $\calM_{\wtx_Z}$ over $\bA_{\inf,K}(Z)$ and denote by $\wtX_Z$ the associated lifting of $X_Z$ over $\BBdRp(Z)$. In this section, we will show that $(\wtx_Z,\calM_{\wtx_Z})$ induces a period sheaf of $\BBdRp$-algebras with a $\BBdRp$-linear connection 
  \[(\calO\bB_{\dR,\pd,Z}^+,\rd:\calO\bB_{\dR,\pd,Z}^+\otimes_{\calO_{\wtx_Z}}\Omega^{1,\log}_{\wtx_Z}\{-1\})\]
  on $X_{Z,v}$ as desired in \S\ref{sec:Introduction}.

  Before we move on, let us discuss the idea behind the construction. The starting point is the construction of the period sheaf $\Fil^0\calO\bB_{\dR}$ in \cite[Definition 2.2.10]{DLLZ} (which is denoted by $\calO\bB_{\dR}^{[0,+\infty]}$ in \cite{LZ}) when $X_Z$ admits a lifting over $\bB^+_{\dR}(Z)$. Denote by $\theta:\calO_{\widetilde X_Z}\widehat \otimes_{\bB_{\dR}^+(Z)}\bB_{\dR}^+\to \widehat \calO_{X_Z}$ the natural projection of sheaves on $X_{Z,v}$. Then $\Fil^0\calO\bB_{\dR}$ can be defined to be the $\xi_K$-adic completion  $\calO_{\widetilde X_Z}\widehat \otimes_{\bB_{\dR}^+(Z)}\bB_{\dR}^+[\frac{\Ker(\theta)}{\xi_K}]^{\wedge_{\xi_K}}$ of the blow-up algebra $\calO_{\widetilde X_Z}\widehat \otimes_{\bB_{\dR}^+(Z)}\bB_{\dR}^+[\frac{\Ker(\theta)}{\xi_K}]$, which can be deduced from the same argument in this section and will be discussed in \cite{CLWYZ}. By \cite[Corollary 2.3.17]{DLLZ}, the reduction of $\Fil^0\calO\bB_{\dR}$ modulo $\xi_K$ is locally isomorphic to a polynomial algebra over $\widehat \calO_{X_Z}$ whose free variables are given by generators of $\frac{\Ker(\theta)}{\xi_K}$. For the desired period sheaf $\calO\bB_{\dR,\pd,Z}^+$, as mentioned in Remark \ref{rmk:IntroPeriodSheaf}, its reduction should be locally isomorphic to a $p$-complete pd-polynomial algebra over $\widehat \calO_{X_Z}$. To construct $\calO\bB_{\dR,\pd,Z}^+$, we have to find a suitable integral lattice of $\calO_{\widetilde X_Z}\widehat \otimes_{\bB_{\dR}^+(Z)}\bB_{\dR}^+$, which allows us to take $p$-adic completion, and consider the pd-envelope of a suitable modification of $\frac{\Ker(\theta)}{\xi_K}$. When $\frakX_Z$ is smooth, we will see that $\calO_{\widetilde \frakX_Z}\widehat \otimes_{\bA_{\inf,K}(Z)}\bA_{\inf,K}$ gives the desired integral lattice and that the pd-envelope of $\frac{u\cdot \Ker(\theta)}{\xi_K}$ works in this case, where $\theta$ denotes the natural surjection $\calO_{\widetilde \frakX_Z}\widehat \otimes_{\bA_{\inf,K}(Z)}\bA_{\inf,K}\to\widehat \calO_{X_Z}^+$.
  When $\frakX_Z$ is semi-stable, we should, in addition, consider the exactification of $\calO_{\widetilde \frakX_Z}\widehat \otimes_{\bA_{\inf,K}(Z)}\bA_{\inf,K}$ along $\theta$ instead. In what follows, we will first deal with the small affine case, and then globalize the local construction by taking sheafification.

\subsection{Construction of $(\calO\bB_{\dR,\pd}^+,\rd)$: The small affine case}

  We first deal with the case where $\frakX_Z = \Spf(\calR_Z)$ is small affine of (relative) dimension $d$ in the sense of Definition \ref{dfn:small affine} so the generic fiber $X_Z = \Spa(R_Z,R_Z^+)$ of $\frakX_Z$ is smooth over $Z$. Then up to a \emph{non-canonical} isomorphism, we have $\wtx_Z = \Spf(\wtr_Z)$ with log-structure associated to $M_{r,a}(Z)\to\wtr_Z$ as described in Example \ref{exam:small affine is liftable}. In particular, this amounts to choosing a $(p,\xi_K)$-completely \'etale morphism
  \begin{equation}\label{equ:etale lifting chart}
      \widetilde \psi:\bA_{\inf,K}(Z)\za T_0,\dots,T_r,T_{r+1}^{\pm 1},\dots,T_d^{\pm 1}\ya/(T_0\cdots T_r-[\varpi^a])\to \wtr_Z
  \end{equation}
  lifting a chart $\psi$ on $\frakX_Z$. Fix such a lifting $\widetilde \psi$ and keep the notation as in the previous section.

  For any affinoid perfectoid $U = \Spa(S,S^+)\in X_{Z,v}$, let $\Sigma_U$ be the set of all \'etale morphisms
  \[i_{\frakY_Z}:\frakY_Z = \Spf(\calR^{\prime}_Z)\to\frakX_Z\]
  over $A^+$ such that the natural map $U\to X_Z$ factors through the generic fiber of $i_{\frakY_Z}$.
  By \'etaleness of $i_{\frakY_Z}$, the lifting $\wtx_Z$ of $\frakX_Z$ induces a unique lifting $\wty_Z = \Spf(\wtr^{\prime}_Z)$ of $\frakY_Z$ such that $i_{\frakY_Z}$ uniquely lifts to an \'etale morphism
  \[i_{\wty_Z}:\wty_Z = \Spf(\wtr_Z^{\prime})\to\wtx_Z = \Spf(\wtr_Z)\]
  over $\bA_{\inf,K}(Z)$. Then we have a natural morphism
  \[\theta_{\frakY_Z}:\wtr^{\prime}_Z\widehat \otimes_{\bA_{\inf,K}(Z)}\bA_{\inf,K}(U)\to S^+\]
  lifting the natural map $\calR_Z^{\prime}\widehat \otimes_{A^+}S^+\to S^+$, which is compatible with log-structures, cf. \eqref{equ:example for exactification-I}. Denote by 
  \[\theta_{\frakY_Z}^{\prime}:\wtr^{\prime}_Z\widehat \otimes_{\bA_{\inf,K}(Z)}\bA_{\inf,K}(U)[(\frac{T_0}{t_0})^{\pm 1},\dots,(\frac{T_r}{t_r})^{\pm 1}]^{\wedge_{(p,\xi_K)}}\to S^+\]
  its $(p,\xi_K)$-complete exactification, cf. \eqref{equ:example for exactification-II}. Here, $t_0,\dots,t_r$ are the elements described in Example \ref{exam:exactification}. For any $r+1\leq j\leq d$, one can always choose $t_j\in \bA_{\inf,K}(U)^{\times}$ whose reduction modulo $\xi_K$ is $T_j$, and we fix such a choice.
  \begin{lem}\label{lem:kernel is finite}
      The kernel of $\theta_{\frakY_Z}^{\prime}$ is the ideal of $\wtr^{\prime}_Z\widehat \otimes_{\bA_{\inf,K}(Z)}\bA_{\inf,K}(U)[(\frac{T_0}{t_0})^{\pm 1},\dots,(\frac{T_r}{t_r})^{\pm 1}]^{\wedge_{(p,\xi_K)}}$ finitely generated by 
      \[1-\frac{T_1}{t_1},\dots,1-\frac{T_d}{t_d},e,\xi_K\]
      where $e$ is an element killed by $\theta_{\frakY_Z}^{\prime}$ satisfying
      \[e\equiv e^2\mod (1-\frac{T_1}{t_1},\dots,1-\frac{T_d}{t_d}).\]
  \end{lem}
    Before proving this lemma, we remark that $1-\frac{T_0}{t_0},\dots,1-\frac{T_d}{t_d}$ and $1-\frac{T_1}{t_1},\dots,1-\frac{T_d}{t_d}$ span the same ideal in \[\wtr_Z^{\prime}\widehat \otimes_{\bA_{\inf,K}(Z)}\bA_{\inf,K}(U)[(\frac{T_0}{t_0})^{\pm 1},\dots,(\frac{T_r}{t_r})^{\pm 1}]^{\wedge_{(p,\xi_K)}}\]
      because we have $\prod_{i=0}^d\frac{T_i}{t_i} = 1$.
  \begin{proof}
      To simplify the notation, we may assume $\frakY_Z = \frakX_Z$ and thus $\wtr_Z^{\prime} = \wtr_Z$. We first deal with the case
      \[\wtr_Z = \bA_{\inf,K}(Z)\za T_0,\dots,T_r,T_{r+1}^{\pm 1},\dots,T_d^{\pm 1}\ya/(T_0\cdots T_r-[\varpi^a]).\]
      In this case, it is easy to see that 
      \[\wtr_Z\widehat \otimes_{\bA_{\inf,K}(Z)}\bA_{\inf,K}(U)[(\frac{T_0}{t_0})^{\pm 1},\dots,(\frac{T_r}{t_r})^{\pm 1}]^{\wedge_{(p,\xi_K)}} = \bA_{\inf,K}(U)\za(\frac{T_0}{t_0})^{\pm 1},\dots,(\frac{T_d}{t_d})^{\pm 1}\ya/(\prod_{i=0}^r\frac{T_i}{t_i}-1)\]
      and $\theta^{\prime}_{\frakY_Z}$ sends each $\frac{T_i}{t_i}$ to $1$. So $\Ker(\theta^{\prime}_{\frakY_Z})$ is generated by
      \[1-\frac{T_1}{t_1},\dots,1-\frac{T_d}{t_d},\xi_K\]
      as desired (by choosing $e = 0$).
      
      In general, consider the morphism
      \[\iota: \wtr_Z\to \bA_{\inf,K}(U)\]
      which is compatible with the log-structure as in Corollary \ref{cor:existence of log-structure}. Then the quotient of
      \[\wtr_Z\widehat \otimes_{\bA_{\inf,K}(Z)}\bA_{\inf,K}(U)[(\frac{T_0}{t_0})^{\pm 1},\dots,(\frac{T_r}{t_r})^{\pm 1}]^{\wedge_{(p,\xi_K)}}\]
      modulo the ideal generated by $1-\frac{T_1}{t_1},\dots,1-\frac{T_d}{t_d}$ is the quotient of 
      \[\wtr_Z\widehat \otimes_{\bA_{\inf,K}(Z)}\bA_{\inf,K}(U) = \wtr_Z\widehat \otimes_{\bA_{\inf,K}(Z)}\wtr_Z\widehat \otimes_{\wtr_Z,\iota}\bA_{\inf,K}(U)\]
      modulo the ideal generated by $T_0-t_0,\dots,T_d-t_d$, which turns out to be
      \[\wtr_Z\widehat \otimes_{\bA_{\inf,K}(Z)\za T_0,\dots,T_r,T_{r+1}^{\pm 1},\dots,T_d^{\pm 1}\ya/(T_0\cdots T_r-[\varpi^a])}\wtr_Z\widehat \otimes_{\wtr_Z,\iota}\bA_{\inf,K}(U).\]
      
      By \'etaleness of 
      \[\bA_{\inf,K}(Z)\za T_0,\dots,T_r,T_{r+1}^{\pm 1},\dots,T_d^{\pm 1}\ya/(T_0\cdots T_r-[\varpi^a])\to\wtr_Z,\]
      the kernel of 
      \[\wtr_Z\widehat \otimes_{\bA_{\inf,K}(Z)\za T_0,\dots,T_r,T_{r+1}^{\pm 1},\dots,T_d^{\pm 1}\ya/(T_0\cdots T_r-[\varpi^a])}\wtr_Z\to\wtr_Z\]
      is principally generated by some idempotent $\overline e$ satisfying $\overline{e}^2=\overline e$. 
      
      Let $e\in \wtr_Z\widehat \otimes_{\bA_{\inf,K}(Z)}\wtr_Z$ be any lifting of $\overline e$ and then one can conclude by checking that $\Ker(\theta^{\prime}_{\frakY_Z})$ is generated by $1-\frac{T_1}{t_1},\dots,1-\frac{T_d}{t_d},e,\xi_K$ directly.
  \end{proof}

  Let $A_{\frakY_Z}$ be the $(p,\Ker(\theta_{\frakY_Z}^{\prime}))$-adic completion of $\wtr^{\prime}_Z\widehat \otimes_{\bA_{\inf,K}(Z)}\bA_{\inf,K}(U)[(\frac{T_0}{t_0})^{\pm 1},\dots,(\frac{T_r}{t_r})^{\pm 1}]^{\wedge_{(p,\xi_K)}}$. Let $u = [\epsilon^{\frac{1}{p}}]-1$ be as in \S\ref{ssec:notation}.
  Denote by $B_{\frakY_Z} = A_{\frakY_Z}[u\cdot\Ker(\theta^{\prime}_{\frakY})\{-1\}]\subset A_{\frakY_Z}[\frac{1}{\xi_K}]$ the blow-up algebra
  \[B_{\frakY_Z}:=A_{\frakY_Z}[\frac{u\cdot\Ker(\theta^{\prime}_{\frakY_Z})}{\xi_K}]\]
  and define
  \[\Gamma_{\frakY_Z}:=\big(B_{\frakY_Z}[u\cdot\Ker(\theta^{\prime}_{\frakY_Z})\{-1\}]_{\pd}\big)^{\wedge} = \big(B_{\frakY_Z}[\frac{u\cdot\Ker(\theta^{\prime}_{\frakY_Z})}{\xi_K}]_{\pd}\big)^{\wedge}\]
  the $(p,\Ker(\theta^{\prime}_{\frakY_Z}))$-adic completion of the pd-envelope of $B_{\frakY_Z}$ with respect to the ideal generated by $\frac{u\cdot\Ker(\theta^{\prime}_{\frakY_Z})}{\xi_K}$. Clearly, it is an algebra over $\bA_{\inf,K}(U)[u]_{\pd}^{\wedge}$, the $(p,\xi_K)$-adic completion of the pd-envelope of $\bA_{\inf,K}(U)$ with respect to the ideal $(u)$.

  Before we move on, let us give a quick review of the sheaves of relative Robba rings $\CCtilde^1_K$ and $\CCtilde_K^{1,+}$. Let $\pi$ and $\underline \pi$ be as in \S\ref{ssec:notation}. Then for any $U = \Spa(S,S^+)\in \Perfd$, we define
  \[\CCtilde^{1,+}_K(U) = \bA_{\inf,K}(U)[\big(\frac{[\underline \pi]}{\pi}\big)^{\pm 1}]^{\wedge_{(p,\xi_K)}} \text{ and }\CCtilde^{1}_K(U) = \CCtilde^{1,+}_K(U)[\frac{1}{p}].\]
  As the kernel of $\bA_{\inf,K}(U)\to S^+$ is generated by $\beta_K:=\pi-[\underline \pi]$ (and thus $\xi_K$ is a unit-multiple of $\beta_K$ in $\bA_{\inf,K}(U)$), the $\pi$-adic topology, $[\underline \pi]$-adic topology, and $(p,\xi_K)$-adic topology on $\CCtilde^{1,+}_K(U)$ are the same.

  The following lemma is well-known.
  \begin{lem}\label{lem:relative Robba ring}
      Keep notations as above.
      \begin{enumerate}
          \item[(1)] The element $u$ is a unit-multiple of $[\epsilon^{1/p}-1]$ and thus admits arbitary pd-powers in $\CCtilde^{1,+}_K(U)$. It is also invertible in $\CCtilde^1_K(U)$.

          \item[(2)] The series $t = \log[\epsilon] = \sum_{n\geq 1}\frac{(1-[\epsilon])^n}{n}$ converges in $\CCtilde^{1,+}_K(U)$ and is a unit-multiple of $u\xi = \phi(u)$.

          \item[(3)] The natural map $\bA_{\inf,K}(U)\to S^+$ extends to a surjection $\CCtilde^{1,+}_K(U)\to S^+$ whose kernel is principally generated by $1-\frac{[\underline \pi]}{\pi} = \frac{\beta_K}{\pi}$. Moreover, the map $\bA_{\inf,K}(U)\to \BBdRp(U)$ extends uniquely to a map
          \[\CCtilde^{1}_K(U)\to \bB_{\dR}^+(U)\]
          that identifies $\bB_{\dR}^+(U)$ with the $\frac{\xi_K}{\pi}$-adic completion of $\CCtilde^{1}_K(U)$.
      \end{enumerate}
  \end{lem}
  \begin{proof}
      For Item (1): We only need to deal with the case where $K = \Qp$ and $U = \Spa(C,\calO_C)$. In this case, we denote $\CCtilde^{1,+}_K(U)$ by $\bfC^{1,+}$ for short and then we have
      \[\widetilde \bfC^{1,+} = \bfA_{\inf}[\big(\frac{[\varpi]}{p}\big)^{\pm 1}]^{\wedge_{(p,\xi)}},\]
      where $\varpi = (p,p^{1/p},\dots)\in C^{\flat}$ (cf. \S\ref{ssec:notation}). Denote by $\nu:C^{\flat}\to\bR\cup\{\infty\}$ the additive valuation on $C^{\flat}$ such that $\nu(\varpi) = 1$.

      Recall that in $\rW(\Fp[X^{\frac{1}{p^{\infty}}},Y^{\frac{1}{p^{\infty}}}])$, if we write
      \[[X]-[Y] = \sum_{n\geq 0}[P_n(X,Y)]p^n,\]
      then $P_n(X,Y)\in \Fp[X^{\frac{1}{p^n}},Y^{\frac{1}{p^n}}]$ such that $X^{\frac{1}{p^n}}-Y^{\frac{1}{p^n}}$ divides $P_n(X,Y)$. By letting $X = \epsilon^{1/p}$ and $Y=1$, we see that if we write
      \[u = [\epsilon^{\frac{1}{p}}]-1 = \sum_{n\geq 0}[u_n]p^n\]
      with $u_n\in C^{\flat}$, then we have $u_0 = \epsilon^{\frac1p}-1$ and that for any $n\geq 0$,
      \[\nu(u_n)\geq \frac{1}{p^n(p-1)}.\]
      In particular, we see that
      \[u = [u_0]\cdot (1+\sum_{n\geq 1}(\frac{p}{[\varpi]})^n[\frac{\varpi^nu_n}{u_0}])\]
      and that 
      \[\nu(\frac{\varpi^nu_n}{u_0}) = n+\frac{1}{p^n(p-1)}-\frac{1}{p-1}>0.\]
      Thus, $\sum_{n\geq 1}(\frac{p}{[\varpi]})^n[\frac{\varpi^nu_n}{u_0}]$ is topologically nilpotent in $\widetilde \bfC^{1,+}$, yielding that $v:=1+\sum_{n\geq 1}(\frac{p}{[\varpi]})^n[\frac{\varpi^nu_n}{u_0}]$ is a unit. Thus, $u = v[\epsilon^{\frac{1}{p}}-1]$ is a unit-multiple of $[\epsilon^{\frac{1}{p}}-1]$. As $[\varpi]$ is invertible in $\widetilde \bfC^{1}=\widetilde \bfC^{1,+}[\frac{1}{p}]$, so is $[u_0]$ and thus so is $u$.

      To see that $u$ admits arbitary pd-powers in $\widetilde \bfC^{1,+}$, it suffices to show that for any $n\geq 0$, $\frac{u^{p^n}}{p^n!}$ is well-defined in $\widetilde \bfC^{1,+}$. To do so, note that for any $n\geq 0$, the following product
      \[v^{p^n-1}\cdot[\frac{u_0^{p^n-1}}{\varpi^{\frac{p^n-1}{p-1}}}]\cdot\frac{p^{\frac{p^n-1}{p-1}}}{p^n!}\cdot\big(\frac{[\varpi]}{p}\big)^{\frac{p^n-1}{p-1}}\]
      is a unit in $\widetilde \bfC^{1,+}$. We can conclude by using that 
      \[\frac{u^{p^n}}{p^n!} = u\cdot\frac{u^{p^n-1}}{p^n!} = u\cdot v^{p^n-1}\cdot[\frac{u_0^{p^n-1}}{\varpi^{\frac{p^n-1}{p-1}}}]\cdot\frac{p^{\frac{p^n-1}{p-1}}}{p^n!}\cdot\big(\frac{[\varpi]}{p}\big)^{\frac{p^n-1}{p-1}}. \]

      For Item (2): It suffices to show that $t = \sum_{n\geq 1}\frac{(1-[\epsilon])^n}{n}$ converges in $\widetilde \bfC^{1,+}$, which is a unit-multiple of $\phi(u) = [\epsilon]-1$. For this purpose, it suffices to show that 
      \[\sum_{n\geq 1}\frac{(1-[\epsilon])^{n-1}}{n} = 1+\sum_{n\geq 1}(-1)^n\frac{\phi(u)^n}{n+1}\]
      converges to a unit in $\widetilde \bfC^{1,+}$. Write $\phi(u) = u^p+p\delta(u)$ for some $\delta(u)\in\Ainf$, and then we have
      \[\phi(u)= puv^{p-1}[\frac{u_0^{p-1}}{\varpi}]\frac{[\varpi]}{p}+p\delta(u) = pa\]
      for some $a\in \widetilde \bfC^{1,+}$. Thus the series
      \[\sum_{n\geq 1}(-1)^n\frac{\phi(u)^n}{n+1} = \sum_{n\geq 1}(-1)^n\frac{p^n}{n+1}a^n\]
      converges to a topologically nilpotent element in $\widetilde \bfC^{1,+}$, yielding that $\sum_{n\geq 1}\frac{(1-[\epsilon])^{n-1}}{n}$ is a unit as desired.

      For Item (3): Note that we have 
      \[\CCtilde^{1,+}_K(U) = \bA_{\inf,K}(U)\za X^{\pm 1}\ya/(\pi X-[\underline \pi]),\]
      where $\bA_{\inf,K}(U)\za X^{\pm 1}\ya$ is the $(p,\xi_K)$-adic completion of $\bA_{\inf,K}(U)[ X^{\pm 1}]$. To see the kernel of $\CCtilde^{1,+}_K(U)\to S^+$ is generated by $1-\frac{[\underline \pi]}{\pi}$, it is enough to show that the kernel of the map 
      \[\bA_{\inf,K}(U)\za X^{\pm 1}\ya/(\pi X-[\underline \pi])\to S^+\]
      sending $X$ to $1$ is generated by $X-1$. However, it is clear because $X-1,\pi-[\underline \pi]$ and $\pi X-[\underline \pi],X-1$ generate the same ideal of $\bA_{\inf,K}(U)\za X^{\pm 1}\ya$ and 
      \[\bA_{\inf,K}(U)\za X^{\pm 1}\ya/(X-1,\pi-[\underline \pi]) = S^+.\]

      It remains to show that the map $\bA_{\inf,K}(U)\to \BBdRp(U)$ extends uniquely to a map 
      \[i:\CCtilde^1_K(U)\to \BBdRp(U)\]
      such that $\BBdRp(U)$ is the $\frac{\xi_K}{\pi}$-adic completion of $\CCtilde^1_K(U)$. Recall that 
      \[\BBdRp(U) = \varprojlim_n\bA_{\inf,K}(U)[\frac{1}{p}]/\xi_K^n\]
      is the $\xi_K$-adic completion of $\bA_{\inf,K}(U)[\frac{1}{p}]$ such that $\pi$ and $[\underline \pi]$ are invertible in $\BBdRp(U)$.
      We claim that for any $n\geq 1$, the map
      \[\bA_{\inf,K}(U)[X^{\pm 1}]\to \bA_{\inf,K}(U)[\frac{1}{p}]/\xi_K^n\]
      sending $X$ to $\frac{[\underline \pi]}{\pi}$ is bounded in the sense that it takes image in 
      \[\pi^{1-n}(\bA_{\inf,K}(U)/\xi_K^n)\subset(\bA_{\inf,K}(U)/\xi_K^n)[\frac{1}{p}].\]
      Note that $(X-1)^n$ is mapped to $\pi^{-n}([\underline \pi]-\pi)^n$ which is vanishing in $(\bA_{\inf,K}(U)/\xi_K^n)[\frac{1}{p}]$. For any $m\in \bZ$, the above morphism sends $\pi^{n-1}X^m$ to 
      \[\pi^{n-1}(\frac{[\underline \pi]}{\pi})^m = \pi^{n-1}\sum_{i\geq 0}\binom{m}{i}(\frac{[\underline \pi]}{\pi}-1)^i = \pi^{n-1}\sum_{i=0}^{n-1}\binom{m}{i}(\frac{[\underline \pi]}{\pi}-1)^i = \sum_{i=0}^{n-1}\binom{m}{i}([\underline \pi]-\pi)^i,\]
      which is an element in $\bA_{\inf,K}(U)/\xi_K^n$. Thus, for any $m\in\bZ$, the above morphism sends $X^m$ to $\pi^{1-n}(\bA_{\inf,K}(U)/\xi_K^n)$, yielding the above claim as desired. 

      Due to the claim above, the morphism $\bA_{\inf,K}(U)[X^{\pm 1}]\to \bA_{\inf,K}(U)[\frac{1}{p}]/\xi_K^n$ uniquely extends to a morphism 
      \[\bA_{\inf,K}(U)\za X^{\pm 1}\ya\to \bA_{\inf,K}(U)[\frac{1}{p}]/\xi_K^n\]
      and thus induces a morphism
      \[\CCtilde^{1,+}_K(U) = \bA_{\inf,K}(U)\za X^{\pm 1}\ya/(\pi X-[\underline \pi])\to \bA_{\inf,K}(U)[\frac{1}{p}]/\xi_K^n.\]
      Inverting $p$ and letting $n$ go to $\infty$, we get the desired morphism
      \[i:\CCtilde^1_K(U)\to\BBdRp(U).\]

      As $\BBdRp(U)$ is $\xi_K$-adic complete, the morphism $i$ uniquely induces morphism
      \[\widehat i:\CCtilde^{1}_K(U)^{\wedge_{\xi_K}}\to \BBdRp(U),\]
      where $\CCtilde^{1}_K(U)^{\wedge_{\xi_K}}$ denotes the $\xi_K$-adic completion of $\CCtilde^{1}_K(U)$ (which is exactly the $\frac{\xi_K}{\pi}$-adic completion of $\CCtilde^{1}_K(U)$ as $\pi$ is invertible in $\CCtilde^{1}_K(U)$). On the other hand, considering the $\xi_K$-adic completion along the natural morphism $\bA_{\inf,K}(U)[\frac{1}{p}]\to \CCtilde^{1}_K(U)$,
      we get a morphism
      \[j:\BBdRp(U)\to \CCtilde^{1}_K(U)^{\wedge_{\xi_K}}.\]
      Then we can conclude that the map $\widehat i$ gives the desired isomorphism $\CCtilde^{1}_K(U)^{\wedge_{\xi_K}}\cong \BBdRp(U)$ by checking that $j$ is the inverse of $\widehat i$ directly.
  \end{proof}

  Now, we are back to the construction of the desired period sheaf with connection.
  For any $?\in\{\emptyset,+\}$, consider the $(p,\xi_K)$-adic complete tensor product
  \[C_{\frakY_Z}^{?}:=\Gamma_{\frakY}\widehat \otimes_{\bA_{\inf,K}(U)}\CCtilde_K^{1,?}(U),\]
  and then we have $C_{\frakY_Z} = C_{\frakY_Z}^+[\frac{1}{p}]$. Define 
  \[B_{\dR,\pd,\frakY_Z}^+:=\xi_K\text{-adic completion of } C_{\frakY_Z}.\]
  Then $B^+_{\dR,\pd,\frakY_Z}$ is a $\BBdRp(U)$-algebra by Lemma \ref{lem:relative Robba ring}(3).

  Let $\rd:\wtr_Z\to\Omega^{1,\log}_{\wtr_Z}$ be the usual $(p,\xi_K)$-complete derivation on $\wtr_Z$, where $\Omega^1_{\wtr_Z}$ denotes the module of $(p,\xi_K)$-adically continuous differentials over $\wtr_Z$ over $\bA_{\inf,K}(Z)$. It extends uniquely to an $\bA_{\inf,K}(U)$-linear derivation 
  \[\rd:\wtr_Z\widehat \otimes_{\bA_{\inf,K}(Z)}\bA_{\inf,K}(U)\to\Omega^{1,\log}_{\wtr_Z}\widehat \otimes_{\bA_{\inf,K}(Z)}\bA_{\inf,K}(U)\]
  such that for any $n\geq 1$,
  \[\rd(\Ker(\theta^{\prime}_{\frakY_Z})^n)\subset\Ker(\theta^{\prime}_{\frakY_Z})^{n-1}\otimes_{\wtr_Z}\Omega^{1,\log}_{\wtr_Z}.\]
  Thus, it extends uniquely to an $\bA_{\inf,K}(U)$-linear derivation 
  \[\rd:A_{\frakY_Z}\to A_{\frakY_Z}\otimes_{\wtr_Z}\Omega^{1,\log}_{\wtr_Z},\]
  and then an $\bA_{\inf,K}(U)$-linear derivation
  \[\rd:B_{\frakY_Z}\to B_{\frakY_Z}\otimes_{\wtr_Z}\Omega^{1,\log}_{\wtr_Z}\{-1\}\]
  where $\Omega^{1,\log}_{\wtr_Z}\{-1\} = \Omega^{1,\log}_{\wtr_Z}\cdot\xi_K^{-1}$. This differential also extends uniquely to an $\bA_{\inf,K}(U)$-linear derivation
  \[\rd:\Gamma_{\frakY_Z}\to u\Gamma_{\frakY_Z}\otimes_{\wtr_Z}\Omega^1_{\wtr_Z}\{-1\}.\]
  Considering its base-change along $\bA_{\inf,K}(U)\to \CCtilde_K^{1,?}(U)$ (and taking $\xi_K$-adic completion), we get a $\CCtilde_K^{1,?}(U)$-linear (resp. $\BBdRp(U)$-linear) derivation
  \[\begin{split}\rd:C_{\frakY_Z}^?\to uC_{\frakY_Z}^?\otimes_{\wtr_Z}\Omega^{1,\log}_{\wtr_Z}\{-1\}
  \text{ (resp.  $\rd:B_{\dR,\pd,\frakY_Z}^+\to B_{\dR,\pd,\frakY_Z}^+\otimes_{\wtr_Z}\Omega^{1,\log}_{\wtr_Z}\{-1\}$)}.\end{split}\]
  Recall that $u$ is invertible in $\CCtilde^1_K(U)$ and $\BBdRp(U)$.

  \begin{dfn}\label{dfn:Period sheaf-local}
      Suppose that $\frakX_Z = \Spf(\calR_Z)$ is small semi-stable of realtive dimension $d$ over $A^+$.
      Let $\calO\bB_{\dR,\pd,Z}^+$ be the sheaf corresponding to the sheafification of the presheaf sending each affinoid perfectoid $U\in X_{Z,v}$ to 
      \[\colim_{\frakY_Z\in \Sigma_U}B_{\dR,\pd,\frakY_Z}^+,\]
      and let
      \[\rd:\calO\bB_{\dR,\pd,Z}^+\to\calO\bB_{\dR,\pd,Z}^+\otimes_{\calO_{\wtx_Z}}\Omega^1_{\wtx_Z}\{-1\}\]
      be the derivation corresponding to the colimit of 
      \[\rd:B_{\dR,\pd,\frakY_Z}^+\to B_{\dR,\pd,\frakY_Z}^+\otimes_{\wtr_Z}\Omega^{1,\log}_{\wtr_Z}\{-1\}\]
      constructed above. Denote by $\rD\rR(\calO\bB_{\dR,\pd,Z}^+,\rd)$ the corresponding de Rham complex
      \[\calO\bB_{\dR,\pd,Z}^+\xrightarrow{\rd}\calO\bB_{\dR,\pd,Z}^+\otimes_{\calO_{\wtx_Z}}\Omega^{1,\log}_{\wtx_Z}\{-1\}\xrightarrow{\rd}\cdots \xrightarrow{\rd}\calO\bB_{\dR,\pd,Z}^+\otimes_{\calO_{\wtx_Z}}\Omega^{d,\log}_{\wtx_Z}\{-d\}.\]
      Note that the construction of $(\calO\bB_{\dR,\pd,Z}^+,\rd)$ is independent of the choice of chart on $\frakX_Z$.
  \end{dfn}
  \begin{rmk}
      It will be seen that for any $\frakY_Z\in \Sigma_U$, we have an isomorphism
      \[B_{\dR,\pd,\frakY_Z}^+\cong \bB_{\dR}^+(U)\za W_0,\dots,W_d\ya_{\pd}/(W_0+\cdots+W_d)_{\pd}^{\wedge},\]
      where $ \bB_{\dR}^+(U)\za W_0,\dots,W_d\ya_{\pd}$ is a certain completion of the free pd-polynomial ring $\bB_{\dR}^+(U)[W_0,\dots,W_d]_{\pd}$ and $(W_0+\cdots+W_d)_{\pd}^{\wedge}$ denotes the complete pd-ideal generated by $W_0+\cdots+W_d$. It follows that, in Definition \ref{dfn:Period sheaf-local}, we have
      \[B_{\dR,\pd,\frakX_Z}^+ = \colim_{\frakY_Z\in \Sigma_U}B_{\dR,\pd,\frakY_Z}^+\cong\bB_{\dR}^+(U)\za W_0,\dots,W_d\ya_{\pd}/(W_0+\cdots+W_d)_{\pd}^{\wedge}.\]
      As a consequence, the presheaf 
      \[U\in X_{Z,v}\mapsto \colim_{\frakY_Z\in \Sigma_U}B_{\dR,\pd,\frakY_Z}^+\]
      is already a sheaf because the presheaf
      \[U\in X_{Z,v}\mapsto \bB_{\dR}^+(U)\za W_0,\dots,W_d\ya_{\pd}/(W_0+\cdots+W_d)_{\pd}^{\wedge}\]
      is so. For details, see Proposition \ref{prop:local description of OB_dR^+} and the proof therein.
  \end{rmk}

  Now, we are going to give an explicit description of $(\calO\bB_{\dR,\pd,Z}^+,\rd)$.

    \begin{lem}\label{lem:local description of A_Y}
      Suppose that $\frakX_Z = \Spf(\calR_Z)$ is small semi-stable with generic fiber $X_Z$ as in Definition \ref{dfn:small affine}.
      For any affinoid perfectoid $U = \Spa(S,S^+)\in X_{Z,v}$ and any $\frakY_Z = \Spf(\calR^{\prime}_Z) \in \Sigma_U$, the map
      \[\iota:\bA_{\inf,K}(U)[[V_0,V_1,\dots,V_d]]/(\prod_{i=0}^r(1+V_i)-1)\to A_{\frakY_Z}\]
      sending each $V_i$ to $\frac{T_i}{t_i}-1$ is a well-defined isomorphism of $\bA_{\inf,K}(U)$-algebras. Here, $t_i$'s are the elements in Lemma \ref{lem:kernel is finite} and $\bA_{\inf,K}(U)[[V_1,\dots,V_d]]$ is the ring of formal power series over $\bA_{\inf,K}(U)$ freely generated by $V_1,\dots,V_d$. Moreover, the isomorphism carries the ideal $(\xi_K,V_1,\dots,V_d)$ onto the ideal $\Ker(\theta^{\prime}_{\frakY_Z})$ of $A_{\frakY_Z}$.
  \end{lem}
  \begin{proof}
      Without loss of generality, we may assume $\frakY_Z = \frakX_Z = \Spf(\calR_Z)$. To see $\iota$ is well-defined, we first consider the map
      \[i:\bA_{\inf,K}(U)[V_0,\dots,V_d]\to A_{\frakY}\]
      sending each $V_i$ to $\frac{T_i}{t_i}-1$. As $i(V_j)\in \Ker(\theta^{\prime}_{\frakY_Z})$ and $A_{\frakY_Z}$ is $\Ker(\theta_{\frakY_Z})$-adic complete, the above morphism $i$ uniquely extends to a map of $\bA_{\inf,K}(U)$-algebras (which is still denoted by)
      \[\bA_{\inf,K}(U)[[V_0,\dots,V_d]]\to A_{\frakY_Z}.\]
      As $\prod_{i=0}^r\frac{T_i}{t_i} = 1$ in $A_{\frakY_Z}$, this map uniquely extends to a map
      \[\bA_{\inf,K}(U)[[V_0,\dots,V_d]]/(\prod_{i=0}^r(1+V_i)-1)\to A_{\frakY_Z}\]
      which is nothing but $\iota$. To see $\iota$ is an isomorphism, it remains to construct a map
      \[j:A_{\frakY_Z}\to\bA_{\inf,K}(U)[[V_0,\dots,V_d]]/(\prod_{i=0}^r(1+V_i)-1)\]
      which is the inverse of $\iota$.

      Consider the morphism of $\bA_{\inf,K}(Z)$-algebras
      \[j_1:\bA_{\inf,K}(Z)\za T_0,\dots,T_r,T_{r+1}^{\pm 1},\dots,T_d^{\pm 1}\ya/(T_0\cdots T_r-[\varpi^a])\to\bA_{\inf,K}(U)[[V_0,\dots,V_d]]/(\prod_{i=0}^r(1+V_i)-1)\]
      sending each $T_i$ to $t_i(1+V_i)$. Clearly, this map is well-defined such that the composite 
      \[\begin{split}\bA_{\inf,K}(Z)\za T_0,\dots,T_r,T_{r+1}^{\pm 1},\dots,T_d^{\pm 1}\ya/(T_0\cdots T_r-[\varpi^a])&\xrightarrow{j_1}\bA_{\inf,K}(U)[[V_0,\dots,V_d]]/(\prod_{i=0}^r(1+V_i)-1)\\
      &\xrightarrow{\mod (\xi_K,V_0,\dots,V_d)}S^+\end{split}\]
      coincides with the morphism 
      \[\bA_{\inf,K}(Z)\za T_0,\dots,T_r,T_{r+1}^{\pm 1},\dots,T_d^{\pm 1}\ya/(T_0\cdots T_r-[\varpi^a])\xrightarrow{\psi}\calR_Z\to S^+.\]

       We claim that the morphism $j_1$ uniquely lifts to a morphism of $\bA_{\inf,K}(Z)$-algebras
      \[\wtr_Z\to\bA_{\inf,K}(U)[[V_0,\dots,V_d]]/(\prod_{i=0}^r(1+V_i)-1).\]
      Put
      \[\widetilde \calR_Z^{\Box}:=\bA_{\inf,K}(Z)\za T_0,\dots,T_r,T_{r+1}^{\pm 1},\dots,T_d^{\pm 1}\ya/(T_0\cdots T_r-[\varpi^a])\] 
      for short.    
Define
\[\widetilde\psi_{m,n}:=\widetilde \psi/(p^m,\xi_K^n):\wtr_Z^{\Box}/(p^m,\xi_K^n)\to \wtr_Z/(p^m,\xi_K^n),\]
which is \'etale for any $m,n\geq 1$ \eqref{equ:etale lifting chart}. Note that we have the following commutative diagram when $n=1$,

\begin{tikzcd}
{\wtr_Z^{\Box}/(p^m,\xi_K^{n+1})} \arrow[d, "{\widetilde \psi_{m,n+1}}"', shift right] \arrow[rr] &                                                   & {\bA_{\inf,K}(U)[[V_0,\dots,V_d]]/(\prod_{i=0}^r(1+V_i)-1,p^m,(\xi_K,V_0,\dots,V_d)^{n+1})} \arrow[d] \\
{\wtr_Z/(p^m,\xi_K^{n+1})} \arrow[r] \arrow[rru, "{\exists ! ~\iota_{m,n+1}}", dotted]            & {\wtr_Z/(p^m,\xi_K^n)} \arrow[r, "{\iota_{m,n}}"] & {\bA_{\inf,K}(U)[[V_0,\dots,V_d]]/(\prod_{i=0}^r(1+V_i)-1,p^m,(\xi_K,V_0,\dots,V_d)^n)},              
\end{tikzcd}
where $\iota_{m,1}$ coincides with the natural map
\[\wtr_Z/(p^m,\xi_K) = \calR_Z/p^m\to S^+/p^m = \bA_{\inf,K}(U)[[V_0,\dots,V_d]]/(\prod_{i=0}^r(1+V_i)-1,p^m,\xi_K,V_0,\dots,V_d).\]
By \'etaleness of $\widetilde \psi_{m,2}$, there exists a unique map $\iota_{m,2}$ representing the dashed arrow, which makes the whole diagram commute (when $n=1$). So, the above diagram exists and commutes when $n=2$ such that one can similarly get a unique morphism $\iota_{m,3}$ representing the dashed arrow, which makes the whole diagram commute (when $n=2$). By iteration, there exists a unique sequence $\{\iota_{m,n}\}_{n\geq 1}$ of maps such that the above diagram exists and commutes for any $m,n\geq 1$. 
Then $\varprojlim \iota_{m,n}$ gives the desired morphism $\wtr_Z\to\bA_{\inf,K}(U)[[V_0,\dots,V_d]]/(\prod_{i=0}^r(1+V_i)-1)$. This completes the proof of claim.

      Return to the construction of $j$. By scalar extension along $\bA_{\inf,K}(Z)\to\bA_{\inf,K}(U)$, $j_1$ uniquely extends to a morphism of $\bA_{\inf,K}(U)$-algebras
      \[j_2:\wtr_Z\widehat \otimes_{\bA_{\inf,K}(Z)}\bA_{\inf,K}(U)\to\bA_{\inf,K}(U)[[V_0,\dots,V_d]]/(\prod_{i=0}^r(1+V_i)-1)\]
      such that the composite
      \[\begin{split}\wtr_Z\widehat \otimes_{\bA_{\inf,K}(Z)}\bA_{\inf,K}(U)\xrightarrow{j_2}\bA_{\inf,K}(U)[[V_0,\dots,V_d]]/(\prod_{i=0}^r(1+V_i)-1)\xrightarrow{\mod(\xi_K,V_0,\dots,V_d)} S^+\end{split}\]
      coincides with $\theta_{\frakY_Z}$.
      By construction, $j_2(T_i) = t_i(1+V_i)$ is a unit-multiple of $t_i$. Therefore, $j_2$ uniquely extends to a morphism
      \[j_3:\wtr_Z\widehat \otimes_{\bA_{\inf,K}(Z)}\bA_{\inf,K}(U)[(\frac{T_0}{t_0})^{\pm 1},\dots,(\frac{T_r}{t_r})^{\pm 1}]^{\wedge_{(p,\xi_K)}}\to\bA_{\inf,K}(U)[[V_0,\dots,V_d]]/(\prod_{i=0}^r(1+V_i)-1)\]
      sending each $\frac{T_i}{t_i}$ to $1+V_i$ such that the composite
      \[\begin{split}\wtr_Z\widehat \otimes_{\bA_{\inf,K}(Z)}\bA_{\inf,K}(U)[(\frac{T_0}{t_0})^{\pm 1},\dots,(\frac{T_r}{t_r})^{\pm 1}]^{\wedge_{(p,\xi_K)}}&\xrightarrow{j_3}\bA_{\inf,K}(U)[[V_0,\dots,V_d]]/(\prod_{i=0}^r(1+V_i)-1)\\
      &\xrightarrow{\mod(\xi_K,V_0,\dots,V_d)} S^+\end{split}\]
      coincides with $\theta^{\prime}_{\frakY_Z}$. 
      In particular, $j_3$ carries $\Ker(\theta^{\prime}_{\frakY_Z})$ into the ideal $(\xi_K,V_1,\dots,V_d)$. Note that we have an isomorphism
      \[\bA_{\inf,K}(U)[[V_0,\dots,V_d]]/(\prod_{i=0}^r(1+V_i)-1)\cong \bA_{\inf,K}(U)[[V_1,\dots,V_d]]\] 
      because
      \[\prod_{i=0}^r(1+V_i)-1\equiv V_0+\cdots+V_r \mod (V_0,\cdots,V_r)^2.\]
      So $\bA_{\inf,K}(U)[[V_0,\dots,V_d]]/(\prod_{i=0}^r(1+V_i)-1)$ is $(p,\xi_K,V_0,\dots,V_d)$-adically complete, and thus $j_3$ uniquely extends to a morphism of $\bA_{\inf,K}(U)$-algebras
      \[j:A_{\frakY_Z}\to \bA_{\inf,K}(U)[[V_0,\dots,V_d]]/(\prod_{i=0}^r(1+V_i)-1).\]
      To complete the proof, we only need to check that $j$ is the inverse of $\iota$. Granting this, it follows from the above construction that
      \[\iota((\xi_K,V_1,\dots,V_d))\subset \Ker(\theta^{\prime}_{\frakY_K})\text{ and }j(\Ker(\theta^{\prime}_{\frakY_K}))\subset(\xi_K,V_1,\dots,V_d),\]
      yielding the ``moreover'' part of the result.

      It is clear from the construction that $j\circ\iota = \id$. We have to show $\iota\circ j = \id$. Note that by construction, the composite
      \[\bA_{\inf,K}(Z)\za T_0,\dots,T_r,T_{r+1}^{\pm 1},\dots,T_d^{\pm 1}\ya/(T_0\cdots T_r-[\varpi^a])\xrightarrow{i\circ j_1} A_{\frakY_Z}\xrightarrow{\mod \Ker(\theta^{\prime}_{\frakY_Z})}S^+\]
      coincides with the natural map 
      \[\bA_{\inf,K}(Z)\za T_0,\dots,T_r,T_{r+1}^{\pm 1},\dots,T_d^{\pm 1}\ya/(T_0\cdots T_r-[\varpi^a])\to\calR_Z\to S^+.\]
      Therefore, by \'etaleness of $\widetilde \psi$ \eqref{equ:etale lifting chart}, a similar argument above shows that the composite
      \[i\circ j_1:\bA_{\inf,K}(Z)\za T_0,\dots,T_r,T_{r+1}^{\pm 1},\dots,T_d^{\pm 1}\ya/(T_0\cdots T_r-[\varpi^a])\to A_{\frakY_Z}\] uniquely extends to a morphism of $\bA_{\inf,K}(Z)$-algebras 
      \[\wtr_Z\to A_{\frakY_Z}.\]
      However, as $i\circ j_1$ coincides with the natural map 
      \[\bA_{\inf,K}(Z)\za T_0,\dots,T_r,T_{r+1}^{\pm 1},\dots,T_d^{\pm 1}\ya/(T_0\cdots T_r-[\varpi^a])\to\wtr_Z\to\wtr_Z\widehat \otimes_{\bA_{\inf,K}(Z)}\bA_{\inf,K}(U)\to A_{\frakY_Z},\]
      by the uniqueness, the above map $\wtr_Z\to A_{\frakY_Z}$ is also the natural map 
      \[h:\wtr_Z\to\wtr_Z\widehat \otimes_{\bA_{\inf,K}(Z)}\bA_{\inf,K}(U)\to \wtr_Z\widehat \otimes_{\bA_{\inf,K}(Z)}\bA_{\inf,K}(U)[(\frac{T_0}{t_0})^{\pm 1},\dots,(\frac{T_r}{t_r})^{\pm 1}]^{\wedge_{(p,\xi_K)}} \to A_{\frakY_Z}.\]
      Using the same argument for constructing $j$, we see that $h$ uniquely extends to a morphism of $\bA_{\inf,K}(U)$-algebras $A_{\frakY_Z}\to A_{\frakY_Z}$, which is nothing but $\iota\circ j$ by construction. Then the uniqueness criterion forces $\iota\circ j = \id$ as desired because the identity $\id$ is another extension of the morphism $\wtr_Z\to A_{\frakY_Z}$ that we obtained above. This completes the proof.
  \end{proof}
  The following remark has appeared in the proof of Lemma \ref{lem:local description of A_Y}.
  \begin{rmk}\label{rmk:pre-normalization}
      As $\prod_{i=0}^r(1+V_i)-1\equiv V_0+\cdots+V_r\mod (V_0,\dots,V_r)^2$, the natural morphism of $\bA_{\inf,K}(U)$-algebras
      \[\bA_{\inf,K}(U)[[V_1,\dots,V_d]]\to\bA_{\inf,K}(U)[[V_0,\dots,V_d]]/(\prod_{i=0}^r(1+V_i)-1)\]
      sending each $V_i$ to $V_i$ is an isomorphism.
  \end{rmk}
  \begin{cor}\label{cor:local description of A_Y}
      Keep notations in Lemma \ref{lem:local description of A_Y}. Then the map
      \[\iota:\left(\bA_{\inf,K}(U)[u,Z_0,\dots,Z_d]_{\pd}^{\wedge}/(\frac{u}{\xi_K}(\prod_{i=0}^r(1+\frac{\xi_K}{u}Z_i)-1))_{\pd}^{\wedge}\right)\widehat \otimes_{\bA_{\inf,K}(U)}\CCtilde_K^{1,+}(U)\to C_{\frakY_Z}^+\]
      sending each $Z_i$ to $\frac{u}{\xi_K}(\frac{T_i}{t_i}-1)$ is a well-defined isomorphism, where $\bA_{\inf,K}(U)[u,Z_0,\dots,Z_d]_{\pd}^{\wedge}$ is the $(p,\xi_K,\frac{\xi_K}{u}Z_0,\dots,\frac{\xi_K}{u}Z_d)$-adic completion of 
      \[\bA_{\inf,K}(U)[[\frac{\xi_K}{u}Z_0,\dots,\frac{\xi_K}{u}Z_d]][u,Z_0,\dots,Z_d]_{\pd}\]
      while $(\frac{u}{\xi_K}(\prod_{i=0}^r(1+\frac{\xi_K}{u}Z_i)-1))_{\pd}^{\wedge}$ denotes the closed pd-ideal generated by $\frac{u}{\xi_K}(\prod_{i=0}^r(1+\frac{\xi_K}{u}Z_i)-1)$.
  \end{cor}
  \begin{proof}
      Let $V_i$ be as in Lemma \ref{lem:local description of A_Y}. By construction of $C_{\frakY_Z}^+$, sending each $V_i$ to $\frac{T_i}{t_i}-1$ gives rise to an isomorphism of $\CCtilde_K^{1,+}(U)$-algebras
      \[\left(\bA_{\inf,K}(U)[[V_0,\dots,V_d]]/(\prod_{i=0}^r(1+V_i)-1)\right)[u,\frac{uV_0}{\xi_K},\dots,\frac{uV_d}{\xi_K}]^{\wedge}_{\pd}\widehat \otimes_{\bA_{\inf,K}(U)}\CCtilde_K^{1,+}(U)\xrightarrow{V_i\mapsto (\frac{T_i}{t_i}-1)}C_{\frakY_Z}^+,\]
      where $\big(\bA_{\inf,K}(U)[[V_0,\dots,V_d]]/(\prod_{i=0}^r(1+V_i)-1)\big)[u,\frac{uV_0}{\xi_K},\dots,\frac{uV_d}{_K}]^{\wedge}_{\pd}$ is the $(p,\xi_K,V_0,\dots,V_d)$-adic completion of the pd-algebra 
      \[\big(\bA_{\inf,K}(U)[[V_0,\dots,V_d]]/(\prod_{i=0}^r(1+V_i)-1)\big)[u,\frac{uV_0}{\xi_K},\dots,\frac{uV_d}{\xi_K}]_{\pd},\]
      which is the pd-envelope with respect to the ideal $(u,\frac{uV_0}{\xi_K},\dots,\frac{uV_d}{\xi_K})$ of the blow-up algebra 
      \[\big(\bA_{\inf,K}(U)[[V_0,\dots,V_d]]/(\prod_{i=0}^r(1+V_i)-1)\big)[\frac{uV_0}{\xi_K},\dots,\frac{uV_d}{\xi_K}].\]
      In particular, for any $n\geq 1$, the pd-power
      \[(\frac{u}{\xi_K}(\prod_{i=0}^r(1+V_i)-1))^{[n]}\]
      vanishes in $\big(\bA_{\inf,K}(U)[[V_0,\dots,V_d]]/(\prod_{i=0}^r(1+V_i)-1)\big)[u,\frac{uV_0}{\xi_K},\dots,\frac{uV_d}{\xi_K}]^{\wedge}_{\pd}$.
      
      In order to show that $\iota$ is an isomorphism, it suffices to show that the map
      \[\begin{split}i\colon&\bA_{\inf,K}(U)[[\frac{\xi_K}{u}Z_0,\dots,\frac{\xi_K}{u}Z_d]][u,Z_0,\dots,Z_d]_{\pd}^{\wedge}/(\frac{u}{\xi_K}(\prod_{i=0}^r(1+\frac{\xi_K}{u}Z_i)-1))^{\wedge}_{\pd}\\
      &\to\left(\bA_{\inf,K}(U)[[V_0,\dots,V_d]]/(\prod_{i=0}^r(1+V_i)-1)\right)[u,\frac{uV_0}{\xi_K},\dots,\frac{uV_d}{\xi_K}]^{\wedge}_{\pd}\end{split}\]
      sending each $Z_i$ to $\frac{uV_i}{\xi_K}$ is a well-defined isomorphism.

      The well-definedness of $i$ follows by unwinding the definitions: Consider the morphism 
      \[\begin{split}i_1\colon&\bA_{\inf,K}(U)[[\frac{\xi_K}{u}Z_0,\dots,\frac{\xi_K}{u}Z_d]][Z_0,\dots,Z_d]\\
      &\to \left(\bA_{\inf,K}(U)[[V_0,\dots,V_d]]/(\prod_{i=0}^r(1+V_i)-1)\right)[u,\frac{uV_0}{\xi_K},\dots,\frac{uV_d}{\xi_K}]^{\wedge}_{\pd}\end{split}\]
      mapping each $Z_i$ to $\frac{u}{\xi_K}V_i$. As $u,\frac{u}{\xi_K}V_0,\dots,\frac{u}{\xi_K}V_d$ admits arbitrary pd-powers in the target, $i_1$ uniquely extends to a morphism
      \[\begin{split}i_2\colon &\bA_{\inf,K}(U)[[\frac{\xi_K}{u}Z_0,\dots,\frac{\xi_K}{u}Z_d]][u,Z_0,\dots,Z_d]_{\pd}\\
      &\to \left(\bA_{\inf,K}(U)[[V_0,\dots,V_d]]/(\prod_{i=0}^r(1+V_i)-1)\right)[u,\frac{uV_0}{\xi_K},\dots,\frac{uV_d}{\xi_K}]^{\wedge}_{\pd}.\end{split}\]
      As its target is $(p,\xi_K,V_0,\dots,V_d)$-adically complete, $i_2$ uniquely extends to a morphism
      \[\begin{split}i_3\colon&\bA_{\inf,K}(U)[[\frac{\xi_K}{u}Z_0,\dots,\frac{\xi_K}{u}Z_d]][u,Z_0,\dots,Z_d]^{\wedge}_{\pd}\\
      &\to \left(\bA_{\inf,K}(U)[[V_0,\dots,V_d]]/(\prod_{i=0}^r(1+V_i)-1)\right)[u,\frac{uV_0}{\xi_K},\dots,\frac{uV_d}{\xi_K}]^{\wedge}_{\pd}\end{split}\]
      sending $\frac{u}{\xi_K}(\prod_{i=0}^r(1+\frac{\xi_K}{u}Z_i)-1)$ to $\frac{u}{\xi_K}(\prod_{i=0}^r(1+V_i)-1)$, which vanishes in the target. So the above map uniquely extends to a morphism
      \[\begin{split}&\bA_{\inf,K}(U)[[\frac{\xi_K}{u}Z_0,\dots,\frac{\xi_K}{u}Z_d]][u,Z_0,\dots,Z_d]^{\wedge}_{\pd}/(\frac{u}{\xi_K}(\prod_{i=0}^r(1+\frac{\xi_K}{u}Z_i)-1))_{\pd}\\
      \to &\left(\bA_{\inf,K}(U)[[V_0,\dots,V_d]]/(\prod_{i=0}^r(1+V_i)-1)\right)[u,\frac{uV_0}{\xi_K},\dots,\frac{uV_d}{\xi_K}]^{\wedge}_{\pd}.\end{split}\]
      Using the $(p,\xi_K,V_0,\dots,V_d)$-adic completeness of the target again, we get the well-definedness of the morphism $i$ as desired.

      It remains to show that $i$ is an isomorphism. We shall do this by constructing its inverse as follows: Consider the map
      \[j_1:\bA_{\inf,K}(U)[V_0\dots,V_d]\to \bA_{\inf,K}(U)[[\frac{\xi_K}{u}Z_0,\dots,\frac{\xi_K}{u}Z_d]][u,Z_0,\dots,Z_d]^{\wedge}_{\pd}/(\frac{u}{\xi_K}(\prod_{i=0}^r(1+\frac{\xi_K}{u}Z_i)-1))_{\pd}^{\wedge}\]
      sending each $V_i$ to $\frac{\xi_K}{u}Z_i$. As $\frac{\xi_K}{u}Z_i$ is topologically nilpotent in the target of $j_1$ (because $n!$ divides $Z_i^n$ in the target), the above map uniquely extends to a morphism 
      \[j_2:\bA_{\inf,K}(U)[[V_0\dots,V_d]]\to \bA_{\inf,K}(U)[[\frac{\xi_K}{u}Z_0,\dots,\frac{\xi_K}{u}Z_d]][u,Z_0,\dots,Z_d]^{\wedge}_{\pd}/(\frac{u}{\xi_K}(\prod_{i=0}^r(1+\frac{\xi_K}{u}Z_i)-1))_{\pd}^{\wedge}\]
      such that $j_2(\prod_{i=0}^r(1+V_i)-1) = \prod_{i=0}^r(1+\frac{\xi_K}{u}Z_i)-1$. Thus, $j_2$ uniquely extends to a morphism 
      \[\begin{split}j_3\colon&\bA_{\inf,K}(U)[[V_0\dots,V_d]]/(\prod_{i=0}^r(1+V_i)-1)\\
      &\to \bA_{\inf,K}(U)[[\frac{\xi_K}{u}Z_0,\dots,\frac{\xi_K}{u}Z_d]][u,Z_0,\dots,Z_d]^{\wedge}_{\pd}/(\frac{u}{\xi_K}(\prod_{i=0}^r(1+\frac{\xi_K}{u}Z_i)-1))_{\pd}^{\wedge}.\end{split}\]
      As $j_3(V_i) = \frac{\xi_K}{u}Z_i$ is divided by $\frac{\xi_K}{u}$, the map $j_3$ uniquely extends to a morphism
      \[\begin{split}j_4\colon&\left(\bA_{\inf,K}(U)[[V_0\dots,V_d]]/(\prod_{i=0}^r(1+V_i)-1)\right)[\frac{uV_0}{\xi_K},\dots,\frac{uV_d}{\xi_K}]\\
      &\to \bA_{\inf,K}(U)[[\frac{\xi_K}{u}Z_0,\dots,\frac{\xi_K}{u}Z_d]][u,Z_0,\dots,Z_d]^{\wedge}_{\pd}/(\frac{u}{\xi_K}(\prod_{i=0}^r(1+\frac{\xi_K}{u}Z_i)-1))_{\pd}^{\wedge}.\end{split}\]
      Finally, as argued in the last paragraph, we see that $j_4$ uniquely extends to a morphism
      \[\begin{split}j\colon&\left(\bA_{\inf,K}(U)[[V_0\dots,V_d]]/(\prod_{i=0}^r(1+V_i)-1)\right)[u,\frac{uV_0}{\xi_K},\dots,\frac{uV_d}{\xi_K}]^{\wedge}_{\pd}\\
      &\to \bA_{\inf,K}(U)[[\frac{\xi_K}{u}Z_0,\dots,\frac{\xi_K}{u}Z_d]][u,Z_0,\dots,Z_d]^{\wedge}_{\pd}/(\frac{u}{\xi_K}(\prod_{i=0}^r(1+\frac{\xi_K}{u}Z_i)-1))_{\pd}^{\wedge}.\end{split}\]
      Now, one can conclude by checking that $j$ is the inverse of $i$ directly.
  \end{proof}

  Before moving on, we give a general lemma.
  \begin{lem}\label{lem:defined d on quotient}
    Let $B$ be a $p$-complete $p$-separated ring with a topologically nilpotent element $b\in B$. Let $B[U_0,\dots,U_d]^{\wedge_p}_{\pd}$ be the $p$-adic completion of the pd-algebra over $B$ freely generated by $U_i$'s. Fix an element $u\in B$.

    \begin{enumerate}
        \item[(1)] Put $Q:=b^{-1}(\prod_{i=0}^r(1+bU_i)-1)\in B[U_0,\dots,U_d]^{\wedge_p}_{\pd}$. Then the differential operator 
        \[\rd:=\sum_{i=0}^du(1+bU_i)\frac{\partial}{\partial U_i}\otimes\rd U_i:B[U_0,\dots,U_d]^{\wedge_p}_{\pd}\to\oplus_{i=0}^dB[U_0,\dots,U_d]^{\wedge_p}_{\pd}\cdot\rd U_i\]
        uniquely induces a differential operator (which is still denoted by)
        \[\begin{split}\rd=\sum_{i=0}^du(1+bU_i)\frac{\partial}{\partial U_i}\otimes\rd U_i: D
        \to(\oplus_{i=0}^dD\cdot\rd U_i)/D\cdot (\rd U_0+\cdots+\rd U_r)\end{split}\]
        where $D =B[U_0,\dots,U_d]^{\wedge_p}_{\pd}/(Q)_{\pd}^{\wedge_p}$ is the quotient of $B[U_0,\dots,U_d]^{\wedge_p}_{\pd}$ by the closed pd-ideal generated by $Q$.

        \item[(2)] Put $Q = U_0+\cdots+U_r\in B[U_0,\dots,U_d]^{\wedge_p}_{\pd}$. Then the differential operator 
        \[\rd:=-\sum_{i=0}^du\frac{\partial}{\partial U_i}\otimes\rd U_i:B[U_0,\dots,U_d]^{\wedge_p}_{\pd}\to\oplus_{i=0}^dB[U_0,\dots,U_d]^{\wedge_p}_{\pd}\cdot\rd U_i\]
        uniquely induces a differential operator (which is still denoted by)
        \[\begin{split}\rd=-\sum_{i=0}^du\frac{\partial}{\partial U_i}\otimes\rd U_i: D
        \to(\oplus_{i=0}^dD\cdot\rd U_i)/D\cdot (\rd U_0+\cdots+\rd U_r)\end{split}\]
        where $D =B[U_0,\dots,U_d]^{\wedge_p}_{\pd}/(Q)_{\pd}^{\wedge_p}$ is the quotient of $B[U_0,\dots,U_d]^{\wedge_p}_{\pd}$ by the closed pd-ideal generated by $Q$.

        \item[(3)] Keep notations as in items (1) or (2). Denote by $\rD\rR(D,\rd)$ the de Rham complex associated to the differential $\rd$ on $D$. If $u$ is invertible in $B$, then the following sequence
        \[0\to B\to \rD\rR(D,\rd)\]
        is exact.
    \end{enumerate}
  \end{lem}
  \begin{proof}
      For Item (1): By definition of $Q$, for any $0\leq i\leq r$ and any $n\geq 1$, we have 
      \[(1+bU_i)\frac{\partial}{\partial U_i}(Q^{[n]}) = (1+bQ)Q^{[n-1]}.\]
      So for any $F\in B[U_0,\dots,U_d]^{\wedge_p}_{\pd}$, we have
      \begin{equation}\label{equ:defined d on quotient}
          \rd(Q^{[n]}F) = Q^{[n]}\rd(F)+F(1+bQ)Q^{[n-1]}\sum_{i=0}^r\rd U_i,
      \end{equation}
      yielding that $\rd$ induces a well-defined differential $D\to (\oplus_{i=0}^dD\cdot\rd U_i)/D\cdot\sum_{i=0}^r\rd U_i$ as desired.

      Item (2) follows from a similar (and much easier) argument.

      For Item (3): We only deal with the case corresponding to Item (1) while the rest can be proved in the same way. Without loss of generality, we may assume $u=1$. Note that the morphism
      \[i:B[Q,U_1,\dots,U_d]^{\wedge_p}_{\pd} \to B[U_0,\dots,U_d]^{\wedge_p}_{\pd}\]
      sending $Q$ to $b^{-1}(\prod_{i=0}^r(1+bU_i)-1)$ and each $U_i$ to $U_i$ is an isomorphism. Indeed, its inverse is given by
      \[i^{-1}:B[U_0,\dots,U_d]^{\wedge_p}_{\pd}\to B[Q, U_1,\dots,U_d]^{\wedge_p}_{\pd}\]
      sending $U_0$ to $b^{-1}((1+bQ)\prod_{i=1}^r(1+bU_i)^{-1}-1)$ (which is well-defined as $b$ is topologically nilpotent). As a consequence, we have an isomorphism
      \[D \cong B[U_1,\dots,U_d]^{\wedge_p}_{\pd}.\]
      Via the isomorphism
      \[(\oplus_{i=0}^dD\cdot \rd U_i)/D\cdot\sum_{i=0}^r\rd U_i\cong \oplus_{j=1}^dD\cdot\rd U_j,\]
      by \eqref{equ:defined d on quotient}, we see that the differential $\rd$ on $D\cong B[U_1,\dots,U_d]^{\wedge_p}_{\pd}$ is given by
      \[\rd = \sum_{j=1}^d(1+bU_j)\frac{\partial}{\partial U_j}\otimes\rd U_j:B[U_1,\dots,U_d]^{\wedge_p}_{\pd}\to \oplus_{j=1}^dB[U_1,\dots,U_d]^{\wedge_p}_{\pd}\cdot\rd U_j.\]
      
      Therefore, we are reduced to the case for $Q=0$. In this case, the de Rham complex 
      \[\rD\rR(B[U_0,\dots,U_d]^{\wedge_p}_{\pd},\rd = \sum_{j=0}^d(1+bU_j)\frac{\partial}{\partial U_j}\otimes\rd U_j)\]
      is represented by the Koszul complex
      \[\rK(B[U_0,\dots,U_d]^{\wedge_p}_{\pd};(1+bU_0)\frac{\partial}{\partial U_0},\dots,(1+bU_d)\frac{\partial}{\partial U_d}).\]
      By K\"unneth formula, there exists a quasi-isomorphism
      \[\rK(B[U_0,\dots,U_d]^{\wedge_p}_{\pd};(1+bU_0)\frac{\partial}{\partial U_0},\dots,(1+bU_d)\frac{\partial}{\partial U_d})\simeq \widehat \otimes_{j=0}^d\rK(B[U_j]^{\wedge_p}_{\pd};(1+bU_j)\frac{\partial}{\partial U_j}).\]
      Therefore, we are reduced to the case for $d = 0$. In other words, we have to show that the complex
      \[0\to B\to B[U]^{\wedge_p}_{\pd}\xrightarrow{(1+bU)\frac{\partial}{\partial U}}B[U]^{\wedge_p}_{\pd}\to0.\]
      is exact. But this is trivial because $1+bU$ is a unit in $B[U]^{\wedge_p}_{\pd}$ and the usual de Rham complex
      \[0\to B\to B[U]^{\wedge_p}_{\pd}\xrightarrow{\frac{\partial}{\partial U}}B[U]^{\wedge_p}_{\pd}\to 0\]
      is exact. This completes the proof.
  \end{proof}
  
  \begin{prop}\label{prop:local description of C_Y}
      Keep notations as above. For any $n\geq 1$, the map of $\CCtilde_K^{1,+}(U)$-algebras
      \[\iota_n:\left(\CCtilde_K^{1,+}(U)[U_0,\dots,U_d]^{\wedge_p}_{\pd}/(\frac{u}{\xi_K}(\prod_{i=0}^r(1+\frac{\xi_K}{u}U_i)-1))^{\wedge_p}_{\pd}\right)[\frac{1}{p}]/(\xi_K \pi^{-1})^n\to C_{\frakY_Z}/(\xi_K \pi^{-1})^n\]
      sending each $U_i$ to $\frac{u}{\xi_K}(\frac{T_i}{t_i}-1)$ is a well-defined isomorphism. Moreover, via the above isomorphism and the isomorphism 
      \[\Omega^{1,\log}_{\wtr_Z}\{-1\}\cong(\oplus_{0=1}^d\wtr_Z\cdot\frac{\dlog T_i}{\xi_K})/\wtr_Z\cdot(\sum_{i=0}^r\frac{\dlog T_i}{\xi_K}),\]
      the derivation $\rd$ on $C_{\frakY_Z}$ is given by
      \[\begin{split}\rd=\sum_{i=0}^d(u+\xi_K U_i)\frac{\partial}{\partial U_i}\otimes\frac{\dlog T_i}{\xi_K}\end{split}\]
      in the sense of Lemma \ref{lem:defined d on quotient}.
      Here, $\CCtilde_K^{1,+}(U)[U_0,\dots,U_d]^{\wedge_p}_{\pd}$ denotes the $p$-adic completion of the free pd-algebra over $\CCtilde_K^{1,+}(U)$ generated by $U_0,\dots,U_d$ and $(\frac{u}{\xi_K}(\prod_{i=0}^r(1+\frac{\xi_K}{u}U_i)-1))^{\wedge_p}_{\pd}$ denotes its closed pd-ideal generated by $\frac{u}{\xi_K}(\prod_{i=0}^r(1+\frac{\xi_K}{u}U_i)-1)$.
  \end{prop}
  \begin{proof}
      It suffices to show that $\iota_n$ is a well-defined isomorphism, and the rest follows directly. Without loss of generality, we may assume $\frakY_Z = \frakX_Z$ and $\wtr^{\prime}_Z = \wtr_Z$ as before.

      Let $Z_i$ be as in Corollary \ref{cor:local description of A_Y}. It suffices to show the map
      \[\begin{split}i\colon&\left(\CCtilde_K^{1,+}(U)[U_0,\dots,U_d]^{\wedge_p}_{\pd}/(\frac{u}{\xi_K}(\prod_{i=0}^r(1+\frac{\xi_K}{u}U_i)-1))^{\wedge_p}_{\pd}\right)[\frac{1}{p}]/(\xi_K \pi^{-1})^n\\
      &\to \left(\bA_{\inf,K}(U)[u,Z_0,\dots,Z_d]_{\pd}^{\wedge}/(\frac{u}{\xi_K}(\prod_{i=0}^r(1+\frac{\xi_K}{u}Z_i)-1))_{\pd}^{\wedge}\right)\widehat \otimes_{\bA_{\inf,K}(U)}\CCtilde_K^{1}(U)/(\xi_K \pi^{-1})^n\end{split}\]
      sending each $U_i$ to $Z_i$ is a well-defined isomorphism.

      The well-definedness of $i$ is trivial: As 
      \[\left(\bA_{\inf,K}(U)[u,Z_0,\dots,Z_d]_{\pd}^{\wedge}/(\frac{u}{\xi_K}(\prod_{i=0}^r(1+\frac{\xi_K}{u}Z_i)-1))_{\pd}^{\wedge}\right)\widehat \otimes_{\bA_{\inf,K}(U)}\CCtilde_K^{1,+}(U)/(\xi_K \pi^{-1})^n\]
      is a $p$-adic complete $\CCtilde_K^{1,+}(U)$-algebra and $Z_i$ admits arbitrary pd-powers for all $i$, there exists a well-defined morphism 
      \[\begin{split}
          &\CCtilde_K^{1,+}(U)[U_0,\dots,U_d]^{\wedge_p}_{\pd}\\
          \to&\left(\bA_{\inf,K}(U)[u,Z_0,\dots,Z_d]_{\pd}^{\wedge}/(\frac{u}{\xi_K}(\prod_{i=0}^r(1+\frac{\xi_K}{u}Z_i)-1))_{\pd}^{\wedge}\right)\widehat \otimes_{\bA_{\inf,K}(U)}\CCtilde_K^{1,+}(U)/(\xi_K \pi^{-1})^n
      \end{split}\]
      sending each $U_i$ to $\frac{uZ_i}{\xi_K}$. As the above map carries $\frac{u}{\xi_K}(\prod_{i=0}^r(1+\frac{\xi_K}{u}U_i)-1)$ to $\frac{u}{\xi_K}(\prod_{i=0}^r(1+Z_i)-1)$, which vanishes in its target, after inverting $p$, we get the well-definedness of $i$ as desired.
      
      It remains to show that $i$ is an isomorphism. We will do this by constructing its inverse as follows: As $u,U_0,\dots,U_d$ admit arbitrary pd-powers in 
      \[\left(\CCtilde_K^{1,+}(U)[U_0,\dots,U_d]^{\wedge_p}_{\pd}/(\frac{u}{\xi_K}(\prod_{i=0}^r(1+\frac{\xi_K}{u}U_i)-1))^{\wedge_p}_{\pd}\right)[\frac{1}{p}]/(\xi_K\pi^{-1})^n,\]
      we get a well-defined morphism of $\bA_{\inf,K}(U)$-algebras
      \[\begin{split}
        j_1:\bA_{\inf,K}(U)[u,Z_0,\dots,Z_d]_{\pd}\to\left(\CCtilde_K^{1,+}(U)[U_0,\dots,U_d]^{\wedge_p}_{\pd}/(\frac{u}{\xi_K}(\prod_{i=0}^r(1+\frac{\xi_K}{u}U_i)-1))^{\wedge_p}_{\pd}\right)[\frac{1}{p}]/(\xi_K\pi^{-1})^n
      \end{split}\]
      sending each $Z_i$ to $U_i$.
      As $\frac{\xi_K}{u} = \frac{\xi_K}{\pi}\frac{\pi}{u}$ is nilpotent in $\CCtilde^{1}_K(U)/(\xi_K\pi^{-1})^n$, the above map $j_1$ uniquely factors through $(p,\xi_K,\frac{\xi_K Z_0}{u},\dots,\frac{\xi_K Z_d}{u})$-adic completion of $\bA_{\inf,K}(U)[[\frac{\xi_K Z_0}{u},\dots,\frac{\xi_K Z_d}{u}]][Z_0,\dots,Z_d]_{\pd}$, yielding a morphism
      \[\begin{split}
        j_2:\bA_{\inf,K}(U)[u,Z_0,\dots,Z_d]_{\pd}^{\wedge}\to\left(\CCtilde_K^{1}(U)[U_0,\dots,U_d]^{\wedge_p}_{\pd}/(\frac{u}{\xi_K}(\prod_{i=0}^r(1+\frac{\xi_K}{u}U_i)-1))^{\wedge_p}_{\pd}\right)[\frac{1}{p}]/(\xi_K\pi^{-1})^n.
      \end{split}\]
      As $j_2(\frac{u}{\xi_K}\prod_{i=0}^r(1+\frac{\xi_K}{u}Z_i)-1)) = \frac{u}{\xi_K}\prod_{i=0}^r(1+\frac{\xi_K}{u}U_i)-1)$ which vanishes in the target, after taking linear extension along $\bA_{\inf,K}(U)\to\CCtilde_K^{1}(U)/(\xi_K \pi^{-1})^n$, the map $j_2$ uniquely extends to a morphism of $\CCtilde_K^{1}(U)$-algebras
      \[\begin{split}
        j\colon&\left(\bA_{\inf,K}(U)[u,Z_0,\dots,Z_d]_{\pd}^{\wedge}/(\frac{u}{\xi_K}(\prod_{i=0}^r(1+\frac{\xi_K}{u}Z_i)-1))_{\pd}^{\wedge}\right)\widehat \otimes_{\bA_{\inf,K}(U)}\CCtilde_K^{1}(U)/(\xi_K \pi^{-1})^n\\
        &\to \left(\CCtilde_K^{1,+}(U)[U_0,\dots,U_d]^{\wedge_p}_{\pd}/(\frac{u}{\xi_K}(\prod_{i=0}^r(1+\frac{\xi_K}{u}U_i)-1))^{\wedge_p}_{\pd}\right)[\frac{1}{p}]/(\xi_K\pi^{-1})^n.
      \end{split}\]
      Now, one can conclude by checking that $j$ is the inverse of $i$ as expected directly.
  \end{proof}

    \begin{cor}\label{cor:local description of C_Y}
      Keep notations as above. The morphism  sending each $U_i$ to $\frac{u}{\xi_K}(\frac{T_i}{t_i}-1)$
      \[\iota:\bB_{\dR\mid_U}^+\za U_0,\dots,U_d\ya_{\pd}/(\frac{u}{\xi_K}(\prod_{i=0}^r(1+\frac{\xi_K}{u}U_i)-1))^{\wedge}_{\pd}\to(\calO\bB_{\dR,\pd,Z}^+)_{\mid U}\]
      is an isomorphism of sheaves of $\bB_{\dR\mid_U}^+$-algebras, where $\bB_{\dR\mid_U}^+\za U_0,\dots,U_d\ya_{\pd}$ denotes the $\xi_K$-adic completion of 
      \[\CCtilde^{1}_{K\mid_{U}}[U_0,\dots,U_d]^{\wedge_p}_{\pd}=\CCtilde^{1,+}_{K\mid_{U}}[U_0,\dots,U_d]^{\wedge_p}_{\pd}[\frac{1}{p}]\]
      while $(\frac{u}{\xi_K}(\prod_{i=0}^r(1+\frac{\xi_K}{u}U_i)-1))^{\wedge}_{\pd}$ denotes the closed pd-ideal generated by $\frac{u}{\xi_K}(\prod_{i=0}^r(1+\frac{\xi_K}{u}U_i)-1)$.
      Moreover, via this isomorphism and the isomorphism 
      \[\Omega^{1,\log}_{\frakX_Z}\{-1\}\cong(\oplus_{0=1}^d\calO_{\frakX_Z}\cdot\frac{\dlog T_i}{\xi_K})/(\calO_{\frakX_Z}\cdot\sum_{i=0}^r\frac{\dlog T_i}{\xi_K}),\]
      the differential $\rd$ on $(\calO\bB_{\dR,\pd,Z}^+)_{\mid U}$ is given by
      \[\begin{split}\rd=\sum_{i=0}^d(u+\xi_K U_i)\frac{\partial}{\partial U_i}\otimes\frac{\dlog T_i}{\xi_K}\end{split}\]
      in the sense of Lemma \ref{lem:defined d on quotient}.
  \end{cor}
  \begin{proof}
      It suffices to show that $\iota$ is an isomorphism and the rest follows directly.
      By Proposition \ref{prop:local description of C_Y}, for any $\frakY_Z\in \Sigma_U$, we have
      \begin{equation*}
      \begin{split}
          B_{\dR,\pd,\frakY_Z}^+& = \varprojlim_n C_{\frakY_Z}/\xi_K^n\\
          &\cong \varprojlim_n \left(\CCtilde_K^{1,+}(U)[U_0,\dots,U_d]^{\wedge_p}_{\pd}/(\frac{u}{\xi_K}(\prod_{i=0}^r(1+\frac{\xi_K}{u}U_i)-1))^{\wedge_p}_{\pd}\right)[\frac{1}{p}]/(\xi_K \pi^{-1})^n\\
          &= \varprojlim_n\CCtilde^{1,+}(U)[U_0,\dots,U_d]^{\wedge_p}_{\pd}[\frac{1}{p}]/(\xi_K\pi^{-1})^n\\
          &=\bB_{\dR}^+(U)\za U_0,\dots,U_d\ya_{\pd}/(\frac{u}{\xi_K}(\prod_{i=0}^r(1+\frac{\xi_K}{u}U_i)-1))^{\wedge}_{\pd},
      \end{split}
      \end{equation*}
      which turns out independent of the choice of $\frakY_Z$. So, the presheaf
      \[U\in X_{Z,v}\mapsto \colim_{\frakY_Z\in\Sigma}B_{\dR,\pd,\frakY_Z}^+\]
      coincides with the presheaf
      \[U\in X_{Z,v}\mapsto \bB_{\dR}^+(U)\za U_0,\dots,U_d\ya_{\pd}/(\frac{u}{\xi_K}(\prod_{i=0}^r(1+\frac{\xi_K}{u}U_i)-1))^{\wedge}_{\pd}\]
      which is actually a sheaf. Thus, we get 
      \[\calO\bB_{\dR,\pd}^+(U) \cong \bB_{\dR}^+(U)\za U_0,\dots,U_d\ya_{\pd}/(\frac{u}{\xi_K}(\prod_{i=0}^r(1+\frac{\xi_K}{u}U_i)-1))^{\wedge}_{\pd}.\]
      and that, by letting $U$ vary, $\iota$ is an isomorphism as desired.
  \end{proof}
  One can give another description of $(\calO\bB^{+}_{\dR,\pd,Z},\rd)$.
  \begin{prop}\label{prop:local description of OB_dR^+}
      Keep notations as above.
      The series $\frac{u}{\xi_K}\log\frac{t_i}{T_i} = -\frac{u}{\xi_K}\sum_{n\geq 1}\frac{(1-\frac{T_i}{t_i})^n}{n}$ converges in $(\calO\bB^{+}_{\dR,\pd,Z})_{\mid_U}$. The morphism sending each $W_i$ to $\frac{u}{\xi_K}\log\frac{t_i}{T_i}$
      \[\iota:\bB^+_{\dR\mid_{U}}\za W_0,\dots,W_d\ya_{\pd}/(W_0+\cdots+W_r)^{\wedge}_{\pd}\to (\calO\bB^+_{\dR,\pd,Z})_{\mid_U}\]
       is an isomorphism of the sheaves of $\bB_{\dR\mid_{U}}^+$-algebras. Moreover, via this isomorphism and the isomorphism \[\Omega^{1,\log}_{\frakX_Z}\{-1\}\cong (\oplus_{i=0}^d\calO_{\frakX_Z}\cdot\frac{\dlog T_i}{\xi_K})/(\calO_{\frakX_Z}\cdot\sum_{i=0}^r\frac{\dlog T_i}{\xi_K}),\] 
       the differential $\rd$ on $(\calO\bB^+_{\dR,\pd,Z})_{\mid_U}$ is given by
      \[\rd = \sum_{i=0}^d-u\frac{\partial}{\partial W_i}\otimes\frac{\dlog T_i}{\xi_K}\]
      in the sense of Lemma \ref{lem:defined d on quotient}.
  \end{prop}
  \begin{proof}
      Thanks to Corollary \ref{cor:local description of C_Y}, the morphism
      \[\bB_{\dR\mid_U}^+\za U_0,\dots,U_d\ya_{\pd}/(\frac{u}{\xi_K}(\prod_{i=0}^r(1-\frac{\xi_K}{u}U_i)-1))^{\wedge}_{\pd}\to(\calO\bB_{\dR,\pd,Z}^+)_{\mid U}\]
      sending each $U_i$ to $\frac{u}{\xi_K}(1-\frac{T_i}{t_i})$ is an isomorphism.
      Thus, we only need to show that the series
      \[-\frac{u}{\xi_K}\log(1-\frac{\xi_K}{u}U_i)=-\frac{u}{\xi_K}\sum_{n\geq 1}\frac{(\frac{\xi_K}{u}U_i)^n}{n}\]
      converges in $\bB^+_{\dR,\pd\mid_{X_{\infty}}}\za U_1,\dots,U_d\ya_{\pd}$ and the morphism 
      \[\begin{split}i:\bB^+_{\dR\mid_{U}}\za W_0,\dots,W_d\ya_{\pd}/(W_0+\cdots+W_r)^{\wedge}_{\pd}\to\bB_{\dR\mid_U}^+\za U_0,\dots,U_d\ya_{\pd}/(\frac{u}{\xi_K}(\prod_{i=0}^r(1-\frac{\xi_K}{u}U_i)-1))^{\wedge}_{\pd}\end{split}\]
      sending each $W_i$ to $-\frac{u}{\xi_K}\log(1-\frac{\xi_K}{u}U_i)$ is an isomorphism. 
      
      Note that both sides of the following equality
      \[\frac{u}{\xi_K}\sum_{n\geq 1}\frac{(\frac{\xi_K}{u}U_i)^n}{n} = \sum_{n\geq 1}(\frac{\xi_K}{u})^{n-1}(n-1)!U_i^{[n]}\]
      reduce to a finite sum modulo $\xi_K^n$ for any $n\geq 1$. So the series
      \[-\frac{u}{\xi_K}\log(1-\frac{\xi_K}{u}U_i)=\frac{u}{\xi_K}\sum_{n\geq 1}\frac{(\frac{\xi_K}{u}U_i)^n}{n}\]
      converges in $\bB_{\dR\mid_U}^+\za U_0,\dots,U_d\ya_{\pd}$. Similarly, the series 
      \[\frac{u}{\xi_K}(1-\exp(-\frac{\xi_K}{u}W_i)) = \sum_{n\geq 1}(-\frac{\xi_K}{u})^{n-1}W_i^{[n]}\]
      converges in $\bB^+_{\dR\mid_U}\za W_0,\dots,W_d\ya_{\pd}$.
      So, the map of $\bB^+_{\dR\mid_{U}}$-algebras
      \[j:\bB^+_{\dR\mid_{U}}\za U_0,\dots,U_d\ya_{\pd}\to\bB^+_{\dR\mid_{U}}\za W_0,\dots,W_d\ya_{\pd}\]
      sending each $U_i$ to $\frac{u}{\xi_K}(1-\exp(-\frac{\xi_K}{u}W_i))$ is a well-defined isomorphism, and its inverse is induced by the morphism sending $W_i$ to $-\frac{u}{\xi_K}\log(1-\frac{\xi_K}{u}U_i)$. 

      Put $P:=\sum_{i=0}^rW_i$ and $Q:=\frac{u}{\xi_K}(\prod_{i=0}^r(1-\frac{\xi_K}{u}U_i)-1)$.
      To see $i$ is an isomorphism, it remains to show that via the identification 
      \[\bB^+_{\dR\mid_{U}}\za U_0,\dots,U_d\ya_{\pd} = \bfB = \bB^+_{\dR\mid_{U}}\za W_0,\dots,W_d\ya_{\pd},\]
      the polynomials $P$ and $Q$ generate the same pd-ideal in $\bfB$. It suffices to show that $P$ is a unit-multiple of $Q$ in $\bfB$. As
      \[\sum_{n\geq 0}\frac{(-\frac{\xi_K}{u}Q)^n}{n+1}\equiv 1\mod \xi_K\bfB,\]
      it is a unit in $\bfB$ because $\xi_K$ is topologically nilpotent.
      Thus, one can conclude by noting that
      \[P = -\frac{u}{\xi_K}\sum_{i=0}^r\log(1-\frac{\xi_K}{u}U_i) = -\frac{u}{\xi_K}\log(1+\frac{\xi_K}{u}Q) = -Q\sum_{n\geq 0}\frac{(-\frac{\xi_K}{u}Q)^n}{n+1}.\]
  \end{proof}

  Now, we are able to show the following (local) Poincar\'e's Lemma for $(\calO\bB_{\dR,\pd,Z}^+,\rd)$:
  \begin{prop}[Local Poincar\'e's Lemma]\label{prop:local Poincare's Lemma}
      Suppose $\frakX_Z$ is small semi-stable. Then the natural morphism $\BBdRp\to\calO\bB_{\dR,\pd,Z}^+$ induces an exact sequence
      \[0\to\BBdRp\to\rD\rR(\calO\bB_{\dR,\pd,Z}^+,\rd).\]
  \end{prop}
  \begin{proof}
      As $u$ is invertible in $\BBdRp$, the result follows from Proposition \ref{prop:local description of OB_dR^+} together with Lemma \ref{lem:defined d on quotient}(3) directly.
  \end{proof}
  \begin{rmk}\label{rmk:Poincare Lemma}
      Suppose $\frakX_Z$ is small semi-stable and $K = \rW(\kappa)[\frac{1}{p}]$ with $p\geq 3$.
      One can similarly define $\calO\CCtilde^{1,+}_{\pd,Z}$ and $\calO\CCtilde^{1}_{\pd,Z}$ to be the sheafifications of the presheaves sending each affinoid perfectoid $U\in X_{Z,v}$ to 
      \[\colim_{\frakY_Z\in \Sigma_U}C_{\frakY}^+ \text{ and } \colim_{\frakY_Z\in \Sigma_U}C_{\frakY}\]
      respectively. Then the usual (continuous) derivation on $\wtx_Z$ induces a derivation
      \[\rd:\calO\bP_{\pd,Z}\to\calO\bP_{\pd,Z}\otimes_{\calO_{\wtx_Z}}u\Omega^{1,\log}_{\wtx_Z}\{-1\}\]
      for $\bP\in \{\CCtilde^{1,+},\CCtilde^{1}\}$. Then the arguments below also proves that there exists an exact sequence
      \[0\to\bP\to\rD\rR(\calO\bP_{\pd,Z},\rd).\]
      Here, we have to assume that $K$ is unramified and $p\geq 3$ because in this case, $\frac{\xi}{u}\in \CCtilde^{1,+}(U)$ is topologically nilpotent. This will be used to construct the morphism $j$ in the proof of (the analogue of) Proposition \ref{prop:local description of C_Y}. We leave the details to interested readers.
  \end{rmk}
  \begin{rmk}\label{rmk:OA_dR}
      Suppose $\frakX_Z$ is small semi-stable and $K = \rW(\kappa)[\frac{1}{p}]$.
      If we put $\AAdR$ and $\calO\bA_{\dR,Z}$ to be the $(\xi p^{-1})$-adic completion of $\CCtilde^{1,+}$ and $\calO\CCtilde^{1,+}_{\pd,Z}$ (cf. Remark \ref{rmk:Poincare Lemma}) respectively, then the derivation $\rd$ on $\calO_{\wtx_Z}$ induces a derivation 
      \[u^{-1}\rd:\calO\AAdR\to \calO\AAdR\otimes_{\calO_{\wtx}}\cdot\Omega^1_{\wtx_Z}\{-1\}.\]
      The last paragraph in the proof of Proposition \ref{prop:local description of OB_dR^+} together with Proposition \ref{prop:local description of C_Y} 
      tells us the morphism
      \[\iota:\bA_{\dR\mid_{U}}[W_1,\dots,W_d]^{\wedge}_{\pd}\to (\calO\bA_{\dR,Z})_{\mid U}\]
      sending $W_i$ to $\frac{u^2}{t}\log\frac{[T_i^{\flat}]}{T_i}$ is a well-defined isomorphism such that the differential $u^{-1}\rd$ is given by 
      \[u^{-1}\rd = -\sum_{i=1}^d\frac{\partial}{\partial W_i}\otimes\frac{u\dlog T_i}{t}.\]
      One can check directly the sequence
      \[0\to\AAdR\to\calO\bA_{\dR,Z}\otimes_{\calO_{\wtx_Z}}\Omega^{1,\log}_{\wtx_Z}\{-1\}\to\cdots\to\calO\bA_{\dR,Z}\otimes_{\calO_{\wtx_Z}}\Omega^d_{\wtx_Z}\{-d\}\]
      is exact. This will lead to an integral version of $p$-adic Riemann--Hilbert correspondence lifting the integral $p$-adic Simpson correspondence in \cite{MW-AIM,SW24}; that is, the relationship between the category of certain $\bA_{\dR}$-local systems on $X_{Z,v}$ and the category of certain connections on $\widetilde \frakX_Z\times_{\Spf(\bA_{\inf}(Z))}\Spf(\bA_{\dR}(Z))$.
  \end{rmk}

    \begin{rmk}
      Keep notations as in Remark \ref{rmk:OA_dR}.
      As mentioned by a referee, our work is closely related with that of Zeyu Liu \cite{Liu}. In \emph{loc.cit.}, for a $p$-adic formal scheme $\frakX$, Liu considered a ``modified prismatization'' $\widetilde{\frakX^{\Prism}}$, on which the coherent vector bundles correspond to the prismatic crystals with coefficients in $\calO_{\Prism}[[\frac{\calI_{\Prism}}{p}]]$ on the prismatic site $(\frakX)_{\Prism}$. For example, for any affinoid perfectoid $Z\in \Perfd$, we have
      \[\calO_{\Prism}[[\frac{\calI_{\Prism}}{p}]](\bA_{\inf}(Z),(\xi)) = \bA_{\inf}(Z)[[\frac{\xi}{p}]] = \bA_{\inf}(Z)[(\frac{[\varpi]}{p})^{\pm 1}]^{\wedge_{(p,\xi p^{-1})}} = \bA_{\dR}(Z).\]
      It seems that one can give a prismatic construction of $\calO\bA_{\dR}$ using $\widetilde{\frakX^{\Prism}}$ that lifts the period sheaf $\calB_{\widetilde \frakX}$ in \cite[Definition 7.17]{AHLB25} at least when $\frakX$ is smooth over $\calO_C$.
  \end{rmk}

  We end this subsection with the following example.
  \begin{exam}\label{exam:Gamma torsor}
      Suppose that $\frakX_Z = \Spf(\calR_Z)$ is small semi-stable with a chart 
      \[\psi:\frakX_Z\to \Spf(A_0^+:=A^+\za T_0,\dots,T_r,T_{r+1}^{\pm 1},\dots,T_d^{\pm 1}\ya/(T_0\cdots T_r-p^a))\]
      and a lifting $\wtx_Z = \Spf(\wtr_Z)$ and let $X_Z = \Spa(R_Z,R_Z^+=\calR_Z)$ be its generic fiber as before. For any $n\geq 0$, define
      \[A_n^+:= A^+\za T_0^{\frac{1}{p^n}},\dots,T_r^{\frac{1}{p^n}},T_{r+1}^{\pm \frac{1}{p^n}},\dots,T_d^{\pm \frac{1}{p^n}}\ya/((T_0\cdots T_r)^{\frac{1}{p^n}}-p^{\frac{a}{p^n}})\]
      and 
      \[\widehat A_{\infty}^+:=(\cup_{n\geq 0}A_n^+)^{\wedge_p} = A^+\za T_0^{\frac{1}{p^{\infty}}},\dots,T_r^{\frac{1}{p^{\infty}}},T_{r+1}^{\pm \frac{1}{p^{\infty}}},\dots,T_d^{\pm \frac{1}{p^{\infty}}}\ya/((T_0\cdots T_r)^{\frac{1}{p^{\infty}}}-p^{\frac{a}{p^{\infty}}})\]
      where $((T_0\cdots T_r)^{\frac{1}{p^{\infty}}}-p^{\frac{a}{p^{\infty}}})$ denotes the closed ideal generated by $\{(T_0\cdots T_r)^{\frac{1}{p^n}}-p^{\frac{a}{p^n}}\}_{n\geq 0}$.    
      Let $X_{\infty,Z}:=\Spa(\widehat R_{\infty,Z},\widehat R_{\infty,Z}^+)$ be the base-change of $X_Z$ along the pro-\'etale $\Gamma$-torsor
      \[\Spa(\widehat A_{\infty} =\widehat A_{\infty}^+[\frac{1}{p}],\widehat A_{\infty}^+)\to \Spa(A_0 = A_0^+[\frac{1}{p}],A_0^+),\]
      where $\Gamma$ is the following group
      \begin{equation}\label{equ:gamma group-I}
          \Gamma = \{\delta_0^{n_0}\cdots\delta_r^{n_r}\delta_{r+1}^{n_{r+1}}\cdots\delta_d^{n_d}\in \bZ_p^{r+1}\mid n_0,\dots,n_d\in\bZ_p \text{ such that } n_0+\cdots+n_r = 0\}
      \end{equation}
      and for any $0\leq i,j\leq d$ and any $n\geq 0$, the $\Gamma$-action on $\widehat A_{\infty}$ is determined by
      \begin{equation}\label{equ:gamma acts on T}
          \delta_i(T_j^{\frac{1}{p^n}}) = \zeta_{p^n}^{\delta_{ij}}T_j^{\frac{1}{p^n}}.
      \end{equation}
      Here, $\delta_{ij}$ denotes the Kronecker's delta.
      If we put $\gamma_i:=\delta_0^{-1}\delta_i$ for any $1\leq i\leq r$ and $\gamma_j = \delta_j$ for any $r+1\leq j\leq d$, then we have an isomorphism
      \begin{equation}\label{equ:gamma group-II}
          \Gamma \cong \oplus_{i=1}^d\Zp\cdot\gamma_i.
      \end{equation}
      Put $T_i^{\flat}:=(T_i,T_i^{1/p},\dots)\in \widehat R_{\infty,Z}^{\flat+}$ for all $i$, and then we have
      \[T_0^{\flat}\cdots T_r^{\flat} = \varpi^a.\] 
      and that the $\Gamma$-action on $\widehat R_{\infty,Z}^{\flat,+}$ is determined by that for any $0\leq i,j\leq d$ and any $n\geq 0$,
      \begin{equation}\label{equ:gamma acts on T^flat}
          \delta_i((T_j^{\flat})^{\frac{1}{p^n}}) = \epsilon^{\delta_{ij}\frac{1}{p^n}}(T_j^{\flat})^{\frac{1}{p^n}}.
      \end{equation}
      Now, for $U = X_{\infty,Z}$ in Lemma \ref{lem:existence of log-structure}, we can choose $(\underline t_i,u_i) = (T_i^{\flat},1)$ for any $0\leq i\leq r$ and then we can choose $t_i = [T_i^{\flat}]$ in Lemma \ref{lem:kernel is finite}. Thus the morphism of the sheaves of $\bB_{\dR\mid_{X_{\infty,Z}}}^+$-algebras
      \[\iota:\bB^+_{\dR\mid_{X_{\infty,Z}}}\za W_0,\dots,W_d\ya_{\pd}/(W_0+\cdots+W_r)^{\wedge}_{\pd}\to (\calO\bB^+_{\dR,\pd,Z})_{\mid_{X_{\infty,Z}}}\]
      sending each $W_i$ to $\frac{u}{\xi_K}\log\frac{[T_i]^{\flat}}{T_i}$ is an isomorphism such that the derivation $\rd$ on $(\calO\bB^+_{\dR,\pd,Z})_{\mid_{X_{\infty,Z}}}$ is given by
      \[\rd = \sum_{i=0}^d-u\frac{\partial}{\partial W_i}\otimes\frac{\dlog T_i}{\xi_K}.\]
      Note that $(\calO\bB^+_{\dR,\pd,Z})_{\mid_{X_{\infty,Z}}}$ admits a natural action of $\Gamma$. Via the above isomorphism $\iota$, by \eqref{equ:gamma acts on T} and \eqref{equ:gamma acts on T^flat}, this action is determinded by
      \begin{equation}\label{equ:gamma acts on W}
          \delta_i(W_j) = W_j+\delta_{ij}u\frac{t}{\xi_K},~\forall~0\leq i,j\leq d,
      \end{equation}
      which clearly commutes with $\rd$. Choose $\xi_K = [\underline \pi]-\pi$ and $\xi = \frac{\phi(u)}{u}$. Then we have
      \[\frac{t}{\xi_K} = u\frac{t}{u\xi}\frac{\xi}{E([\underline \pi])}\frac{E([\underline \pi])-E(\pi)}{[\underline \pi]-\pi}.\]
      Denote by $\rho_K$ the reduction of $\frac{t}{\xi_K}\in \BdRp$ modulo $t$. Note that
      \begin{enumerate}
          \item[$\bullet$] $\frac{\xi}{E([\underline \pi])}\in\bfA_{\inf}^{\times}$,

          \item[$\bullet$] $u\equiv \zeta_p-1 \mod t\BdRp$, $\frac{t}{u\xi}\equiv 1\mod t\BdRp$ and $\frac{E([\underline \pi])-E(\pi)}{[\underline \pi]-\pi}\equiv E'(\pi)\mod t\BdRp$.
      \end{enumerate}
      We have $\nu_p(\rho_K) = \nu_p((\zeta_p-1)E'(\pi)) = \frac{1}{p-1}+\nu_p(\calD_K)$, where $\calD_K$ is the ideal of differentials of $K$. Modulo $t$, we see that the $\Gamma$-action on $(\calO\bB^+_{\dR,\pd,Z})_{\mid_{X_{\infty,Z}}}/t$ is given by
      \begin{equation}\label{equ:gamma acts on W-II}
          \delta_i(W_j) = W_j+\delta_{ij}\rho_K(\zeta_p-1),~\forall~0\leq i,j\leq d.
      \end{equation}
      So we get the same $\Gamma$-action as in \cite[Prop. 3.10]{SW24}.
  \end{exam}

\subsection{Construction of $(\calO\bB_{\dR,\pd}^+,\rd)$: The general case}
  Let $\frakX_Z$ be a semi-stable formal scheme over $A^+$ and let $X_Z$ be its generic fiber. Note that the sub-site
  \[X_{Z,v}^{\Box}:=\{U\in X_{Z,v}\mid \text{The composite $U\to X_Z\to \frakX_Z$ factors through some small affine $\frakX^{\prime}_Z\to\frakX_Z$}\}\]
  forms a basis for the topology on $X_{Z,v}$.
  \begin{dfn}\label{dfn:Period sheaf-global}
      Let $\frakX_Z$ be a liftable semi-stable formal scheme of relative dimension $d$ over $A^+$ with a fixed lifting $\wtx_Z$ over $\bA_{\inf,K}(Z)$ as above.
      Denote by $(\calO\bB_{\dR,\pd,Z}^+,\rd)$ the period sheaf with connection on $X_{Z,v}$ associated to the presheaf with connection which sends each $U\in X_{Z,v}^{\Box}$ to $((\calO\bB_{\dR,\pd,Z}^+)_{\mid_U},\rd)(U)$ as defined in Definition \ref{dfn:Period sheaf-local}. Denote by $\rD\rR(\calO\bB_{\dR,\pd,Z}^+,\rd)$ the corresponding de Rham complex
      \[\begin{split}\calO\bB_{\dR,\pd,Z}^+\xrightarrow{\rd}\calO\bB_{\dR,\pd,Z}^+\otimes_{\calO_{\wtx_Z}}\Omega^{1,\log}_{\wtx_Z}\{-1\}\xrightarrow{\rd}\cdots \xrightarrow{\rd}\calO\bB_{\dR,\pd,Z}^+\otimes_{\calO_{\wtx_Z}}\Omega^{d,\log}_{\wtx_Z}\{-d\}.\end{split}\]
  \end{dfn}
  \begin{thm}[Poincare's Lemma]\label{thm:Poincare's Lemma}
      Keep notations as above. Then the following complex
      \[0\to\BBdRp\to \rD\rR(\calO\bB_{\dR,\pd,Z}^+,\rd)\]
      of sheaves of $\BBdRp$-algebras on $X_{Z,v}$ is exact.
  \end{thm}
  \begin{proof}
      Since the problem is local on $\frakX_{Z,\et}$, we may assume $\frakX_Z$ is small affine, and then Proposition \ref{prop:local Poincare's Lemma} applies.
  \end{proof}
  \begin{rmk}\label{rmk:functoriality of period rings}
      Let $\frakX_Z$ be a semi-stable formal scheme over $A^+$ with a fixed flat lifting $\wtx_Z$ as before.
      The construction of $(\calO\bB_{\dR,\pd,Z}^+,\rd)$ is clearly functorial in $Z$ in the following sense:
      Let $Z^{\prime} = \Spa(A^{\prime},A^{\prime+})$ be an affinoid perfectoid space over $Z$. Let $\frakX_{Z^{\prime}}$, $X_{Z^{\prime}}$ and $\wtX_{Z^{\prime}}$ be the base-changes of $\frakX_Z$, $X_Z$ and $\wtx_Z$ along $\Spf(A^{\prime+})\to\Spf(A^+)$, $Z^{\prime}\to Z$ and $\Spf(\bA_{\inf,K}(Z^{\prime}))\to\Spf(\bA_{\inf,K}(Z))$ respectively. Let $(\calO\bB_{\dR,\pd,Z}^+,\rd)$ be the period sheaf with connection on $X_{Z,v}$ associated to the given lifting $\wtx_Z$. Then by construction, its restriction $(\calO\bB_{\dR,\pd,Z}^+,\rd)_{\mid_{X_{Z^{\prime},v}}}$ to $X_{Z^{\prime},v}$ is nothing but the period sheaf with connection on $X_{Z^{\prime},v}$ associated to the lifting $\wtx_{Z^{\prime}}$ of $\frakX_{Z^{\prime}}$.
  \end{rmk}

\begin{rmk}\label{rmk:compare to MW-Integral}
    The period sheaves with connnection $(\calO\bA_{\dR,\pd,Z},-u^{-1}\rd)$ and $(\calO\bB_{\dR,\pd,Z}^+,-u^{-1}\rd)$ reduce to $(\calO\widehat \bC_{\pd}^+,\Theta)$ and $(\calO\widehat \bC_{\pd},\Theta)$ in \cite{SW24} modulo $\xi_K \pi^{-1}$. For simplicity, we only check this on $X_{\infty,Z,v}$ when $\frakX_Z = \Spf(\calR_Z)$ is small semi-stable: Keep notations in Example \ref{exam:Gamma torsor}, and then the reduction of $(\calO\bA_{\dR,\pd,Z},-u^{-1}\rd)$ is given by
    \[(\widehat \calO_{X_Z}^+[W_0,\dots,W_d]^{\wedge_p}_{\pd}/(W_0+\cdots+W_r)^{\wedge_p}_{\pd},\sum_{i=0}^d\frac{\partial}{\partial W_i}\otimes\frac{\dlog T_i}{\xi})\]
    such that the $\Gamma$-action is uniquely determined by
    \[\delta_i(W_j) = W_j+\delta_{ij}\rho_K(\zeta_p-1),~\forall~1\leq i,j\leq d\]
    by noting that the reduction of $u$ is $\zeta_p-1$. Now, one can conclude by applying \cite[Prop. 3.10]{SW24}. When $\frakX_Z$ is smooth, one can also check the reduction of $\calO\bA_{\dR,\pd,Z}$ modulo $\xi_K \pi^{-1}$ behaves like $\calB_{\widetilde \calX}$ appearing in \cite[Lem. 3.8]{AHLB23} (see also \cite[Prop. 4.8]{AHLB25}) and like $(\calO\widehat \bC_{\pd}^+,\Theta)$ in \cite{MW-AIM}.
\end{rmk}

\section{Local small Riemann--Hilbert correspondence}\label{sec:local RH}
  In this section, we fix a $Z=\Spa(A,A^+)\in\Perfd$ and let $\frakX_Z$ be a liftable semi-stable formal scheme over $A^+$ with a fixed lifting $\wtx_Z$ over $\bA_{\inf,K}(Z)$. Let $(\calO\bB_{\dR,\pd,Z}^+,\rd)$ be the period sheaf with connection as defined in the previous section and denote by $(\calO\widehat \bC_{\pd,Z},\Theta)$ its reduction modulo $t$. For any $1\leq n\leq \infty$ and $r\geq 1$, denote by 
  \[\rL\rS_r(X_Z,\BBdRpn)^{\Hsmall}(Z) \text{ and }\MIC_r(\wtX_{Z,n})^{\Hsmall}(Z)\]
  the categories of Hitchin-small $\BBdRpn$-local systems on $X_{Z,v}$ and integrable connections on $\wtX_{Z,n}$ of rank $r$ respectively. In particular, when $n=1$, we put
  \[v\Bun_r(X_Z)^{\Hsmall}(Z):=\rL\rS_r(X_Z,\bB_{\dR,1}^+)^{\Hsmall}(Z) \text{ and }\HIG_r(X_Z)^{\Hsmall}(Z):=\MIC_r(\wtX_{Z,1})^{\Hsmall}(Z)\]
  for short. Denote by
  \[\begin{split}
      &\rL\rS(X_Z,\BBdRpn)^{\Hsmall}(Z):=\cup_{r\geq 1}\rL\rS_r(X_Z,\BBdRpn)^{\Hsmall}(Z)\\
      &\MIC(\wtX_{Z,n})^{\Hsmall}(Z) = \cup_{r\geq 1}\MIC_r(\wtX_{Z,n})^{\Hsmall}(Z)\\
      &v\Bun(X_Z)^{\Hsmall}(Z):=\cup_{r\geq 1}v\Bun_r(X_Z)^{\Hsmall}(Z)\\
      &\HIG(X_Z)^{\Hsmall}(Z):=\cup_{r\geq 1}\HIG_r(X_Z)^{\Hsmall}(Z)\\
  \end{split}\]
  the categories of Hitchin-small $\BBdRpn$-local systems on $X_{Z,v}$, Hitchin-small integrable connections on $\wtX_{Z,n}$, Hitchin-small $v$-bundles on $X_{Z,v}$, and Hitchin-small Higgs bundles on $X_{Z,\et}$, respectively. Recall in \S\ref{sec:Introduction}, the Hitchin-smallness was defined by Hitchin fibration. In this section, we want to give an explicit description of Hitchin-smallness, and establish a rank-preserving equivalence of categories
  \[\rL\rS^{\Hsmall}_r(X_Z,\BBdRpn)(Z)\simeq\MIC^{\Hsmall}_r(\wtX_{Z,n})(Z)\]
  when $\frakX_Z$ is furthermore small semi-stable.
  
\subsection{A criterion for being Hitchin-small}\label{ssec:Review of SW24}
  Recall by Remark \ref{rmk:Hitchin small can be checked modulo t}, to check that a $\BBdRpn$-local system (resp. integrable connection on $\wtX_{Z,n}$) is Hitchin-small, it suffices to show that its reduction modulo $t$ is a Hitchin-small $v$-bundle on $X_{Z,v}$ (resp. Hitchin-small Higgs bundle on $X_{Z,\et}$). So, it is enough to give explicit descriptions of Hitchin-small $v$-bundle on $X_{Z,v}$ and Hitchin-small Higgs bundle on $X_{Z,\et}$. 

    \begin{dfn}\label{dfn:Hitchin-small representation}
      Suppose that $\frakX_Z = \Spf(\calR_Z)$ is small semi-stable and keep notations in Example \ref{exam:Gamma torsor}.
      Let $\bfB\in\{\calR_Z=R_Z^+,R_Z,\widehat R_{\infty,Z}^+,\widehat R_{\infty,Z}\}$
      \begin{enumerate}
          \item By a \emph{Hitchin-small representation} of $\Gamma$ over $\calR_Z$ of rank $r$, we mean a finite projective $\calR_Z$-module $L^+$ of rank $r$ endowed with an action of $\Gamma$ such that via the isomorphism \eqref{equ:gamma group-II}
          \[\Gamma = \Zp\gamma_1\oplus\cdots\oplus\Zp\gamma_d,\]
          for any $1\leq i\leq d$, there exists a topologically nilpotent $\theta_{i}\in\End_{\calR_Z}(L^+)$ such that for any $x\in L^+$, we have 
          \[\gamma_i(x) = \exp(\rho_K(\zeta_p-1)\theta_{i})(x).\]
          Denote by $\Rep_{\Gamma}^{\Hsmall}(\calR_Z)$ the category of Hitchin-small representations of $\Gamma$ over $\calR_Z$.

          \item In general, by a \emph{Hitchin-small representation} of $\Gamma$ over $\bfB$ of rank $r$, we mean a finite projective $\bfB$-module $L$ endowed with an action of $\Gamma$ satisfying the condition that there exists some $L^+\in\Rep_{\Gamma}^{\Hsmall}(\calR_Z)$ of rank $r$ such that 
          \[L\cong L^+\otimes_{\calR_Z}\bfB.\]
          Denote by $\Rep_{\Gamma}^{\Hsmall}(\bfB)$ the category of Hitchin-small representations of $\Gamma$ over $\bfB$.

          \item By a \emph{Hitchin-small Higgs module} over $\calR_Z$ of rank $r$, we mean a pair 
          \[(H^+,\theta^+:H^+\to H^+\otimes_{\calR_Z}\Omega^{1,\log}_{\calR_Z}\{-1\})\]
          satisfying $\theta^+\wedge\theta^+ = 0$ such that $H^+$ is a finite projective $\calR_Z$-module of rank $r$ and that via the identification
          \[\Omega^{1,\log}_{\calR_Z} = (\oplus_{i=0}^r\calR_Z\cdot\dlog T_i)/(\calR_Z\cdot\sum_{i=0}^r\dlog T_i)\cong \oplus_{i=1}^d\calR_Z\cdot\dlog T_i,\]
          the Higgs field $\theta^+$ is of the form
          \[\theta^+ = \sum_{i=1}^d(\zeta_p-1)\theta_i\otimes\frac{\dlog T_i}{\xi_K}\]
          with $\theta_i\in\End_{\calR_Z}(H^+)$ topologically nilpotent. Denote by $\HIG^{\Hsmall}(\calR_Z)$ the category of Hitchin-small Higgs module over $\calR_Z$. A Higgs module $(H,\theta)$ over $R_Z$ of rank $r$ is called \emph{Hitchin-small} if it is of the form $(H,\theta) = (H^+,\theta^+)[\frac{1}{p}]$ for some $(H^+,\theta^+)\in\HIG^{\Hsmall}(\calR_Z)$ of rank $r$. Denote by $\HIG^{\Hsmall}(R_Z)$ the category of Hitchin-small Higgs module over $R_Z$.
      \end{enumerate}
  \end{dfn}
  
  \begin{prop}\label{prop:Being Hitchin-small}
      Suppose that $\frakX_Z$ is a semi-stable (not necessarily small) formal scheme over $A^+$ with generic fiber $X_Z$.
      \begin{enumerate}
          \item[(1)] A $v$-bundle $\bL$ on $X_{Z,v}$ of rank $r$ is Hitchin-small if and only if there exists an \'etale covering $\{\frakX_{Z,i}\to\frakX_Z\}_{i\in I}$ of $\frakX_Z$ such that for each $i$, $\frakX_{Z,i}=\Spf(\calR_{Z,i})$ is small semi-stable and if we denote by $X_{i,\infty,Z} = \Spa(\widehat R_{i,\infty,Z},\widehat R_{i,\infty,Z}^+)\to X_{i,Z}$ the corresponding $\Gamma$-torsor in the sense of Example \ref{exam:Gamma torsor}, then the evaluation $\bL(X_{i,\infty,Z})$ is a Hitchin-small representation of $\Gamma$ of rank $r$ over $\widehat R_{i,\infty,Z}$.

          \item[(2)] A Higgs bundle $(\calH,\theta)$ of rank $r$ on $X_{Z,\et}$ is Hitchin-small if and only if  there exists an \'etale covering $\{\frakX_{Z,i}\to\frakX_Z\}_{i\in I}$ of $\frakX_Z$ such that for each $i$, $\frakX_{Z,i}=\Spf(\calR_{Z,i})$ is small semi-stable and the evaluation $(\calH,\theta)(X_{i,Z})$ is a Hitchin-small Higgs module of rank $r$ over $R_{i,Z} = \calR_{i,Z}[\frac{1}{p}]$.
      \end{enumerate}
  \end{prop}
  \begin{proof}
       Comparing Definition \ref{dfn:Hitchin-small representation} with \cite[Def. 4.1, Def. 5.14 and Def. 5.15]{SW24}, the result is exactly \cite[Prop. 5.21]{SW24}.
  \end{proof}

  Now, we are going to give an explicit description for Hitchin-small $\BBdRpn$-local systems on $X_{Z,v}$ and Hitchin-small integrable connections on $\wtX_{Z,n}$ for general $n\geq 1$.

  \begin{construction}\label{Construction:local coefficients ring}
      Suppose that $\frakX_Z = \Spf(\calR_Z)$ is small semi-stable and keep notations in Example \ref{exam:Gamma torsor}. By Corollary \ref{cor:existence of log-structure}, there exists a unique morphism of $\bA_{\inf,K}(Z)$-algebras 
      \[\widetilde \psi:\wtr_Z\to \bA_{\inf,K}(X_{\infty,Z})\]
      sending each $T_i$ to $[T_i^{\flat}]$ for any $0\leq i\leq d$. Inverting $p$ and taking $\xi_K$-adic completion, we get a morphism of $\BBdRp(Z)$-algebras
      \[\widetilde \psi_{\dR}:\wtR_Z\to \BBdRp(X_{\infty,Z}).\]
        Put $\Bppsi = \BBdRp(X_{\infty,Z})$ and then it is equipped with a $\Gamma$-action such that for any $0\leq i,j\leq d$,
      \[\delta_i([T_j^{\flat}]) = [\epsilon]^{\delta_{ij}}[T_j^{\flat}].\]
      Then the morphism $\widetilde{\psi}_{\dR}$ identifies $\wtR_Z$ with a $\Gamma$-stable sub-$\BBdRp(Z)$-algebra of $\Bppsi$, denoted by $\Bpsi$, such that $\Bppsi$ admits a $\Gamma$-equivariant decomposition
      \[\Bppsi = \widehat \oplus_{\underline \alpha \in J_r}\Bpsi\cdot[(T_0^{\flat})^{\alpha_0}]\cdots[(T_d^{\flat})^{\alpha_d}],\]
      where 
      \[J_r = \{\underline \alpha = (\alpha_0,\dots.\alpha_d)\in (\bN[\frac{1}{p}]\cap[0,1))^{d+1}\mid \alpha_0+\cdots+\alpha_r = 0\}.\]
      Clearly, $\Bppsi/t = \widehat R_{\infty,Z}$ and $\Bpsi/t = R_Z$. For simplicity, in what follows, we put
      \[\Bppsi/t^{\infty}:=\Bppsi, \Bpsi/t^{\infty} := \Bpsi \text{ and } \wtR_Z/t^{\infty} := \wtR_Z.\]
  \end{construction}

  \begin{dfn}\label{dfn:local Hitchin-small local systems}
      Fix an $1\leq n\leq \infty$.
      \begin{enumerate}
          \item Let $\bfB\in \{\Bppsi/t^n,\Bpsi/t^n\}$ By an \emph{Hitchin-small representation} of $\Gamma$ over $\bfB$ of rank $r$, we mean a finite projective $\bfB$-module $L$ of rank $r$ endowed with an action of $\Gamma$ such that its reduction $L/t\in\Rep_{\Gamma}^{\Hsmall}(\bfB/t)$. Denote by $\Rep_{\Gamma}^{\Hsmall}(\bfB)$ the category of Hitchin-small representations of $\Gamma$ over $\bfB$.
          
          \item By an \emph{integrable connection} over $\wtR_Z/t^n$ of rank $r$, we mean a pair 
          \[(D,\nabla:D\to D\otimes_{\wtr_Z}\Omega^{1,\log}_{\wtr_Z}\{-1\})\]
          satisfying $\nabla\wedge\nabla = 0$ such that $D$ is a finite projective $\wtR_Z/t^n$-module of rank $r$ and $\nabla$ satisfies the Leibniz rule with respect to the derivation $\rd:\wtr_Z\to\Omega^{1,\log}_{\wtr_Z}\{-1\}$. An integrable connection $(D,\nabla)$ is called \emph{Hitchin-small} if its reduction $(D,\nabla)/t\in \HIG^{\Hsmall}(R_Z)$. Denote by $\MIC^{\Hsmall}(\wtR_Z/t^n)$ the category of Hitchin-small integrable connections over $\wtR_Z/t^n$. For an integrable connection $(D,\nabla)$, denote by
          \[\rD\rR(D,\nabla):=[D\xrightarrow{\nabla}D\otimes_{\wtr_Z}\Omega^{1,\log}_{\wtr_Z}\{-1\}\xrightarrow{\nabla}\cdots\xrightarrow{\nabla} D\otimes_{\wtr_Z}\Omega^{d,\log}_{\wtr_Z}\{-d\}]\]
          its associated de Rham complex and define
          \[\rH^n_{\dR}(D,\nabla):=\rH^n(\rD\rR(D,\nabla)).\]
      \end{enumerate}
  \end{dfn}
  We have the following analogue of Proposition \ref{prop:Being Hitchin-small}.
  \begin{cor}\label{cor:Being Hitchin-small}
      Suppose that $\frakX_Z$ is a liftable (not necessarily small) semi-stable formal scheme over $A^+$ with a fixed lifting $\wtx_Z$ over $\bA_{\inf,K}(Z)$. Denote by $X_Z$ the generic fiber of $\frakX_Z$ and $\wtX_Z$ the lifting of $X_Z$ over $\BBdRp(Z)$ associated to $\wtx_Z$. Fix an $1\leq n\leq \infty$.
      \begin{enumerate}
          \item[(1)] A $\BBdRpn$-local system $\bL$ on $X_{Z,v}$ of rank $r$ is Hitchin-small if and only if there exists an \'etale covering $\{\frakX_{Z,i}\to\frakX_Z\}_{i\in I}$ of $\frakX_Z$ with $\frakX_{Z,i}=\Spf(\calR_{Z,i})$ small semi-stable such that for each $i$, if we denote by $X_{i,\infty,Z} = \Spa(\widehat R_{i,\infty,Z},\widehat R_{i,\infty,Z}^+)\to X_{i,Z}$ the corresponding $\Gamma$-torsor in the sense of Example \ref{exam:Gamma torsor}, then the evaluation $\bL(X_{i,\infty,Z})$ is a Hitchin-small representation of $\Gamma$ over $\BBdRpn(X_{i,\infty,Z})$.

          \item[(2)] An integrable connection $(\calD,\nabla)$ on $X_{Z,n}$ is Hitchin-small if and only if  there exists an \'etale covering $\{\frakX_{Z,i}\to\frakX_Z\}_{i\in I}$ of $\frakX_Z$ such that for each $i$, $\frakX_{Z,i}=\Spf(\calR_{Z,i})$ is small semi-stable with generic fiber $X_{i,Z} = \Spa(R_{i,Z},R_{i,Z}^+)$ and the evaluation $(\calD,\nabla)(X_{i,Z})$ is a Hitchin-small integrable connection over $\wtR_{i,Z}/t^n$, where $\wtR_{i,Z}$ denotes the lifting of $R_{i,Z}$ over $\BBdRp(Z)$ induced from $\wtX_Z$.
      \end{enumerate}
  \end{cor}
  \begin{proof}
      By Remark \ref{rmk:Hitchin small can be checked modulo t}, this follows from Proposition \ref{prop:Being Hitchin-small} directly.
  \end{proof}

\subsection{Local small Simpson correspondence}
  This subsection is essentially same as \cite[\S 4]{SW24} and aims to establish an equivalence of categories
  \[\Rep^{\Hsmall}_{\Gamma}(\widehat R_{\infty,Z})\simeq \HIG^{\Hsmall}(R_Z)\]
  as in \cite[Cor. 4.9]{SW24}. 
  The only difference from \cite{SW24} is that the reduction of our period sheaf $(\calO\bB_{\dR,\pd,Z}^+,\rd)$ modulo $t$ is actually $(\calO\widehat \bC_{\pd},(1-\zeta_p)\Theta)$ in \cite{SW24} (cf. Remark \ref{rmk:compare to MW-Integral}).
  
  Assume $\frakX_Z = \Spf(\calR_Z)$ is small semi-stable as before and keep notations in Example \ref{exam:Gamma torsor}. In particular, we have the following isomorphism
  \[(\calO\bB_{\dR,\pd,Z}^+,\rd)(X_{\infty,Z}) = (\Bppsi\za\underline W\ya_{\pd}:=\Bppsi\za W_1,\dots W_d\ya_{\pd},\rd=\sum_{i=1}^d-u\frac{\partial}{\partial W_i}\otimes\frac{\dlog T_i}{\xi_K})\]
  via the identification
  \[\Bppsi\za W_0,\dots W_d\ya_{\pd}/(W_0+\cdots+W_r)^{\wedge}_{\pd}\cong \Bppsi\za W_1,\dots W_d\ya_{\pd}\]
  and the identification
  \begin{equation}\label{equ:identification omega}
      \Omega^{1,\log}_{\wtr_Z}\cong (\oplus_{i=0}^d\wtr_Z\cdot\dlog T_i)/(\wtr_Z\cdot\sum_{i=0}^r\dlog T_i)\cong \oplus_{i=1}^d\wtr_Z\cdot \dlog T_i
  \end{equation}
  such that via the isomorphism \eqref{equ:gamma group-II}, the $\Gamma$-action on $\calO\bB_{\dR,\pd,Z}^+(X_{\infty,Z}) = \Bppsi\za\underline W\ya_{\pd}$ is given by
  \[\gamma_i(W_j) = W_j+\delta_{ij}u\frac{t}{\xi_K},~\forall~1\leq i,j\leq d.\]
  Define
  \[(B_{\infty},\Theta):=\left((\calO\bB^+_{\dR,\pd,Z},\rd)(X_{\infty})\right)/t\]
  and then we have
  \begin{equation}\label{equ:B_infty}
      (B_{\infty},\Theta) = (\widehat R_{\infty,Z}[\underline W]^{\wedge}_{\pd},-\sum_{i=1}^d(\zeta_p-1)\frac{\partial}{\partial W_i}\otimes\frac{\dlog T_i}{\xi_K})
  \end{equation}
  such that the action of $\Gamma \cong \oplus_{i=1}^d\Zp\gamma_i$ on $B_{\infty}$ is given by
  \[\gamma_i(W_j) = W_j+\delta_{ij}\rho_K(\zeta_p-1),~\forall~1\leq i,j\leq d.\]
  Set 
  \[B^+ := R_Z^+[\underline W]^{\wedge}_{\pd}\subset \widehat R_{\infty,Z}^+[\underline W]^{\wedge}_{\pd}=: B_{\infty}^+.\]
  Then both $B_{\infty}^+$ and $B^+$ are $\Theta$-preserving sub-$R_Z^+$-algebras of $B_{\infty}$ and are stable under the $\Gamma$-action.

  Now let $(H^+,\theta^+)$ be a Hitchin-small Higgs module of rank $r$ over $\calR_Z$ such that via the identification \eqref{equ:identification omega}, the Higgs field $\theta^+$ is of the form
  \[\theta^+ = \sum_{i=1}^d(\zeta_p-1)\theta_i\otimes\frac{\dlog T_i}{\xi_K}\]
  with $\theta_i\in\End_{\calR_Z^+}(H^+)$ topologically nilpotent. Put
  \[\Theta_H^+:=\theta^+\otimes\id+\id\otimes \Theta:H^+\otimes_{\calR_Z}B_{\infty}^+\to H^+\otimes_{\calR_Z}B_{\infty}^+\otimes_{\calR_Z}\Omega^{1,\log}_{\calR_Z}\{-1\}\]
  and then $H^+\otimes_{\calR_Z}B^+$ is a $\Theta_H^+$-preserving sub-$\calR_Z$-module of $H^+\otimes_{\calR_Z}B_{\infty}^+$. Put
  \[L^+(H^+,\theta^+):=(H^+\otimes_{\calR_Z}B^+)^{\Theta_H^+=0} \text{ and }L_{\infty}^+(H^+,\theta^+):=(H^+\otimes_{\calR_Z}B_{\infty}^+)^{\Theta_H^+=0}.\]
  Then $L^+(H^+,\theta^+)$ and $L_{\infty}^+(H^+,\theta^+)$ inherit the $\Gamma$-actions from $B_{\infty}^+$ satisfying 
  \[L_{\infty}^+(H^+,\theta^+) = L^+(H^+,\theta^+)\otimes_{\calR_Z}\widehat R_{\infty,Z}^+.\]
  \begin{lem}\label{lem:L^+}
      Keep notations as above.
      The $\calR_Z$-module $L^+:=L^+(H^+,\theta^+)$ is finite projective of rank $r$ and we have
      \begin{equation}\label{equ:L^+}
          L^+ = \exp(\sum_{i=1}^d\theta_iW_i)(H^+) := \{x\in H^+\otimes_{\calR_Z}B^+\mid x=\sum_{J\in\bN^d}\underline \theta^{J}(y)\underline W^{[J]}, y\in H^+\}
      \end{equation}
      such that the $\Gamma$-action on $L^+$ is given by that for any $x\in L^+$ and any $1\leq i\leq d$,
      \[\gamma_i(x) = \exp(\theta_i\rho_K(\zeta_p-1))(x).\]
      As a consequence, we have
      \[L^+(H^+,\theta^+)\in \Rep^{\Hsmall}_{\Gamma}(\calR_Z) \text{ and }L_{\infty}^+(H^+,\theta^+)\in \Rep^{\Hsmall}_{\Gamma}(\widehat R_{\infty,Z}^+).\]
  \end{lem}
  \begin{proof}
      Let $x = \sum_{J\in\bN^d}x_{J}\underline W^{[J]}$ be an element in $H^+\otimes_{\calR_Z}B^+$ with $x_J\in H^+$ for any $J\in\bN^d$. Then we have 
      \begin{equation*}
          \begin{split}
              \Theta_H^+(x) = &\sum_{i=1}^d\sum_{J\in\bN^d}\big((\zeta_p-1)\theta_i(x_J)\underline W^{[J]}-(\zeta_p-1)x_J\underline W^{[J-E_i]}\big)\otimes\frac{\dlog T_i}{\xi_K}\\
              =& (\zeta_p-1)\sum_{i=1}^d\sum_{J\in\bN^d}\big(\theta_i(x_J)-x_{J+E_i}\big)\underline W^{[J]}\otimes\frac{\dlog T_i}{\xi_K}.
          \end{split}
      \end{equation*}
      Thus, $x\in L^+$ if and only if for any $J\in\bN^d$ and any $1\leq i\leq d$, 
      \[x_{J+E_i} = \theta_i(x_J),\]
      where $E_i = (0,\dots,0,1,0,\dots 0)\in \bN^d$ denotes the generator of the $i$-th component of $\bN^d$.
      
      Using this, by iteration, we see that $x\in L^+$ if and only if 
      \[x = \sum_{J\in\bN^d}\underline \theta^J(x_0)\underline W^{[J]} = \exp(\sum_{i=1}^d\theta_iW_i)(x_0).\]
      So we see that
      \[L^+ = \exp(\sum_{i=1}^d\theta_iW_i)(H^+)\]
      is isomorphic to $H^+$ as an $\calR_Z$-module and thus is finite projective of rank $r$. The description of $\Gamma$-action on $L^+$ follows as $\gamma_i(W_j) = W_j+\delta_{ij}\rho_K(\zeta_p-1)$ for any $1\leq i,j\leq d$.
  \end{proof}
  
  Before we move on, let us recall some lemmas in \cite{SW24}
  \begin{lem}\label{lem:decompletion}
      The base-change functor $L^+\mapsto L_{\infty}^+:=L^+\otimes_{\calR_Z}\widehat R_{\infty}^+$ induces an equivalence of categories
          \[\Rep_{\Gamma}^{\Hsmall}(\calR_Z)\simeq \Rep_{\Gamma}^{\Hsmall}(\widehat R_{\infty,Z}^+)\]
          such that the natural map
          \[\rR\Gamma(\Gamma,L^+)\to\rR\Gamma(\Gamma,L_{\infty}^+)\]
          identifies the former with a direct summand of the latter whose complement is concentrated in degree $\geq 1$ and killed by $\zeta_p-1$.
  \end{lem}
  \begin{proof}
     This is \cite[Prop. 4.2]{SW24}.
  \end{proof} 
  \begin{lem}\label{lem:MW-Integral}
      Fix an $L^+\in \Rep^{\Hsmall}_{\Gamma}(\calR_Z)$ of rank $r$ and put  $L_{\infty}^+ = L^+\otimes_{\calR_Z}\widehat R_{\infty,Z}^+$. Then the following assertions are true.

      \begin{enumerate}
          \item[(1)] The natural inclusion
          \[\rR\Gamma(\Gamma,L^+\otimes_{\calR_Z}B^+)\to\rR\Gamma(\Gamma,L_{\infty}^+\otimes_{\widehat R_{\infty,Z}^+}B_{\infty}^+)\]
          identifies the former with a direct summand of the latter whose complement is concentrated in degree $\geq 1$ and killed by $\zeta_p-1$.

          \item[(2)] Let $\theta_i\in \End_{\calR}(L^+)$ be as in Definition \ref{dfn:Hitchin-small representation}(1). Then the $\calR_Z$-module
          \begin{equation}\label{equ:MW-Integral-I}
          \begin{split}
              \rH^+(L^+):=\rH^0(\Gamma,L^+\otimes_{\calR_Z}B^+)=&\exp(-\sum_{i=1}^d\theta_iW_i)(L^+)\\
              :=&\{x\in L^+\otimes_{\calR_Z}B^+\mid x = \sum_{J\in \bN^d}(-1)^{|J|}\underline \theta^J(y)\underline W^{[J]}, y\in L^+\}.
          \end{split}
          \end{equation}
          is finite projective of rank $r$ such that the natural inclusion $\rH^+(L^+)\hookrightarrow L^+\otimes_{\calR_Z}B^+$ induces an isomorphism
          \begin{equation}\label{equ:MW-Integral-II}
            H^+(L^+)\otimes_{\calR_Z}B^+\xrightarrow{\cong}L^+\otimes_{\calR_Z}B^+.
          \end{equation}
          Moreover for any $n\geq 1$, $\rH^n(\Gamma,L^+\otimes_{\calR_Z}B^+)$ is killed by $\rho_K(\zeta_p-1)$.

          \item[(3)] Via the isomorphism \eqref{equ:identification omega}, the restriction $\theta^+$ of 
          \[\Theta_L^+:=\id_{L^+}\otimes\Theta: L^+\otimes_{\calR_Z}B^+\to L^+\otimes_{\calR_Z}B^+\otimes_{\calR_Z}\Omega^{1,\log}_{\calR_Z}\{-1\}\]
          to $H^+(L^+)$ is given by
          \[\theta^+:=\sum_{i=1}^d(\zeta_p-1)\theta_i\otimes\frac{\dlog T_i}{\xi_K}.\]
          As a consequence, we have $(H^+(L^+),\theta^+)\in \HIG^{\Hsmall}(\calR_Z)$.
      \end{enumerate}
  \end{lem}
  \begin{proof}
     Items (1) and (2) follow from \cite[Cor. 4.8]{SW24}. Item (2) follows \eqref{equ:MW-Integral-I} together with that $\Theta = -\sum_{i=1}^d(\zeta_p-1)\frac{\partial}{\partial W_i}\otimes\frac{\dlog T_i}{\xi_K}$ (cf. \eqref{equ:B_infty}).
  \end{proof}
  Now, we are going to prove the following Simpson correspondence.
  \begin{prop}\label{prop:integral simpson}
      \begin{enumerate}
            \item For any $L_{\infty}^+\in \Rep^{\Hsmall}_{\Gamma}(\widehat R_{\infty,Z}^+)$ of rank $r$, we have
            \begin{equation*}
                \rH^i(\Gamma,L_{\infty}^+\otimes_{\widehat R_{\infty}^+}B^+_{\infty}) = \left\{\begin{array}{rcl}
                    H^+(L_{\infty}^+), & i=0\\
                  \rho_K(\zeta_p-1)\text{-torsion}, &i\geq 1,
                \end{array}\right.
            \end{equation*}
            where $H^+(L_{\infty}^+)$ is a finite projective $\calR_Z$-module of rank $r$. The restriction of 
            \[\Theta_{L_{\infty}}^+=\id_{L_{\infty}^+}\otimes\Theta:L_{\infty}^+\otimes_{\widehat R_{\infty,Z}^+}B^+_{\infty}\to L_{\infty}^+\otimes_{\widehat R_{\infty,Z}^+}B^+_{\infty}\otimes_{\calR_Z}\Omega^{1,\log}_{\calR_Z}\{-1\}\]
            to $\rH^+(L_{\infty}^+)$ defines a Higgs field $\theta^+$ on $H^+(L^+_{\infty})$ such that $(H^+(L^+_{\infty}),\theta^+)\in \HIG^{\Hsmall}(\calR_Z)$.

            (2) For any $(H^+,\theta^+)\in\HIG^{\Hsmall}(\calR_Z)$ of rank $r$, put 
            \[\Theta_H^+:=\theta^+\otimes\id+\id\otimes \Theta:H^+\otimes_{\calR_Z}B_{\infty}^+\to H^+\otimes_{\calR_Z}B_{\infty}^+\otimes_{\calR_Z}\Omega^{1,\log}_{\calR_Z}\{-1\}.\]
            Then $L^+_{\infty}(H^+,\theta^+):=(H^+\otimes_{\calR_Z}B_{\infty}^+)^{\Theta_H^+ = 0}$ equipped with the $\Gamma$-action induced from the $\Gamma$-action on $B_{\infty}^+$ is a well-defined object of rank $r$ in $\Rep_{\Gamma}^{\Hsmall}(\widehat R_{\infty,Z}^+)$.

            (3) The functors $L_{\infty}^+\mapsto (H^+(L_{\infty}^+),\theta^+)$ and $(H^+,\theta^+)\mapsto L^+_{\infty}(H^+,\theta^+)$ are quasi-inverse of each other and thus induce an equivalence of categories
            \[\Rep^{\Hsmall}_{\Gamma}(\widehat R_{\infty,Z}^+)\simeq \HIG^{\Hsmall}(\calR_Z)\]
            which preserves ranks, tensor products and dualities. Moreover, for any $L_{\infty}^+\in \Rep^{\Hsmall}_{\Gamma}(\widehat R_{\infty,Z}^+)$ with corresponding $(H^+,\theta^+)\in\HIG^{\Hsmall}(\calR_Z)$, we have a natural isomorphism that is compatible with Higgs fields
            \begin{equation}\label{equ:compatible with Higgs fields}
                (L_{\infty}^+\otimes_{\widehat R_{\infty,Z}^+}B^+_{\infty},\Theta_{L_{\infty}}^+)\cong (H^+\otimes_{\calR_Z}B_{\infty}^+,\Theta_H^+)
            \end{equation}
            and a quasi-isomorphism 
            \[\rR\Gamma(\Gamma,L_{\infty})\simeq \HIG(H,\theta)\]
            where $L_{\infty} = L_{\infty}^+[\frac{1}{p}]$ and $(H,\theta) = (H^+,\theta^+)[\frac{1}{p}]$.
      \end{enumerate}
  \end{prop}
  \begin{proof}
      Item (1) follows from Lemma \ref{lem:MW-Integral} and Item (2) follows from Lemma \ref{lem:L^+}.

      For Item (3): Granting that we have already obtained the desired equivalence, then (\ref{equ:compatible with Higgs fields}) follows from (\ref{equ:MW-Integral-II}). Note that Item (1) implies $\rH^i(\Gamma,L_{\infty}\otimes_{\widehat R_{\infty,Z}}B_{\infty}) = 0$ for any $i\geq 1$.
      As the Higgs complex $\HIG(B_{\infty},\Theta)$ is a resolution of $\widehat R_{\infty,Z}$ by Theorem \ref{thm:Poincare's Lemma}, we have quasi-isomorphisms
      \[\rR\Gamma(\Gamma,L_{\infty})\simeq \rR\Gamma(\Gamma,\HIG(L_{\infty}\otimes_{\widehat R_{\infty,Z}}B_{\infty},\Theta_{L_{\infty}}))\simeq \HIG(H,\theta_H)\]
      as desired.
      
      So it remains to establish the desired equivalence. 
      
      For any $L_{\infty}^+\in \Rep^{\Hsmall}_{\Gamma}(\widehat R_{\infty,Z}^+)$, put $(H^+,\theta^+) = (H^+(L_{\infty}^+),\theta^+)$. Then we have a canonical morphism 
      \[\iota_{L_{\infty}^+}:L_{\infty}^+(H^+,\theta^+)\to L_{\infty}^+\]
      compatible with $\Gamma$-actions defined by the composite of the following morphisms
      \begin{equation*}
          \begin{split}
              L_{\infty}^+(H^+,\theta^+) = & (H^+\otimes_{\calR_Z}B_{\infty}^+)^{\Theta_H^+ := \theta^+\otimes\id+\id\otimes\Theta =0}\\
              \xrightarrow{=} &((L^+_{\infty}\otimes_{\widehat R_{\infty,Z}^+}B_{\infty}^+)^{\Gamma}\otimes_{\calR_Z}B_{\infty}^+)^{\Theta_L^+\otimes\id+(\id\otimes\id)\otimes\Theta = 0}\\
              \hookrightarrow&(L^+_{\infty}\otimes_{\widehat R_{\infty,Z}^+}B_{\infty}^+\otimes_{\widehat R_{\infty,Z}^+}B_{\infty}^+)^{\id\otimes\Theta\otimes\id+\id\otimes\id\otimes\Theta = 0}\\
              \to & (L_{\infty}^+\otimes_{\widehat R_{\infty,Z}^+}B_{\infty}^+)^{\id\otimes\Theta = 0} \\
              =& L_{\infty}^+.
          \end{split}
      \end{equation*}
      Here, the last but second arrow 
      \[(L^+_{\infty}\otimes_{\widehat R_{\infty,Z}^+}B_{\infty}^+\otimes_{\widehat R_{\infty,Z}^+}B_{\infty}^+)^{\id\otimes\Theta\otimes\id+\id\otimes\id\otimes\Theta = 0}\to (L_{\infty}^+\otimes_{\widehat R_{\infty,Z}^+}B_{\infty}^+)^{\id\otimes\Theta = 0}\]
      is induced by the multiplication on $B_{\infty}^+$.

      For any $(H^+,\theta^+)\in \HIG^{\Hsmall}(\calR)$, put $L_{\infty}^+ = L_{\infty}^+(H^+,\theta^+)$. Then we have a canonical morphism
      \[\iota_{(H^+,\theta^+)}:H^+(L_{\infty}^+)\to H^+\]
      compatible with Higgs fields defined by the composite of the following morphisms
      \begin{equation*}
          \begin{split}
              H^+(L_{\infty}^+) = &(L_{\infty}^+\otimes_{\widehat R_{\infty,Z}^+}B_{\infty}^+)^{\Gamma}\\
              \xrightarrow{=} & ((H^+\otimes_{\calR_Z}B_{\infty}^+)^{\Theta_H^+ = 0}\otimes_{\widehat R_{\infty,Z}^+}B_{\infty}^+)^{\Gamma}\\
              \hookrightarrow &(H^+\otimes_{\calR_Z}B_{\infty}^+\otimes_{\widehat R_{\infty,Z}^+}B_{\infty}^+)^{\Gamma}\\
              \to &(H^+\otimes_{\calR_Z}B_{\infty}^+)^{\Gamma}\\
              = & H^+.
          \end{split}
      \end{equation*}
      Here, the last but second arrow 
      \[(H^+\otimes_{\calR_Z}B_{\infty}^+\otimes_{\widehat R_{\infty,Z}^+}B_{\infty}^+)^{\Gamma}\to (H^+\otimes_{\calR_Z}B_{\infty}^+)^{\Gamma}\]
      is again induced by the multiplication on $B_{\infty}^+$, and the last equality 
      \[(H^+\otimes_{\calR_Z}B_{\infty}^+)^{\Gamma} = H^+\]
      follows as $(B_{\infty}^+)^{\Gamma} = (B^+)^{\Gamma}=\calR_Z$ (cf. Lemma \ref{lem:MW-Integral}). 
      
      To conclude, it is enough to show that $\iota_{L_{\infty}^+}$ and $\iota_{(H^+,\theta^+)}$ are both isomorphisms. But this follows from the explicit description (\ref{equ:L^+}) and (\ref{equ:MW-Integral-I}) immediately.
  \end{proof}
\subsection{Local small Riemann--Hilbert correspondence}
  
  Now, we can give the local Riemann--Hilbert correspondence for Hitchin-small representations of $\Gamma$ over $\Bppsi/t^n$ and Hitchin-small integrable connections over $\wtR_Z/t^n$.

  \begin{prop}\label{prop:local RH over BdRp}
  Fix $1\leq n\leq \infty$. 
      \begin{enumerate}
          \item For any $L_{\infty}\in \Rep^{\Hsmall}_{\Gamma}(\Bppsi/t^n)$ of rank $r$, we have
            \[
              \rH^i(\Gamma,L_{\infty}\otimes_{\Bppsi}\Bppsi\za\underline W\ya_{\pd}) = 
                 \left\{
                    \begin{array}{rcl}
                      D(L_{\infty}), & i = 0 \\
                      0, & i\geq 1,
                    \end{array}
                 \right.
            \]
        where $D(L_{\infty})$ is a finite free $\wtR_Z/t^n$-module of rank $r$. Moreover, the restriction of $\rd$ to $D(L_{\infty})$ induces a flat connection $\nabla$ on $D(L_{\infty})$ such that $(D(L_{\infty}),\nabla)\in\MIC^{\Hsmall}(\wtR_Z/t^n)$.

        \item For any $(D,\nabla)\in \MIC^{\Hsmall}(\wtR_Z/t^n)$ of rank $r$, define
        \[\nabla_{D} = \nabla\otimes\id+\id\otimes\nabla:D\otimes_{\wtR_Z}\Bppsi\za\underline W\ya_{\pd}\to D\otimes_{\wtR_Z}\Bppsi\za\underline W\ya_{\pd}\otimes_{\wtR_Z}\Omega^1_{\wtR_Z}\{-1\}.\]
        Then it satisfies $\nabla_D\wedge\nabla_D = 0$ and we have
        \[
          \rH^i_{\dR}(D\otimes_{\wtR_Z}\Bppsi\za\underline W\ya_{\pd},\nabla_{D}) = 
            \left\{
              \begin{array}{rcl}
                  L_{\infty}(D,\nabla), & i=0 \\
                  0, & i\geq 1,
              \end{array}
            \right.
        \]
        where $L_{\infty}(D,\nabla)$ is a finite free $\Bppsi/t^n$-module of rank $r$. Moreover, the $\Gamma$-action on $\Bppsi\za\underline W\ya_{\pd}$ induces a $\Gamma$-action on $L_{\infty}(D,\nabla)$ such that $L_{\infty}(D,\nabla)\in\Rep_{\Gamma}^{\Hsmall}(\Bppsi/t^n)$.

        \item The functors $L_{\infty}\mapsto (D(L_{\infty}),\nabla)$ and $(D,\nabla)\mapsto L_{\infty}(D,\nabla)$ are quasi-inverses of each other and thus induce an equivalence of categories
        \[\Rep_{\Gamma}^{\Hsmall}(\Bppsi/t^n)\simeq \MIC^{\Hsmall}(\wtR_Z/t^n)\]
        which preserves ranks, tensor products and dualities. Moreover for any $L_{\infty}\in \Rep_{\Gamma}^{\Hsmall}(\Bppsi/t^n)$ with associated $(D,\nabla)\in \MIC^{\Hsmall}(\wtR_Z/t^n)$, we have a quasi-isomorphism
        \[\RGamma(\Gamma,L_{\infty})\simeq \rD\rR(D,\nabla).\]
      \end{enumerate}
  \end{prop}

  Before proving Proposition \ref{prop:local RH over BdRp}, we first recall some well-known lemmas.
  
  \begin{lem}\label{lem:projectivity}
      Let $B$ be a ring with a non-zero divisor $t$ and $B_m:=B/t^m$ for any $m\geq 1$. 
      
      \begin{enumerate}
          \item Let $M$ be a $B_m$-module. If $M/tM$ is a finite projective $B_1$-module of rank $r$ and for any $0\leq n\leq m-1$, the multiplication by $t^n$ induces an isomorphism of $B_1$-modules 
          \[M/tM\xrightarrow{\cong} t^nM/t^{n+1}M,\]
          then $M$ is a finite projective $B_m$-module of rank $r$.

          \item Assum moreover $B$ is $t$-adically complete. Let $M$ be a $B$-module such that $M/tM$ is a finite projective $B_1$-module of rank $r$ and for any $n\geq 0$, the multiplication by $t^n$ induces an isomorphism of $B_1$-modules 
          \[M/tM\xrightarrow{\cong} t^nM/t^{n+1}M,\]
          then $M$ is a finite projective $B$-module of rank $r$.
      \end{enumerate}
  \end{lem}
  \begin{proof}
      Item (2) is a consequence of Item (1) together with \cite[Lem. 1.9]{MT}. So, it is enough to prove Item (1).

      By faithfully flat descent for finite projective modules, we are reduced to the case where $M/tM$ is finite free over $B_1$ of rank $r$.
      Let $e_1,\dots,e_r$ be elements in $M$ whose reductions modulo $t$ give rise to a $B_1$-basis of $M/tM$. We claim that $e_1,\dots,e_r$ form a $B_m$-basis of $M$. 
      We prove this by induction on $m$. We may assume $m>1$ and then by inductive hypothesis, the reduction of $e_1,\dots,e_r$ modulo $t^{m-1}$ form a $B_{m-1}$-basis of $M/t^{m-1}M$. Thus, for any $x\in M$, there exist $b_1,\dots,b_r\in B_m$ such that
      \[x\equiv b_1e_1+\cdots+b_re_r\mod t^{m-1}M.\]
      As $t^{m-1}M \cong M/tM$, one can find $c_1,\dots,c_r\in B_m$ such that 
      \[x-(b_1e_1+\cdots+b_re_r) = t^{m-1}c_1e_1+\cdots+t^{m-1}c_re_r.\]
      Put $a_i = b_i+t^{m-1}c_i$ for any $1\leq i\leq r$ and then we have
      \[x = a_1e_1+\cdots+a_re_r.\]
      So, $e_1,\dots,e_r$ generate $M$ over $B_m$. To conclude, it remains to show that for any $d_1,\dots,d_r\in B_m$ such that
      \[d_1e_1+\cdots+d_re_r = 0,\]
      we must have
      \[d_1=d_2=\cdots=d_r=0.\]
      By inductive hypothesis, we have 
      \[d_1\equiv d_2\equiv\cdots\equiv d_r\mod t^{m-1}.\]
      That is, there are $f_1,\dots,f_r\in B_m$ such that $d_i = t^{m-1}f_i$ for any $1\leq i\leq r$. Thus, we have
      \[f_1t^{m-1}e_1+\cdots+f_rt^{m-1}e_r = t^{m-1}(f_1e_1+\cdots+f_re_r) = 0.\]
      Using $t^{m-1}M\cong M/tM$ again, we conclude that 
      \[f_1\equiv f_2\equiv \cdots\equiv f_r\equiv 0\mod t,\]
      yielding that $d_i = t^{m-1}f_i = 0$ for any $1\leq i\leq r$ as desired. This completes the proof.
  \end{proof}

    \begin{lem}\label{lem:from n to infty}
      \begin{enumerate}
          \item Fix a $B\in \{\Bppsi,\Bpsi\}$. Then the functor $M\mapsto \{M/t^n\}_{n\geq 1}$ induces an equivalence of categories  
          \[\Rep_{\Gamma}^{\Hsmall}(B)\simeq \varprojlim_n\Rep_{\Gamma}^{\Hsmall}(B/t^n)\]
          which is compatible with cohomology in the sense that we have a canonical quasi-isomorphism
          \[\rR\Gamma(\Gamma,M)\xrightarrow{\simeq}\varprojlim_n\rR\Gamma(\Gamma,M/t^n).\]

          \item The functor $(D,\nabla)\mapsto (D/t^n,\nabla)$ induces an equivalence of categories  
          \[\MIC^{\Hsmall}(\wtR_Z)\simeq \varprojlim_n\MIC^{\Hsmall}(\wtR_Z/t^n),\]
          which is compatible with cohomology in the sense that we have a canonical quasi-isomorphism
          \[\rD\rR(D,\nabla) = \varprojlim_n\rD\rR(D/t^n\nabla)\]
      \end{enumerate}
  \end{lem}
  \begin{proof}
      Note that $\Bpsi$, $\Bppsi$ and $\wtR_Z$ are all $t$-adically complete and $t$-torsion free. 
      The desired equivalences follows immediately from Lemma \ref{lem:projectivity}(2) while the cohomological comparison follows from the (derived) Nakayama's Lemma.
  \end{proof}

  \begin{proof}[\textbf{Proof of Proposition \ref{prop:local RH over BdRp}:}]
      By Lemma \ref{lem:from n to infty}, it suffices to deal with the case where $n<\infty$.
      Note that the case for $n=1$ has been proved in Proposition \ref{prop:integral simpson}. We will finish the proof by induction on $n$. From now on, assume that there exists some $n\geq 1$ such that the result holds true for any $1\leq m\leq n$, and it sufficse to show that items (1), (2) and (3) hold true for $n+1$.

      For Item (1): Fix an $M_{\infty}\in \Rep^{\Hsmall}_{\Gamma}(\Bppsi/t^{n+1})$ of rank $r$. Then we have 
      \[tM_{\infty}\in \Rep^{\Hsmall}_{\Gamma}(\Bppsi/t^n)\text{ and }M_{\infty}/t\in \Rep^{\Hsmall}_{\Gamma}(\Bppsi/t).\]
      The short exact sequence
      \[0\to tM_{\infty}\to M_{\infty}\to M_{\infty}/t\to 0\]
      gives rise to an exact triangle
      \[\rR\Gamma(\Gamma,tM_{\infty}\otimes_{\Bppsi}\Bppsi\za\underline W\ya_{\pd})\to\rR\Gamma(\Gamma,M_{\infty}\otimes_{\Bppsi}\Bppsi\za\underline W\ya_{\pd})\to\rR\Gamma(\Gamma,M_{\infty}/t\otimes_{\Bppsi}\Bppsi\za\underline W\ya_{\pd}).\]
      Considering the induced long exact sequence, by inductive hypothesis, we have
      \[\rH^i(\Gamma,M_{\infty}\otimes_{\Bppsi}\Bppsi\za\underline W\ya_{\pd}) = 0\]
      for any $i\geq 1$ and there is a short exact sequence
      \[0\to D(tM_{\infty})\to D(M_{\infty})\to D(M_{\infty}/t)\to 0.\]
      
      We claim that 
      \[D(tM_{\infty}) = tD(M_{\infty})\text{ and }D(M_{\infty}/t) = D(M_{\infty})/t.\]
      Indeed, applying $\rR\Gamma(\Gamma,-\otimes_{\Bppsi}\Bppsi\za\underline W\ya_{\pd})$ to the following the exact sequence
      \[\cdots \to M_{\infty}\xrightarrow{\times t}M_{\infty}\xrightarrow{\times t^n}M_{\infty}\xrightarrow{\times t}M_{\infty}\to M_{\infty}/t\to 0,\]
      we obtain an exact sequence
      \[\cdots \to D(M_{\infty})\xrightarrow{\times t}D(M_{\infty})\xrightarrow{\times t^n}D(M_{\infty})\xrightarrow{\times t}D(M_{\infty})\to D(M_{\infty}/t)\to 0.\]
      So we have the short exact sequence
      \[0\to tD(M_{\infty})\to D(M_{\infty})\to D(M_{\infty}/t)\to 0.\]
      yielding the claim as desired.

      A similar argument implies that for any $0\leq m\leq n$, we have $D(t^mM_{\infty})\cong t^mD(M_{\infty})$,
      yielding that 
      \[t^mD(M_{\infty})/t^{m+1}D(M_{\infty})\cong D(t^mM_{\infty}/t^{m+1}M_{\infty})\cong D(M_{\infty}/t).\]
      By Lemma \ref{lem:projectivity}, we see that $D(M_{\infty})$ is finite projective of rank $r$ over $\wtR_Z/t^{n+1}$ as desired. 
      
      Now, by the $n=1$ case, we have
      \[(D(M_{\infty}/t),\nabla) = (D(M_{\infty})/t,\nabla)\in \MIC^{\Hsmall}(\wtR_Z/t),\]
      yielding that
      
      \[(D(M_{\infty}/t),\nabla)\in \MIC^{\Hsmall}(\wtR_Z/t^{n+1})\]
      as desired. This completes the proof of Item (1).

      Item (2) can be deduced from the similar argument above.

      For Item (3): Similar to the proof of Proposition \ref{prop:integral simpson} (3), for any $M_{\infty}\in \Rep_{\Gamma}^{\Hsmall}(\Bppsi/t^n)$ and any $(D,\nabla)\in \MIC^{\Hsmall}(\wtR_Z/t^{n+1})$, one can construct natural morphisms
      \[\iota_{M_{\infty}}:M_{\infty}(D(M_{\infty}),\nabla)\to M_{\infty}\]
      and 
      \[\iota_{(D,\nabla)}:(D(M_{\infty}(D,\nabla)),\nabla)\to(D,\nabla).\]
      It suffices to prove that they are both isomorphisms.
      To do so, by Nakayama's Lemma, one can check this modulo $t$, and then apply Proposition \ref{prop:integral simpson} (3). This establishes the desired equivalence of categories. The cohomological comparison follows easily.
  \end{proof}

  To complete the local theory, we state the following result, which will not be used in this paper.
  \begin{cor}
      For any $1\leq n\leq \infty$, the base-change functor $L\mapsto L_{\infty}:=L\otimes_{\Bpsi}\Bppsi$ induces an equivalence of categories
      \[\Rep_{\Gamma}^{\Hsmall}(\Bpsi/t^n)\xrightarrow{\simeq}\Rep_{\Gamma}^{\Hsmall}(\Bppsi/t^n)\]
      which is compatible with cohomology in the sense that we have a quasi-isomorphism
      \[\rR\Gamma(\Gamma,L)\cong\rR\Gamma(\Gamma,L_{\infty}).\]
  \end{cor}
  \begin{proof}
      The full faithfulness together with the cohomological comparison follows from the $n=1$ case (cf. Lemma \ref{lem:decompletion}) together with the derived Nakayama's Lemma. It remains to show the essential surjectivity. Note that 
      \[\Bpsi\za\underline W\ya_{\pd}\subset\Bppsi\za\underline W\ya_{\pd}\]
      is a $\rd$-preserving sub-$\Bpsi$-algebra which is stable under the action of $\Gamma$. For any $L_{\infty}\in \Rep_{\Gamma}^{\Hsmall}(\Bppsi/t^n)$ with the induced $(D,\nabla)\in\MIC^{\Hsmall}(\wtR_Z/t^n)$, one can obtain an $L\in \Rep_{\Gamma}^{\Hsmall}(\Bpsi/t^n)$ by using $\Bpsi\za\underline W\ya_{\pd}$ instead of $\Bppsi\za\underline W\ya_{\pd}$ in Proposition \ref{prop:local RH over BdRp}. Then one can conclude by checking that $L\otimes_{\Bpsi}\Bppsi\cong L_{\infty}$ directly.
  \end{proof}

\section{The stacky Riemann--Hilbert correspondence: Proof of Theorem \ref{thm:Stacky RH}}\label{sec:global RH}

  This section is devoted to proving Theorem \ref{thm:Stacky RH}. To do so, we need the following result:
    \begin{thm}\label{thm:global RH}
      Fix a $Z=\Spa(A,A^+)\in \Perfd$. Let $\frakX_Z$ be a liftable semi-stable formal scheme over $A^+$ with a fixed lifting $\wtx_Z$ over $\bA_{\inf,K}(Z)$. Let $X_Z$ be the generic fiber of $\frakX_Z$ and $\wtX_Z$ be the lifting of $X_Z$ over $\BBdRp(Z)$ associated to $\wtx_Z$. 
      Let $(\calO\bB_{\dR,\pd,Z}^+,\rd)$ be the period sheaf with connection corresponding to $\wtx_Z$ as constructed in \S\ref{sec:period sheaf}.
      Let $\nu:X_{Z,v}\to X_{Z,\et}$ be the natural morphism of sites. Let $1\leq n\leq \infty$. 
      \begin{enumerate}
          \item For any $\bL\in \rL\rS(X_Z,\bB_{\dR,n}^+)^{\Hsmall}(Z)$ of rank $r$, we have 
          \[\rR^i\nu_*(\bL\otimes_{\BBdRp}\calO\bB_{\dR,\pd,Z}^+) = \left\{
          \begin{array}{rcl}
              \calD(\bL), & i=0 \\
              0, & i\geq 1
          \end{array}
          \right.\]
          where $\calD(\bL)$ is a locally finite free $\calO_{\wtX_{Z,n}}$-module of rank $r$ on $X_{Z,\et}$ such that 
          \[\id_{\bL}\otimes\rd:\bL\otimes_{\BBdRp}\calO\bB_{\dR,\pd,Z}^+\to \bL\otimes_{\BBdRp}\calO\bB_{\dR,\pd,Z}^+\otimes_{\calO_{\wtx_Z}}\Omega^{1,\log}_{\wtx_Z}\{-1\}\]
          induces a flat connection $\nabla_{\bL}$ on $\calD(\bL)$ making $(\calD(\bL),\nabla_{\bL})\in\MIC(\wtX_{Z,n})^{\Hsmall}(Z)$.

          \item For any $(\calD,\nabla)\in \MIC^{\Hsmall}(\wtX_{Z,n})$ of rank $r$, define 
          \[\nabla_{\calD}:=\nabla\otimes\id+\id\otimes\rd:\calD\otimes_{\calO_{\wtx_Z^I}}\calO\bB_{\dR,\pd,Z}^+\to \calD\otimes_{\calO_{\wtX_Z}}\calO\bB_{\dR,\pd,Z}^+\otimes_{\calO_{\wtx_Z}}\Omega^{1,\log}_{\wtx_Z}\{-1\}.\]
          Then 
          \[\bL(\calD,\nabla):=(\calD\otimes_{\calO_{\wtX_Z}}\calO\bB_{\dR,\pd,Z}^+)^{\nabla_{\calD} = 0}\]
          is an object of rank $r$ in $\rL\rS(X_Z,\bB_{\dR,n}^+)^{\Hsmall}(Z)$.

          \item The functors $\bL\mapsto (\calD(\bL),\nabla_{\bL})$ and $(\calD,\nabla)\mapsto \bL(\calD,\nabla)$ described above are quasi-inverses of each other and thus define an equivalence of categories
          \[\rho_{\wtx_Z}:\rL\rS(X,\bB_{\dR,n}^+)^{\Hsmall}(Z)\xrightarrow{\simeq} \MIC^{\Hsmall}(\wtX_n)(Z)\]
          which preserves ranks, tensor products and dualities. Moreover, for any $\bL\in \rL\rS(X,\bB_{\dR,n}^+)^{\Hsmall}(Z)$ with corresponding $(\calD,\nabla)\in\MIC(\wtX_{Z,n})^{\Hsmall}(Z)$, this equivalence induces a quasi-isomorphism
          \[\rR\nu_*\bL\simeq \rD\rR(\calD,\nabla),\]
          and thus a quasi-isomorphism
          \[\rR\Gamma(X_{Z,v},\bL)\simeq \rR\Gamma(X_{Z,\et},\rD\rR(\calD,\nabla)).\]
      \end{enumerate}
  \end{thm}
  Before proving this theorem, we explain how to use it to obtain Theorem \ref{thm:Stacky RH}.
  \begin{proof}[\textbf{Proof of Theorem \ref{thm:Stacky RH}:}]
    Let $\frakX$ be a liftable semi-stable formal scheme over $\calO_C$ with a fixed flat lifting $\wtx$ over $\bfA_{\inf,K}$. Denote by $X$ its generic fiber and let $\wtX$ be the lifting of $X$ over $\BdRp$ induced by $\wtx$. For any $Z=\Spa(A,A^+)\in \Perfd$, we respectively denote by $\frakX_Z$, $\wtx_Z$, $X_Z$ and $\wtX_Z$ the corresponding base-changes of $\frakX$, $\wtx$, $X$ and $\wtX$ along $\calO_C\to A^+$, $\bfA_{\inf,K}\to \bA_{\inf,K}(Z)$, $C\to A$ and $\BdRp\to \BBdRp(Z)$. Let $(\calO\bB_{\dR,\pd,Z}^+,\rd)$ be the period sheaf with connection on $X_{Z,v}$ corresponding to the lifting $\wtx_Z$. Then it follows from Remark \ref{rmk:functoriality of period rings} that for any $f:Z_2=\Spa(A_2,A_2^+)\to Z_1=\Spa(A_1,A_1^+)$, $(\calO\bB_{\dR,\pd,Z_2}^+,\rd)$ is the restriction of $(\calO\bB_{\dR,\pd,Z_1}^+,\rd)$ to $X_{Z_2,v}$. Note that $f$ induces the obvious base-change functors
    \[f_{\rL\rS}:\rL\rS(X,\BBdRpn)^{\Hsmall}(Z_1)\to\rL\rS(X,\BBdRpn)^{\Hsmall}(Z_2),\quad \bL\mapsto \bL_{\mid_{X_{Z_2,v}}}\]
    and
    \[f_{\MIC}:\MIC(\wtX_n)^{\Hsmall}(Z_1)\to\MIC^{\Hsmall}(
    \wtX_n)(Z_2), \quad(\calD,\nabla)\mapsto (\calD\widehat \otimes_{\BBdRp(Z_1)}\BBdRp(Z_2),\nabla\otimes\id).\]
    As the construction in Theorem \ref{thm:global RH} is clearly functorial in $Z=\Spa(A,A^+)$, we have
    \[f_{\MIC}\circ\rho_{\wtx_{Z_1}} = \rho_{\wtx_{Z_2}}\circ f_{\rL\rS}.\]
    As $\rho_{\wtx_Z}$ preserves ranks, the equivalence criterion of $\rho_{\wtx_Z}$
    yields the desired equivalence of stacks
    \[\rho_{\wtx}:\rL\rS_r(X,\BBdRpn)^{\Hsmall}\xrightarrow{\simeq}\MIC_r(\wtX_n)^{\Hsmall}.\]
    This completes the proof.
  \end{proof}
  Now, we focus on the proof of Theorem \ref{thm:global RH}.
  \begin{lem}\label{lem:vanishing}
      Suppose that $\frakX_Z$ is small affine with generic fiber $X_Z$ and let $\bL$ be a $\BBdRpn$-local system on $X_{Z,v}$. Then for any affinoid perfectoid $U = \Spa(S,S^+)\in X_{Z,v}$ and for any $i\geq 1$, we have
      \[\rH^i(U,\bL\otimes_{\BBdRp}\calO\bB_{\dR,\pd,Z}^+) = 0.\]
  \end{lem}
  \begin{proof}
      We need to show that $\rR\Gamma(U,\bL\otimes_{\BBdRp}\calO\bB_{\dR,\pd,Z}^+)$ is concentrated in degree $0$. By derived Nakayama's Lemma, we are reduced to showing the case for $n=1$. In other words, we have to show for any $v$-vector bundle $\bL$,
      \[\rR\Gamma(U_v,\bL\otimes_{\widehat \calO_{X_Z}}\calO\widehat \bC_{\pd,Z})\]
      is concentrated in degree $0$, where $\calO\widehat \bC_{\pd,Z}:=\calO\bB_{\dR,\pd,Z}^+/t$ denotes the reduction of $\calO\bB_{\dR,\pd,Z}^+$ modulo $t$. It follows from \cite[Th. 3.5.8]{KL16} that there is a finite projective $S$-module $L$ such that $\bL_{\mid_U}\cong L\otimes_{S}\widehat \calO_U$. Therefore, without loss of generality, we may assume $\bL_{\mid_U} = \widehat \calO_U$, and are reduced to showing that 
      \[\rR\Gamma(U_v,\calO\widehat \bC_{\pd,Z}) = \rR\Gamma(U_v,\widehat \calO_U[\underline W]_{\pd}^{\wedge})\]
      is concentrated in degree $0$, by using Proposition \ref{prop:local description of OB_dR^+}. It suffices to show that
      \[\rR\Gamma(U_v,\widehat \calO_U^+[\underline W]^{\wedge}_{\pd})\]
      has cohomologies killed by $\frakm_C$ in degree $\geq 1$. However, as $\widehat \calO_U^+[\underline W]^{\wedge}_{\pd}$ is the $p$-adic completion of the free $\widehat \calO_U^+$-module $\widehat \calO_U^+[\underline W]_{\pd}$, using the proof of \cite[Lem. 5.7]{Wan23}, we are reduced to showing that 
      \[\rR\Gamma(U_v,\widehat \calO_U^+)\]
      has cohomologies killed by $\frakm_C$ in degree $\geq 1$. This is well-known (cf. \cite[Prop. 8.8]{Sch-Diamond}).
  \end{proof}
  \begin{lem}\label{lem:reduce to Gamma-cohomology}
      Suppose that $\frakX_Z$ is small affine with generic fiber $X_Z$ and let $\bL$ be a $\BBdRpn$-local system on $X_{Z,v}$. Then there exists a natural quasi-isomorphism
      \[\rR\Gamma(\Gamma,\bL\otimes_{\BBdRp}\calO\bB_{\dR,\pd,Z}^+(X_{\infty,Z}))\to\rR\Gamma(X_{Z,v},\bL\otimes_{\BBdRp}\calO\bB_{\dR,\pd,Z}^+),\]
      where $X_{\infty,Z}\to X_Z$ is the $\Gamma$-torsor considered in Example \ref{exam:Gamma torsor}.
  \end{lem}
  \begin{proof}
      This follows from the same argument for the proof of \cite[Lem. 5.11]{Wan23} by using Lemma \ref{lem:vanishing} instead of \cite[Lem. 5.7]{Wan23}.
  \end{proof}

  Now, we are prepared to show Theorem \ref{thm:global RH}.
  \begin{proof}[\textbf{Proof of Theorem \ref{thm:global RH}:}]

      For Item (1): Fix an $\bL\in \rL\rS(X_Z,\bB_{\dR,n}^+)^{\Hsmall}(Z)$ of rank $r$. Since the problem is local on $\frakX_{Z,\et}$, we may assume $\frakX_Z = \Spf(\calR_Z)$ is small semi-stable. By Corollary \ref{cor:Being Hitchin-small}(1), one can deduce Item (1) from Lemma \ref{lem:reduce to Gamma-cohomology} together with Proposition \ref{prop:local RH over BdRp}(1).

      For Item (2): Fix a $(\calD,\nabla)\in\MIC^{\Hsmall}(\wtX_{Z,n})(Z)$ of rank $r$. Since the problem is again local on $\frakX_{Z,\et}$, we may assume that $\frakX_Z = \Spf(\calR_Z)$ is small semi-stable with the generic fiber $X_Z = \Spa(R_Z,R_Z^+)$. Let $\wtR_Z$ be the lifting of $R_Z$ over $\BBdRp(Z)$ induced by the given lifting $\wtx_Z$. According to Corollary \ref{cor:Being Hitchin-small}(2), we see that the global section 
      \[(D,\nabla):=(\calD,\nabla)(X_Z)\in \MIC^{\Hsmall}(\wtR_Z/t^n).\]
      Then one can deduce Item (2) from Proposition \ref{prop:local RH over BdRp}(2).

      For Item (3): For any $\bL\in \rL\rS(X_Z,\bB_{\dR,n}^+)^{\Hsmall}(Z)$ and any $(\calD,\nabla)\in\MIC(\wtX_{Z,n})^{\Hsmall}(Z)$, similar to the proof of Proposition \ref{prop:local RH over BdRp}(3), one can construct two canonical morphisms
      \[\iota_{\bL}:\bL(\calD(\bL),\nabla_{\bL})\to\bL\]
      and
      \[\iota_{(\calD,\nabla)}:(\calD(\bL(\calD,\nabla)),\nabla_{\bL(\calD,\nabla)})\to(\calD,\nabla).\]
      To get the desired equivalence of categories, it is enough to show that both $\iota_{\bL}$ and $\iota_{(\calD,\nabla)}$ are isomorphisms. But this is still a local problem on $\frakX_{Z,\et}$ and thus we can conclude by Proposition \ref{prop:local RH over BdRp} (3). Finally, we have to show that for any $\bL\in \rL\rS(X_Z,\bB_{\dR,n}^+)^{\Hsmall}(Z)$ with corresponding $(\calD,\nabla)\in\MIC(\wtX_{Z,n})^{\Hsmall}(Z)$, there exists a quasi-isomorphism
      \[\rR\nu_*\bL\simeq \rD\rR(\calD,\nabla).\] 
      By Poincar\'e's Lemma (cf. Theorem \ref{thm:Poincare's Lemma}), we have
      \[\rR\nu_*\bL\simeq \rR\nu_*(\rD\rR(\bL\otimes_{\BBdRp}\calO\bB_{\dR,\pd,Z}^+,\id_{\bL}\otimes\rd)).\]
      On the other hand, it follows from Item (1) that
      \[\rR\nu_*(\rD\rR(\bL\otimes_{\BBdRp}\calO\bB_{\dR,\pd,Z}^+,\id_{\bL}\otimes\rd))\simeq \nu_*(\rD\rR(\bL\otimes_{\BBdRp}\calO\bB_{\dR,\pd,Z}^+,\id_{\bL}\otimes\rd))\simeq \rD\rR(\calD,\nabla_{\calD}).\]
      So we get a quasi-isomorphism
      \[\rR\nu_*\bL\simeq\rD\rR(\calD,\nabla_{\calD})\]
      as expected. This completes the proof.
  \end{proof}

  As an application, we give another proof of arithmetic Riemann--Hilbert correspondence in \cite{GMWrelative} in the good reduction case.

  \begin{cor}[Arithmetic Riemann--Hilbert correspondence]\label{cor:arithmetic RH}
      Let $\frakX_0$ be a smooth formal scheme over $\calO_K$ with generic fiber $X_0$. For any $1\leq n\leq \infty$, let $\wtX_n$ be the base-change of $X_0$ along $K\to \bfB_{\dR}^+/t^n$. Then there exists an equivalence of categories
      \[\rL\rS(X_{0,v},\BBdRpn)\simeq\MIC^{G_K}(\wtX_n)\]
      where $\rL\rS(X_{0,v},\BBdRpn)$ is the category of $\BBdRpn$-local systems on $X_{0,v}$ and $\MIC^{G_K}(\wtX_n)$ is the category of $G_K$-equivariant integrable connections on $\wtX_n$.
  \end{cor}
  \begin{proof}
      Note that $X = \wtX_1$ is base-change of $X_0$ along $K\to C$, which admits an obvious smooth formal model $\frakX$ over $\calO_C$; namely, the base-change of $\frakX_0$ along $\calO_K\to\calO_C$. Let $\wtx$ be the base-change of $\frakX_0$ along $\calO_K\to\bfA_{\inf,K}$. Then $\wtx$ is the lifting of $\frakX$ over $\bfA_{\inf,K}$. As $X\to X_0$ is a pro-\'etale Galois covering with Galois group $G_K$, by $v$-descent, the restriction from $X_{0,v}$ to $X_v$ induces an equivalence of categories
      \[\rL\rS(X_{0,v},\BBdRpn)\simeq \rL\rS^{G_K}(X_v,\BBdRpn),\]
      where $\rL\rS^{G_K}(X_v,\BBdRpn)$ denotes the category of $G_K$-equivariant $\BBdRpn$-local systems on $X_v$. As $G_K$-equivariant $\BBdRpn$-local systems on $X_v$ and integrable connections on $\wtX_n$ are always Hitchin-small (cf. \cite[Rem. 3.2]{MW-JEMS}), one can conclude by Theorem \ref{thm:global RH} immediately.
  \end{proof}
    We give a remark to explain how to get rid of the smoothness condition on the formal model in the above corollary.
  \begin{rmk}\label{rmk:arithmetic RH}
      One can obtain the arithmetic Riemann--Hilbert correspondence in Corollary \ref{cor:arithmetic RH} for $X_0$ which has a semi-stable formal model $\frakX_0$ over $\calO_K$. Put 
      \[\bfA_{\dR,K}:=\text{$\xi_K\pi^{-1}$-adic completion of }\bfA_{\inf,K}[(\frac{[\underline \pi]}{\pi})^{\pm 1}]^{\wedge_p} \]
      and replace $(\bfA_{\inf,K},\xi_K)$ by $(\bfA_{\dR,K},\xi_K\pi^{-1})$. Then for a semi-stable formal scheme $\frakX$ over $\calO_C$ with a fixed flat lifting $\wtx$ over $\bfA_{\dR,K}$, one can similarly construct a period sheaf (of $\BBdRp$-algebras) with connection 
      \[(\calP,\rd:\calP\to \calP\otimes_{\calO_{\wtx}}\Omega^{1,\log}_{\wtx}\langle-1\rangle)\]
      where for any $n\in \bZ$ and for any $\bfA_{\dR,K}$-module $M$, 
      \[M\langle n \rangle:=M\otimes_{\bfA_{\dR,K}}\Ker(\bfA_{\dR,K}\to\calO_C)^{\otimes n} = M\cdot(\xi_K\pi^{-1})^n,\]
      such that the induced de Rham complex
      \[0\to\BBdRp\to \rD\rR(\calP,\rd)\]
      is exact. Define the following locus on Hitchin-base
      \[\calA_r^{\circ}:=\oplus_{i\geq 1}^rp^{>\frac{i}{p-1}}\cdot\rH^0(\frakX,\Sym^i(\Omega_{\frakX}^{1,\log}\langle-1\rangle))\subset \calA_r.\]
      Similar to Theorem \ref{thm:Stacky RH}, one can establish an equivalence of stacks
      \[\rL\rS_r(X,\BBdRpn)\times_{\calA_r}\calA_r^{\circ} \simeq \MIC_r(\wtX_n)\times_{\calA_r}\calA_r^{\circ}\]
      by using $(\calP,\rd)$ instead of $(\calO\bB_{\dR,\pd}^+,\rd)$ above.
      
      As the base-change $\wtx$ of $\frakX_0$ along $\calO_K\to \bfA_{\dR,K}$ is a lifting of $\frakX = \frakX_0\times_{\Spf(\calO_K)}\Spf(\calO_C)$ over $\bfA_{\dR,K}$, using the same proof of Corollary \ref{cor:arithmetic RH}, we get an equivalence of categories
      \[\rL\rS(X_{0,v},\BBdRpn)\simeq\MIC^{G_K}(\wtX_n)\]
      in this case, where $\wtX_n$ is the base-change of $X_0$ along $K\to\bfB_{\dR}^+/t^n$. We leave the details to the interested readers.
  \end{rmk}
   \begin{rmk}\label{rmk:equivalence on neighborhood}
      Let $X$ be an abelian variety over $C$. By spreading-out and the stable reduction theorem of Grothendieck \cite[Expos\'e IX, Th. 3.6]{SGA7I}, there exists a complete discrete valuation sub-field $K$ of $C$ and a semi-stable formal scheme $\frakX_0$ over $\calO_K$ whose generic fiber $X_0$ is an abelian variety over $K$ such that $X=X_0\times_KC$. In particular, we see that $X$ has a semi-stable formal model $\frakX=\frakX_0\times_{\Spf(\calO_K)}\Spf(\calO_C)$ which admits a lifting $\wtx$ over $\bfA_{\dR,K}$ as in Remark \ref{rmk:arithmetic RH}. So one can show that there exists an open neighborhood $\calA_r^{\circ}\subset\calA_r$ of the origin $0\in\calA_r$ (depending on $K$ and thus on $X$) such that there is an equivalence of stacks
      \[\rL\rS_r(X,\bB_{\dR,n}^+)\times_{\calA_r}\calA_r^{\circ}\simeq \MIC_r(\wtX_n)\times_{\calA_r}\calA_r^{\circ}.\]
  \end{rmk}






\end{document}